\documentclass{amsart}

\usepackage{amsmath,amsthm,amssymb}

\usepackage{graphicx}

\usepackage{hyperref}

\newcommand{\bim}{\mathfrak{M}}
\newcommand{\bimL}{\mathfrak{L}}
\newcommand{\bimR}{\mathfrak{R}}

\newcommand{\nuke}{\mathcal{N}}
\newcommand{\img}[1]{\mathop{\mathrm{IMG}}\left(#1\right)}
\newcommand{\limg}[1][G]{\mathcal{X}_{#1}}
\newcommand{\arr}{\rightarrow}

\newcommand{\C}{\mathbb{C}}
\newcommand{\CP}{\mathbb{P}\C^2}

\newcommand{\Z}{\mathbb{Z}}
\newcommand{\R}{\mathbb{R}}
\newcommand{\wt}[1]{\widetilde{#1}}
\newcommand{\M}{\mathcal{M}}
\newcommand{\lims}[1][G]{\mathcal{J}_{#1}}

\newcommand{\bex}{\mathtt{x}}
\newcommand{\bel}{\mathtt{L}}
\newcommand{\ber}{\mathtt{R}}
\newcommand{\bee}{\mathtt{e}}
\newcommand{\si}{\mathsf{s}}
\newcommand{\group}{\mathcal}

\title{Mating, paper folding, and an endomorphism of $\CP$}
\author{Volodymyr Nekrashevych}

\newtheorem{theorem}{Theorem}[section]
\newtheorem{proposition}[theorem]{Proposition}
\newtheorem{corollary}[theorem]{Corollary}
\newtheorem{lemma}[theorem]{Lemma}

\theoremstyle{definition}
\newtheorem{defi}{Definition}[section]

\begin{document}
\maketitle 
\begin{abstract}
We are studying topological properties of the Julia set of the map $F(z, p)=\left(\left(\frac{2z}{p+1}-1\right)^2,
\left(\frac{p-1}{p+1}\right)^2\right)$ of the complex projective plane
$\CP$ to itself. We show a relation of this rational function with an
uncountable family of ``paper folding'' plane filling curves..
\end{abstract}

\tableofcontents

\section{Introduction}

We study topology of the Julia set of the map
\[F(z, p)=\left(\left(\frac{2z}{p+1}-1\right)^2,
\left(\frac{p-1}{p+1}\right)^2\right)\] of the projective plane $\CP$
to itself. In particular, we show an interesting 
connection between this map and an uncountable class of plane-filling curves coming from folding a
strip of paper. One of the goals of the paper is to show
how techniques of iterated monodromy groups can be used to
obtain interesting topological  properties of dynamical
systems. All our results are proved using algebraic computations with
self-similar groups. 

Let us start from the end of the article, and describe the plane-filling curves.
Take a long narrow strip of paper and fold it in two. Then repeat the procedure
several times. Note that each time you have a choice of two directions of folding.
Then unfold it so that you get right angles at the creases. You
will get something like it is shown on Figure~\ref{fig:realpaper}.

\begin{figure}[h]
\centering
\includegraphics{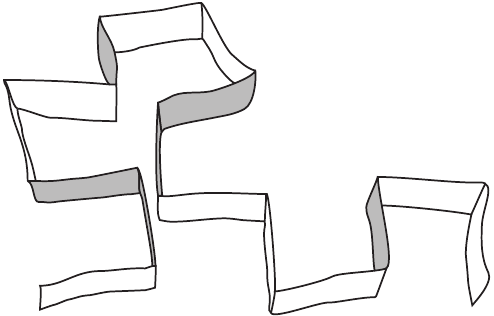}
\caption{Folded paper}
\label{fig:realpaper}
\end{figure}

Let us record the way we folded the paper as a sequences of letters
$\bel$, $\ber$, standing for ``left'' and ``right'', respectively. 

Fold now two equal strips of paper in the same way (described by the
same sequence of letters $\bel$, $\ber$), and rotate one
with respect to the other by $180^\circ$. Put them down so that their
endpoints touch, see Figure~\ref{fig:realpaper2}.

\begin{figure}[h]
\centering
\includegraphics{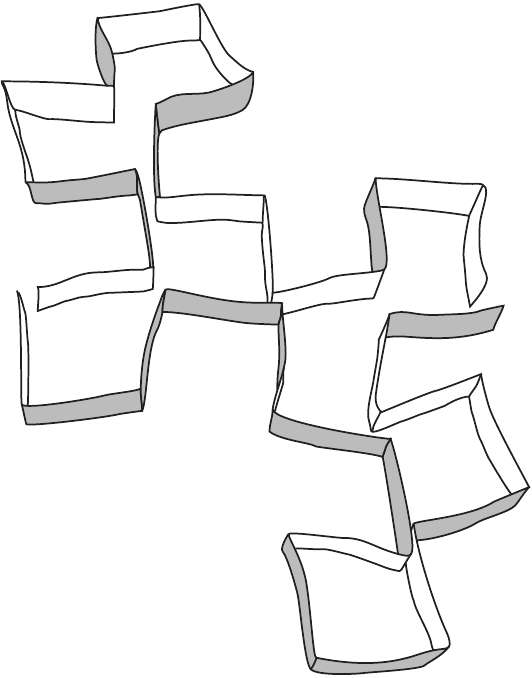}
\caption{Two strips}
\label{fig:realpaper2}
\end{figure}

You will get a closed curve $\gamma$, bounding a connected maze of square
rooms. See two examples of such mazes on Figure~\ref{fig:maze}. The rooms are
shaded black. Two red dots mark the endpoints of the strips of paper,
the green dots mark their midpoints (i.e., the creases of the first folding).

\begin{figure}[h]
\centering
\includegraphics{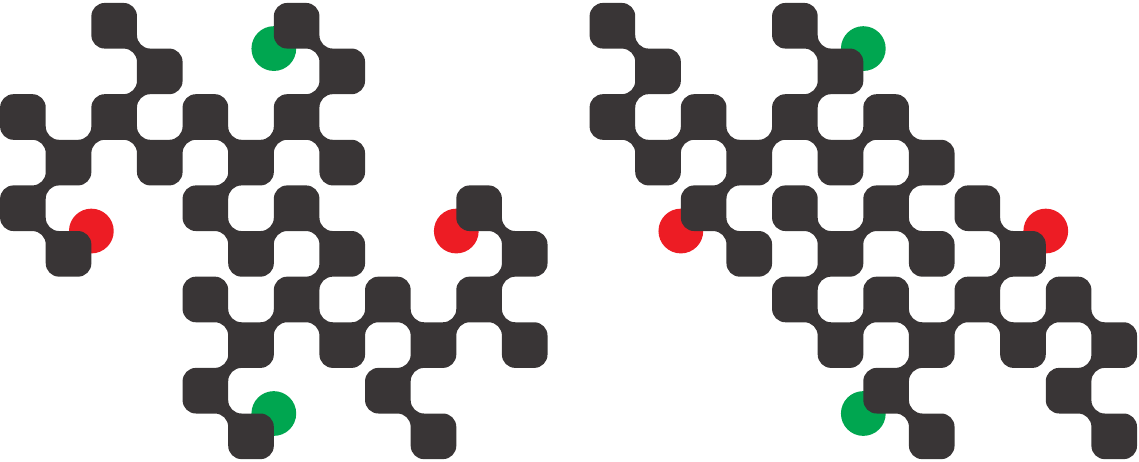}
\caption{Mazes}
\label{fig:maze}
\end{figure}

The marked points are vertices of a square (this easily
follows from the construction). Let us choose an infinite sequence
$w=X_1X_2\ldots$ of letters $\bel$ and $\ber$. Let us draw rescaled closed
curves $\gamma_{w_n}$
corresponding to finite sequences $w_n=X_1X_2\ldots X_n$ of instructions
in such a way that the marked points stay at the vertices of a fixed square $Q$.
 Let us parametrize the 
curves $\gamma_{w_n}$ uniformly (proportionally to the arclength) by
$t\in [0, 1]$, so that $\gamma_{w_n}|_{[0, 1/2]}$ and
$\gamma_{w_n}|_{[1/2, 1]}$ are copies of the folded paper strip. Then
the vertices of $Q$ are $\gamma_{w_n}(0)$,
$\gamma_{w_n}(1/2)$ (the endpoints of the strips), and
$\gamma_{w_n}(1/4)$, $\gamma_{w_n}(3/4)$ (their midpoints).

It follows from the description of the 
folding procedure  that the maze $\gamma_{w_{n+1}}$ is
obtained from the maze $\gamma_{w_n}$  by replacing each wall by a
corner (so that the old wall is the hypotenuse and the new walls are
legs of an isosceles right triangle).
We get then a sequence of curves converging uniformly to some limit curve
$\gamma_{X_1X_2\ldots}$.

The best known and studied example is the \emph{Heighway dragon curve}, which
corresponds to a constant sequence $w=\bel\bel\bel\cdots$ (or
$\ber\ber\ber\cdots$). It
was defined for the first time by NASA physicists J.~Heighway,
B.~Banks, and W.~Harter, and popularized by M.~Gardner in Scientific
American. It
is also called sometimes the ``Jurassic Park Fractal'', as the curves
$\gamma_{\bel^n}|_{[0, 1/2]}$ appear at the beginning of each chapter of
``Jurassic Park'' by M.~Crichton.
The closed version $\gamma_w$ is called
sometimes the \emph{twin-dragon
  curve}. See the images of
these curves on Figure~\ref{fig:dragon}. In~\cite[p.~190]{knuth} a
relation of the twin dragon curve to numeration systems on complex numbers
is discussed. 

\begin{figure}[h]
\centering
\includegraphics{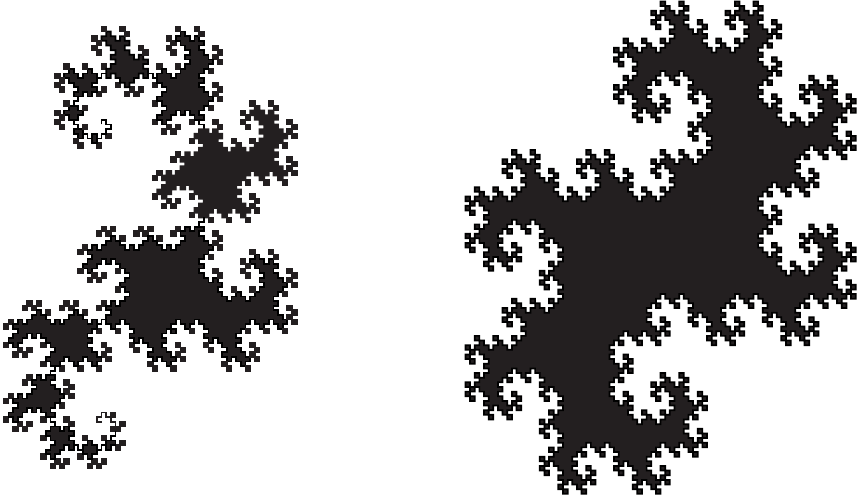}
\caption{Dragon and Twin Dragon curves}
\label{fig:dragon}
\end{figure}

Consider the group $H$ of transformations of the plane generated by
rotations by $180^\circ$ around the vertices of the square $Q$ (recall
that it is the square whose vertices are the endpoints and the
midpoints of the strips of paper). The fundamental domain of the group
$H$ is the square $Q'$ of twice bigger area such that vertices of $Q$
are midpoints of the sides of $Q'$. The quotient $\R^2/H$ of the plane by the
action of $H$ is homeomorphic to the sphere, and can be realized as
the pillowcase obtained from the square $Q'$ by folding its corners over
the sides or the square $Q$.

If we take the curve $\gamma_{w_n}$, then its image $\gamma_{w_n}/H$ on the
pillowcase $\R^2/H$ is a nice Eulerian path tracing a square grid on
the pillow, see Figure~\ref{fig:mazer}. The figure shows how pieces of the
curve $\gamma_{w_n}$ that are outside the fundamental domain $Q'$ are
moved inside by elements of $H$ (actually, just by the generators). 

\begin{figure}[h]
\centering
\includegraphics{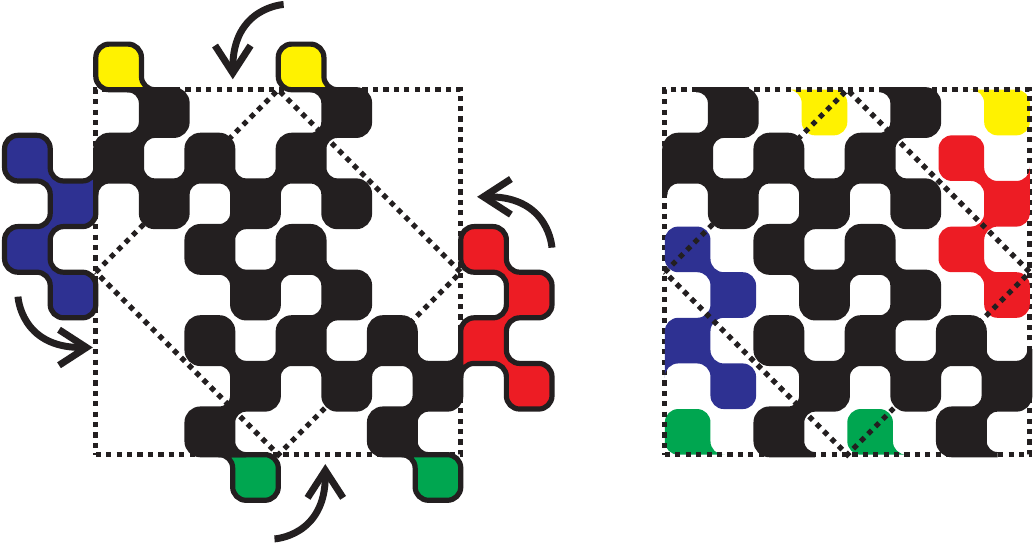}
\caption{Curve on the pillowcase $\R^2/H$}
\label{fig:mazer}
\end{figure}

For every infinite sequence $w=X_1X_2\ldots$ the image
$\gamma_w/H$ of the curve $\gamma_w$ is
a curve passing through every point of the sphere $\R^2/H$. 

It follows directly from the construction, that the curves
$\gamma_{X_1X_2\ldots X_n}|_I$ for $I=[0, 1/4], [1/4, 1/2], [1/2,
3/4], [3/4, 1]$ are similar (with the similarity coefficient
$1/\sqrt{2}$) to the curves $\gamma_{X_2X_3\ldots X_n}|_{[0, 1/2]}$
and $\gamma_{X_2X_3\ldots X_n}|_{[1/2, 1]}$. The partition of
$\gamma_{X_1X_2\ldots X_n}$ into the defined above sub-curves
correspond to splitting the original strip of paper in two (and
``forgetting'' about the first folding). Moreover, the similarities
$\gamma_{X_1X_2\ldots X_n}|_{[0, 1/4]}\arr\gamma_{X_2X_3\ldots
  X_n}|_{[0, 1/2]}$, $\gamma_{X_1X_2\ldots X_n}|_{[1/4, 1/2]}\arr\gamma_{X_2X_3\ldots
  X_n}|_{[1/2, 1]}$, $\gamma_{X_1X_2\ldots X_n}|_{[1/2, 3/4]}\arr\gamma_{X_2X_3\ldots
  X_n}|_{[0, 1/2]}$, $\gamma_{X_1X_2\ldots X_n}|_{[3/4, 1]}\arr\gamma_{X_2X_3\ldots
  X_n}|_{[1/2, 1]}$ are restrictions of a branched self-covering
$S_{X_1}:\R^2/H\arr\R^2/H$. See
Figure~\ref{fig:F} for a description of $S_{X_1}$.
In the limit $S_{X_1}$ induces a piece-wise similarity map from $\gamma_{X_1X_2\ldots}$ to
$\gamma_{X_2X_3\ldots}$.

\begin{figure}[h]
\centering
\includegraphics{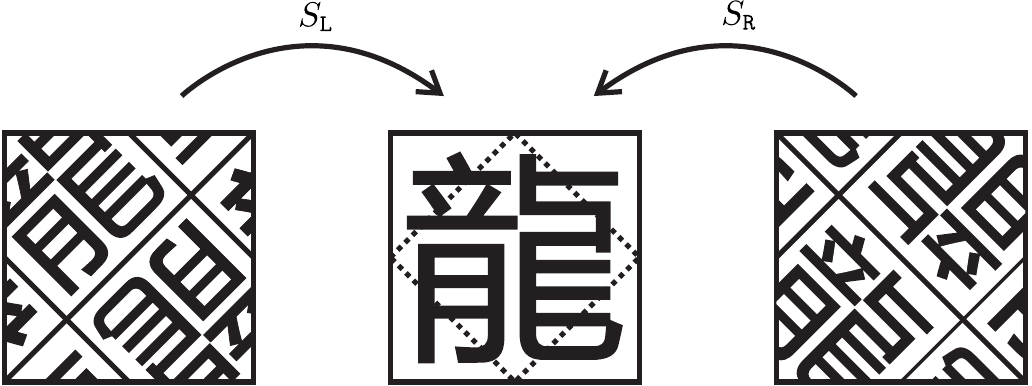}
\caption{Pillowcase folding}
\label{fig:F}
\end{figure}

Each of the curves $\gamma_{X_1X_2\ldots X_n}$ divides the sphere into
two parts: it separates squares of different color in a checkerboard
coloring of the pillow $\R^2/\Z^2$. The squares of one color
are connected to each other by the corridors of the maze in a
tree-like fashion. It is
reasonable to assume that in the limit the curve
$\gamma_{X_1X_2\ldots}$ goes around a dendrite that is a limit of the
sequence of the trees bounded by $\gamma_{X_1X_2\ldots X_n}$. See, for
example~\cite[II.7]{mandelbrot}, where relation between plane-filling
curves and ``rivers'' they bound is explored.

The map $S_{X_1}$ induces a degree two branched covering of the dendrite
bounded by $\gamma_{X_2X_3\ldots}$ by the dendrite bounded by
$\gamma_{X_1X_2\ldots}$.

The case of the twin-dragon curve was analyzed in
detail by
J.~Milnor in~\cite{milnor:dragons}. He showed that the dendrites into which
the curve $\gamma_{\ber\ber\ber\ldots}/H$ separates the sphere can be
naturally identified with the Julia set of the polynomial
$f(z)=z^2+c$, where $c\approx-0.228+1.115i$ is a root of the polynomial $c^3+2c^2+2c+2$.
Moreover, the curve $\gamma_{\ber\ber\ber\ldots}/H$ goes around
the dendrite in the same way as the \emph{Caratheodory loop} goes
around the Julia set, so that the sphere $\R^2/H$ is the \emph{mating}
of the polynomial $z^2+c$ with itself: it is obtained by gluing two Julia
sets along the Caratheodory loops so that one loop is a complex
conjugate of the other.

The general paper folding curves $\gamma_w$ were studied in~\cite{davisknuth:dragons} where
a connection with numeration systems on complex numbers was described
(and closely related to the ``pillowcase folding'' maps
$S_\bel, S_\ber:\R^2/H\arr\R^2/H$).

In our paper we show a relation of the sphere-filling curves
$\gamma_w$ with dynamics of the map
\[F(z, p)=\left(\left(\frac{2z}{p+1}-1\right)^2,
\left(\frac{p-1}{p+1}\right)^2\right)\]
of the projective plane $\CP$ to itself. Note that on the second
coordinate $F$ is the rational function $f:p\mapsto
\left(\frac{p-1}{p+1}\right)^2$, whereas on the first coordinate
iterations of $F$ are compositions of polynomials
$h_p:z\mapsto\left(\frac{2z}{p+1}-1\right)^2$, where $p$ runs through
a forward orbit of iterations of $f$. Intersections of the Julia set
of $F$ with planes $p=p_0$ are then Julia sets of the non-autonomous
iterations
$\C\stackrel{h_{p_0}}{\arr}\C\stackrel{h_{p_1}}{\arr}\C\stackrel{h_{p_2}}{\arr}\cdots$,
where $p_{n+1}=\left(\frac{p_n-1}{p_n+1}\right)^2$. These Julia sets,
which we will denote $J(p_0)$, are dendrites, see
Figure~\ref{fig:pjuliasets}, if $p_0$ belongs to the Julia set of $f$.

We show that the dendrites $J(p_0)$ are precisely the dendrites
around which the loops $\gamma_w$ go. Moreover, the pillow $\R^2/\Z^2$
can be obtained by gluing two copies of $J(p_0)$ along the
Caratheodory loop going around $J(p_0)$ (one loop is glued to the other using
reflection with respect to a diameter). The curve $\gamma_w$ is the
image of the Caratheodory loop. Moreover, the construction is
dynamical: the map $h_{p_0}$ agrees with the double coverings
$\gamma_{X_1X_2\ldots}\arr\gamma_{X_2X_3\ldots}$.

The paper is organized as follows. We start with a short reminder of
the main notions and techniques of self-similar
groups. Section~\ref{s:julsp} is devoted to the study of dynamics of
the endomorphism $F:\CP\arr\CP$. We start with computation of the
iterated monodromy group of $F$ in Theorem~\ref{th:imgF}. It is a
self-similar group acting on a degree 4 rooted tree. A computationally
more convenient group is an index two extension of $\img{F}$, which is
defined in Subsection~\ref{ss:2extension}. We denote it
$\overline{\group{G}}$, and it is the iterated monodromy group of the
quotient of the dynamical system $F:\CP\arr\CP$ by complex
conjugation.

The group $\img{F}$ contains a natural subgroup $\group{G}$
corresponding to non-autonomous iterations $h_p$ of polynomials on the
first coordinate. It is not transitive on the levels of the rooted
tree, and we get an uncountable family of quotients of $\group{G}$
coming from restricting of $\group{G}$ to invariant binary subtrees. A
similar family of groups was studied
in~\cite{nek:ssfamilies,nek:dendrites,nek:genus,nek:nonuniform}.

The graphs of the action of $\group{G}$ on the levels of binary
sub-trees are approximations of the Julia sets $J(p)$ of the
non-autonomous iterations of $h_p$ (i.e., the corresponding slices of the Julia
set of $F$). They are trees, in accordance with
the fact that $J(p)$ are dendrites. We study these trees
in~\ref{ss:schreiergraphs} using recursive description of the generators of
the self-similar group $\group{G}$. In particular, we describe
inductive algorithms of constructing these graphs, see
Corollaries~\ref{cor:rule1} and~\ref{cor:rule2}.

In the next subsection we study external angles to the Julia sets
$J(p)$, i.e., the Caratheodory loops around them. The bundle of
Caratheodory loops is the limit space of the subgroup $\group{R}$ of
$\img{F}$ generated by the loops not intersecting the Julia sets
$J(p)$. We derive from the structure of the group $\group{R}$ how the
bundle of Caratheodory loops is glued together from a Cantor set of
circles, and which external angles land on the points of the line $z=p$. In
particular, we show that for a countable set of parameters $p$
(equal to the backward orbit of the unique real fixed point of 
$f(p)=\left(\frac{1-p}{1+p}\right)^2$) there are two external rays
to $J(p)$ landing on $(p, p)$, and that in all the other cases such ray is
unique.

In Section~\ref{s:matings} we define matings of the non-autonomous
iterations $h_p$. Namely, for every $p$ in the Julia set of
$f(p)=\left(\frac{1-p}{1+p}\right)^2$ consider the corresponding slice
$J(p)$ of the Julia set of $F$, and the Caratheodory loop $\gamma$
around it. Take then another copy of $J(p)$, and glue them together
along the Caratheodory loops so that one loop is a mirror image of the
other with respect to the diameter containing an external angle
landing on $(p, p)$. Note that in the case when $p$ belongs to the
backward orbit of the real fixed point of $f$, there are two such
rays, and we have therefore two possible choices for the mating.

We define the mating and study it in purely algebraic
terms. We construct an ``amalgam'' of the group $\group{G}$ with
itself, generating a group $\widehat{\group{G}}$ by two copies of
$\group{G}$. Then the inclusion of the two copies of $\group{G}$
induce semi-conjugacies of the limit dynamical systems. We show that
these semi-conjugacies realize the matings as described in the
previous paragraph.

We study then the limit dynamical system of $\widehat{\group{G}}$,
i.e., the obtained matings. We show that the group
$\widehat{\group{G}}$ contains a virtually abelian subgroup
$\group{H}$ such that the inclusion
$\group{H}\hookrightarrow\widehat{\group{G}}$ induces a conjugacy of
the limit dynamical systems. Then we show that the limit space
of $\widehat{\group{G}}$ is a direct product of the Cantor set
$\{\bel, \ber\}^\infty$ with the sphere
$\C/H$, where $H$ is the group of affine transformations of the form
$z\mapsto \pm z+a+ib$ for $a, b\in\Z$, i.e., the pillowcase described
above. The limit dynamical system acts as the
binary one-sided shift on the Cantor set and as
multiplication by $1+i$ or $1-i$ on the pillowcase $\C/H$ (the choice of the
coefficient $1\pm i$ depends on the first letter of the corresponding
element of the one-sided shift).

It follows that the constructed matings are \emph{Latt\`es examples},
though non-autonomous, as we can choose one of two multiplications.

Since $\group{G}<\widehat{\group{G}}$ the graphs of action of
$\group{G}$ on the levels of the tree are sub-graphs of the graphs of
action of $\widehat{\group{G}}$. The graphs of action of
$\widehat{\group{G}}$ are square grids on the pillowcases. We show that
the graphs of action of the two copies of $\group{G}$ partition the edges
of the grid into two disjoint sub-trees, see
Figure~\ref{fig:schrdouble}. These partitions of a square grid into
two subtrees converge to the decomposition of the pillowcase $\C/H$
into two dendrites. 

In Section~\ref{s:paperfolding} we relate the obtained
results about the mating with the paper folding curves. Namely, we
show that the curve separating the two subtrees of the square grid on
$\C/H$ coincide with the paper-folding curves $\gamma_v$ described in
this introduction, and that in the limit the curves $\gamma_v$
converge to the image of the Caratheodory loop in the mating.

The last section~\ref{s:tuning} describes structure of the slices of
the Julia set of $F$ that correspond to the values of $p$ belonging to
the boundaries of the Fatou components of
$f(p)=\left(\frac{p-1}{p+1}\right)^2$. We show that they are obtained
by ``flattening'' the boundaries of the Fatou components of the
polynomials $16z^2(1-z)^2$ and $(2z^2-4z+1)^2$. In other terminology
they are obtained by (rotated) tuning of these polynomials by the
polynomial $z^2-2$. As usual, the proof is carried out using just algebraic
computations with the iterated monodromy groups.

\section{Self-similar groups}
We present here, in a very condensed form, the main definitions and
results of the theory of self-similar and iterated monodromy
groups. For a more detailed account, see~\cite{nek:book,nek:filling,nek:bath,nek:models}.

\subsection{Covering bisets}
\begin{defi}
Let $G$ be a group. A \emph{$G$-biset} is a set $\bim$ together with commuting
left and right actions of $G$ on $\bim$. In other words, we have two
maps $G\times\bim\arr\bim:(g, x)\mapsto g\cdot x$ and $\bim\times
G\arr\bim:(x, g)\mapsto x\cdot g$ satisfying $1\cdot x=x\cdot 1=x$ for
all $x\in\bim$, and:
\[g_1\cdot(g_2\cdot x)=(g_1g_2)\cdot x, \quad(x\cdot g_1)\cdot
g_2=x\cdot (g_1g_2)
\quad (g_1\cdot x)\cdot g_2=g_1\cdot (x\cdot g_2)\]
for all $g_1, g_2\in G$, $x\in\bim$.

We say that $\bim$ is a \emph{covering biset} if the right action of
$G$ on $\bim$ is free, i.e., if $x\cdot g=x$ implies $g=1$. We also
assume then that the number of right orbits is finite.
\end{defi}

Sometimes we also consider $G_1-G_2$ bisets $\bim$, which are sets
with commuting left $G_1$-action and right $G_2$-action.

\begin{defi}
The isomorphism class of a pair $(G, \bim)$, where $G$ is a group and
$\bim$ is a covering $G$-biset is called a \emph{self-similar
  group}. Here two pairs $(G_1, \bim_1)$ and $(G_2, \bim_2)$ are
\emph{isomorphic} (the corresponding self-similar groups are called
\emph{equivalent}) if there exists an isomorphism $\phi:G_1\arr G_2$
and a bijection $f:\bim_1\arr\bim_2$ such that $f(g_1\cdot x\cdot
g_2)=\phi(g_1)\cdot f(x)\cdot \phi(g_2)$ for all $g_1, g_2\in G_1$ and
$x\in\bim_1$.
\end{defi}

Let $\bim_1, \bim_2$ be $G$-bisets. Then their \emph{tensor
  product}, denoted $\bim_1\otimes\bim_2$, is the quotient of
$\bim_1\times\bim_2$ by the equivalence relation $x_1\otimes g\cdot
x_2=x_1\cdot g\otimes x_2$, with the actions $g_1\cdot
(x_1\otimes x_2)\cdot g_2=(g_1\cdot x_1)\otimes (x_2\cdot g_2)$. It is
easy to show that tensor product of two covering bisets is a covering
biset. It is also easy to see that
$(\bim_1\otimes\bim_2)\otimes\bim_3$ is naturally isomorphic to $\bim_1\otimes(\bim_2\otimes\bim_3)$.

Let $\bim$ be a covering $G$-biset. Consider the disjoint union
$\bim^*=\bigcup_{n\ge 0}\bim^{\otimes n}$ of the
tensor powers $\bim^{\otimes n}$, $n\ge 0$, where $\bim^{\otimes 0}$
is the group $G$ with the natural left and right actions on itself. Denote
by $T_\bim$ the set of orbits of the right action of $G$ on $\bim^*$. Then
$G$ acts on $T_\bim$ from the left. The set $T_\bim$ has a natural structure of
a rooted tree, where the right orbit of $x_1\otimes x_2\otimes\cdots\otimes
x_n$ is connected by an edge to the right orbit of $x_1\otimes x_2\otimes\cdots\otimes x_n\otimes
x_{n+1}$ for $x_i\in\bim$. The left action of $G$ on $T_\bim$ is an action
by automorphisms of the rooted tree (the root is the unique right orbit
of the action of $G$ on itself).

The action of $G$ on $T_\bim$ is not faithful in general. Let $N$ be
the kernel of the action. Denote then by $\bim/N$ the set of right
$N$-orbits. It is easy to see that $h\cdot x$ and $x$ belong to one right $N$-orbit
for every $x\in\bim$ and $h\in N$, and that the left and right actions
of $G$ on $\bim$ descend to left and right actions of $G/N$ on
$\bim/N$. We call the self-similar group $(G/N, \bim/N)$ the
\emph{faithful quotient} of the self-similar group $(G, \bim)$.

\subsection{Wreath recursions and virtual endomorphisms}
\label{ss:wreathvirtend}

Let $(G, \bim)$ be a self-similar group. We say that $X=\{x_1, x_2, \ldots,
x_d\}\subset\bim$ is a \emph{basis} if it is a transversal of right $G$-orbits. In other
words, it is such a subset of $\bim$ that for every $x\in\bim$ there
exists a unique $x_i\in X$ such that $x=x_i\cdot g$ for some $g\in
G$. Note that $g$ is also uniquely determined by $x$, since the right
action is free.

For every $n$ the set $X^n=\{a_1\otimes a_2\otimes\cdots\otimes
a_n\;:\;a_i\in X\}$ is a basis of $\bim^{\otimes n}$. Consequently,
$X^*=\bigcup_{n\ge 0}X^n$ is a basis of $\bim^*$. Here $X^0$ consists
of a single \emph{empty word} identified with the identity element of
$G=\bim^{\otimes 0}$. We will usually write $a_1\otimes
a_2\otimes\cdots\otimes a_n$ just as a word $a_1a_2\ldots a_n$.
We get a natural bijection between the set
$T_\bim$ of right orbits of $\bim^*$ and the set $X^*$ of finite words
over $X$. The vertex adjacency on the tree $T_\bim$ corresponds to a
similar adjacency of elements of $X^*$: a word $v\in X^*$ is adjacent to the
words of the form $vx$ for $x\in X$. In other words, we consider $X^*$
as the right Cayley graph of the free monoid generated by $X$. The
action of $G$ on $T_\bim$ is transformed hence to an action of $G$ on
$X^*$. We call this action the \emph{self-similar action} associated with
the basis $X$ of the biset $\bim$.

Let $g\in G$ and $x\in X$, then there exist unique $y\in X$ and $h\in
G$ such that $g\cdot x=y\cdot h$. The induced map $x\mapsto y$ is a
permutation coinciding with the restriction of the action of $g$
to the first level $X\subset X^*$ of the rooted tree $X^*$. Let us
denote this permutation by $\sigma_g\in Symm(X)$. Denote also
$h=g|_x$. We get a map
\[\Phi:g\mapsto \sigma_g(g|_x)_{x\in X}\]
from $G$ to the semidirect product $Symm(X)\ltimes G^X$. The
semi-direct product is called the \emph{(permutational) wreath
  product} of the symmetric group $Symm(X)$ with $G$. 

It is easy to check that $\Phi:G\arr Symm(X)\ltimes G^X$ is a
homomorphism. We call it the \emph{wreath recursion} associated with
the self-similar group $(G, \bim)$ (and the basis $X$). If we change
the basis $X$ to another basis $Y$ and identify $G^X$ and $G^Y$ with
$G^{|X|}=G^{|Y|}$ using some bijections $X\longleftarrow\{1, 2, \ldots, d\}\arr
Y$ (i.e., \emph{orderings} of $X$ and $Y$),
then the wreath recursions associated with $X$ and $Y$ differ from
each other by a post-composition with an inner automorphism of the
wreath product $Symm(d)\ltimes G^d$. 

If the left action of $G$ on the set of right orbits of $\bim$ (i.e.,
on the first level of the tree $T_\bim$) is
transitive, then the self-similar group is uniquely determined by
the associated virtual endomorphism. Let $x\in X$, and denote
by $G_x$ the stabilizer of the vertex $x$ of the tree $X^*$. Then the
\emph{associated virtual endomorphism} $\phi_x$ is the homomorphism
$\phi_x:G_x\arr G$ defined by the condition
\[g\cdot x=x\cdot \phi_x(g).\]

The biset $\bim$ is reconstructed (up to isomorphism) from the
associated virtual endomorphism in the following way. Let
$\phi:G_1\arr G$ be a virtual endomorphism (where $G_1$ is a subgroup
of finite index in $G$). Let $\bim_\phi$ be the set of formal
expressions $[\phi(g_1)g_2]$, for $g_1, g_2\in G$, where $[\phi(g_1)g_2]=[\phi(h_1)h_2]$ if and only
if $h_1^{-1}g_1\in G_1$ and $\phi(h_1^{-1}g_1)=h_2g_2^{-1}$.
This convention agrees with identification of an expression
$[\phi(g_1)g_2]$ with the partial 
transformation
\[[\phi(g_1)g_2]:x\mapsto \phi(xg_1)g_2,\]
of $G$.
Note that $\phi(xg_1)g_2=\phi(xh_1)h_2$ is equivalent to
$\phi(xg_1)=\phi(xg_1)\phi(g_1^{-1}h_1)h_2g_2^{-1}$, hence to $\phi(h_1^{-1}g_1)=h_2g_2^{-1}$.

The set $\bim_\phi$ is invariant with respect to pre- and
post-composition with right translations $x\mapsto xg$. We get hence a
natural $G$-biset structure on $\bim_\phi$. It is given by the formulas
\[[\phi(g_1)g_2]\cdot g=[\phi(g_1)(g_2g)],\qquad
g\cdot[\phi(g_1)g_2)]=[\phi(gg_1)g_2].\]

It is easy to see that if $\phi=\phi_x$ is the virtual endomorphism
associated with a covering biset $\bim$, then $\bim$ is isomorphic to
$\bim_\phi$. Since the action of $G$ on the set of right
orbits is transitive, every element $y\in\bim$ can be written as
$y=g_1\cdot x\cdot g_2$, and then the isomorphism
maps $y$ to $[\phi(g_1)g_2]$.

\subsection{Iterated monodromy groups}

Let $M_1, M_0$ be path connected and locally path connected
topological spaces or orbispaces. A \emph{topological correspondence} is a pair of
maps $f, \iota:M_1\arr M_0$, where $f$ is a finite degree covering map
and  $\iota$ is a continuous map.

Choose a basepoint $t\in M_0$.
Let $\bim$ be the set of pairs $(z, \ell)$, where $z\in f^{-1}(t)$,
and $\ell$ is homotopy class of a path in $M_0$ from $t$ to
$\iota(z)$. The set $\bim$ has a natural structure of a covering
$\pi_1(M_0, t)$-biset. Namely, for a loop $\gamma\in\pi_1(M_0, t)$,
denote by $(z, \ell)\cdot\gamma$ the element $(z, \ell\gamma)$. Here
and in the sequel we compose paths as functions: in a product
$\ell\gamma$ the path $\gamma$ is passed first, then $\ell$.
Denote by $\gamma\cdot (z, \ell)$ the element $(y,
\iota(\gamma_z)\ell)$, where $\gamma_z$ is the unique lift of $\gamma$
by $f$ that starts at $z$, and $y$ is the end of $\gamma_z$.

A basis of $\bim$ is any set $\{(z_1, \ell_1), (z_2, \ell_2), \ldots,
(z_d, \ell_d)$, where $d=\deg f$, and $f^{-1}(t)=\{z_1, z_2, \ldots, z_d\}$.

The faithful quotient of the self-similar group $(\pi_1(M_0, t),
\bim)$ is called the \emph{iterated monodromy group} of the
correspondence $f, \iota:M_1\arr M_0$.

The virtual endomorphism associated with the biset $\bim$ is equal to
$\iota_*:\pi_1(M_1)\arr\pi_1(M_0)$, where $\pi_1(M_1)$ is identified
with a subgroup finite index in $\pi_1(M_0)$ by the isomorphism
$f_*$. It is well defined up to compositions with inner automorphisms
of $\pi_1(M_0)$, as any virtual endomorphism associated with a
self-similar group.

If $f:M\arr M$ is a branched self-covering, then we may transform it
into a topological correspondence by removing from $M$ the closure $P$ of the
union of the forward orbits of branch points of $M$. If $P$
is not big enough, in particular, if it does not disconnect $M$, then
we can consider the topological correspondence $f, \iota:M_1\arr M_0$,
where $M_0=M\setminus P$, $M_1=f^{-1}(M_0)$, and $\iota:M_1\arr M_0$
is the identical embedding. This is done, for example, if $f$ is a
post-critically finite rational function, or a
post-critically finite endomorphism of $\CP$ (which means that $P$ is
a union of a finite number of varieties).
In these cases we represent the elements of $\bim$ just as paths
$\ell=\iota(\ell)$, since their endpoints are uniquely determined.

\subsection{Contracting self-similar groups}

Let $\bim$ be a covering $G$-biset, and let $X$ be a basis of
$\bim$. For every $v\in X^*$ and $g\in G$ denote by $g|_v$ the unique
element such that $g\cdot v=u\cdot g|_v$ for some $u\in X^*$. We call
it the \emph{section} of $g$ in $v$.

\begin{defi}
The self-similar group $(G, \bim)$ (with a chosen basis $X$ of $\bim$)
is said to be contracting if there
exists a finite set $\nuke\subset G$ such that for every $g\in G$
there exists $n$ such that $g|_v\in\nuke$ for every $v\in X^k$ such
that $k\ge n$. The smallest set $\nuke$ satisfying the above condition
is called the \emph{nucleus} of the group.
\end{defi}

It is proved in~\cite[Corollary~2.11.7]{nek:book} that the property of a biset to be hyperbolic does
not depend on the choice of a basis $X$. The nucleus, however,
depends on $X$.

Since a biset $\bim$ and a basis $X$ is uniquely determined, up to an
isomorphism, by the associated wreath recursion, we call a wreath
recursion $G\arr Symm(X)\ltimes G^X$ contracting if the corresponding
self-similar group is contracting. Sometimes we say that a $G$-biset
$\bim$ is \emph{hyperbolic} if the self-similar group $(G, \bim)$ is
contracting. It is easy to see that if $(G, \bim)$ is contracting,
then its faithful quotient is also contracting.

The nucleus $\nuke$ satisfies the property that $g|_x\in\nuke$ for all
$g\in\nuke$ and $x\in X$. We will often represent $\nuke$ as an
\emph{automaton} using its \emph{Moore diagram}. It is the oriented
graph with the set of vertices $\nuke$ in which for every $g\in\nuke$
and $x\in X$ we have an arrow from $g$ to $g|_x$ labeled by $x|y$,
where $y$ is the image of $x$ under the action of $g$ on the first
level of the tree $X^*$, i.e., we have $g\cdot x=y\cdot g|_x$.

Let $(G, \bim)$ be a contracting self-similar group, and let
$X\subset\bim$ be a basis. Consider the space $X^{-\omega}$ of
left-infinite sequences $\ldots x_2x_1$ of elements of $X$. We say
that $\ldots x_2x_1, \ldots y_2y_1\in X^{-\omega}$ are
\emph{asymptotically equivalent} if there exists a sequence $g_n\in G$
taking a finite number of values such that $g_n(x_n\ldots
x_1)=y_n\ldots y_1$ (with respect to the action of $G$ on $X^*$). The
quotient of the topological space $X^{-\omega}$ by the asymptotic
equivalence relation is called the \emph{limit space} of the group
$(G, \bim)$, and is denoted $\lims$. The shift  $\ldots x_2x_1\mapsto\ldots x_3x_2$ agrees with
the asymptotic equivalence relation, so that it induces a continuous
map $\si:\lims\arr\lims$. We call the pair $(\lims, \si)$ the
\emph{limit dynamical system} of the self-similar group.

One can show, see~\cite[Theorem~3.6.3]{nek:book}, that two sequences $\ldots x_2x_1$ and $\ldots
y_2y_1$ are asymptotically equivalent if and only if there exists an
oriented path $\ldots e_2e_1$ of arrows in the Moore diagram of the
nucleus such that $e_n$ is labeled by $x_n|y_n$.

Consider now $X^{-\omega}\times G$, where $G$ is discrete. We write
elements of the space $X^{-\omega}\times G$ as $\ldots x_2x_1\cdot g$
for $x_i\in X$ and $g\in G$. Two sequences $\ldots x_2x_1\cdot g$ and
$\ldots y_2y_1\cdot h$ are asymptotically equivalent if there exists a
sequence $g_k\in G$ taking a finite set of values such that $g_n\cdot
x_n\ldots x_2x_1\cdot g=y_n\ldots y_2y_1\cdot h$ in $\bim^{\otimes n}$
for all $n$. One can  show that $\ldots x_2x_1\cdot g$ and $\ldots
y_2y_1\cdot h$ are asymptotically equivalent if and only if there
exists an oriented path $\ldots e_2e_1$ in the Moore diagram of the
nucleus such that $e_n$ is labeled by $x_n|y_n$ for every $n$ and the
last vertex of the path is $hg^{-1}$.

The quotient of $X^{-\omega}\times G$ by the asymptotic
equivalence relation is called the \emph{limit $G$-space} and is
denoted $\limg$. The group $G$ acts naturally on $X^{-\omega}\times G$
by right multiplication. This action agrees with the asymptotic
equivalence relation, so that it induces a right action of $G$ on
$\limg$ by homeomorphisms.

For every $x\in\bim$ and $\ldots x_2x_1\cdot g\in X^{-\omega}\times G$
the asymptotic equivalence class of $\ldots x_2x_1\cdot g\otimes
x=\ldots x_2x_1y\cdot h$, where $h\in G$ and $x\in X$ are such that
$g\cdot x=y\cdot h$, is uniquely determined by $x$ and the asymptotic
equivalence class of $\ldots x_2x_1\cdot g$. It follows that we get a
well defined continuous map $\xi\mapsto\xi\otimes x$ of $\limg$ to
itself.

The biset structure of $\bim$ agrees with the maps
$\xi\mapsto\xi\otimes x$ on $\limg$, so that $\left(\xi\cdot g_1\otimes
x\right) g_2=\xi\otimes (g_1\cdot x\cdot g_2)$ for all
$\xi\in\limg$, $g_1, g_2\in G$, $x\in\bim$. Moreover, we have the
following rigidity theorem, see~\cite[Theorem~3.4.13]{nek:book}.

\begin{theorem}
\label{th:uniquenesslimg}
Let $\bim$ be a hyperbolic $G$-biset.
Let $\mathcal{X}$ be a metric space such that $G$ acts on $\mathcal{X}$
co-compactly and properly by isometries from the right. Suppose that for every
$x\in\bim$ we have a continuous strictly contracting map
$\xi\mapsto\xi\otimes x$ such that $\left(\xi\cdot g_1\otimes x\right)\cdot
  g_2=\xi\otimes (g_1\cdot x\cdot g_2)$ for all $\xi\in\mathcal{X}$, $g_1,
  g_2\in G$, and $x\in\bim$. Then there exists a homeomorphism
  $\Phi:\mathcal{X}\arr\limg$ such that $\Phi(\xi\cdot g)=\Phi(\xi)\cdot g$ and
  $\Phi(\xi\otimes x)=\Phi(\xi)\otimes x$ for all $\xi\in\mathcal{X}$, $g\in
  G$, and $x\in\bim$.
\end{theorem}

\subsection{Contracting correspondences}

Let $f, \iota:M_1\arr M_0$ be a topological correspondence. Its
\emph{limit space} $M_\infty$ is subspace of all sequences $(x_1, x_2,
\ldots)\in M_1^\infty$ such that $f(x_n)=\iota(x_{n+1})$. For example, if $\iota$ is an
identical embedding, then $M_\infty$ is the intersection of the
domains of all iterations of the partial map $f$.

The shift $(x_1, x_2, \ldots)\mapsto (x_2, x_3, \ldots)$ is a
continuous self-map on $M_\infty$, which we will denote $f_\infty$.

\begin{defi}
Let $f, \iota:M_1\arr M_0$ be a topological correspondence. We say
that it is \emph{contracting} if $M_0$ is a compact length metric space
(i.e., there is a notion of length of arcs such that distance between
two points is the infimum of lengths of arcs connecting them), and
$\iota$ is contracting with respect to the length metric on $M_0$ and
the lift of the length metric from $M_0$ to $M_1$ by $f$.
\end{defi}

In particular, if $f$ is expanding, and $\iota$ is an identical
embedding, then the correspondence $f, \iota:\M_1\arr M_0$ is contracting.

One can show that the iterated
monodromy group of a contracting topological correspondence is a
contracting self-similar group, see~\cite{nek:models}.

\begin{theorem}
\label{th:limspcontracting}
The limit dynamical system of a contracting topological correspondence $\mathcal{F}$
is topologically conjugate to the limit dynamical system of the
iterated monodromy group of $\mathcal{F}$.
\end{theorem}

The correspondence between the contracting self-similar groups and
expanding self-coverings is functorial in a precise way, see~\cite{nek:filling}. For
example, any embedding of self-similar contracting groups (preserving
self-similarity) induces a semi-conjugacy of their limit dynamical systems.

\section{The Julia set of an endomorphism of $\CP$}
\label{s:julsp}

\subsection{The endomorphism and its iterated monodromy group}

Consider the following map on $\C^2$:
\[F(z, p)=\left(\left(\frac{2z}{p+1}-1\right)^2,
\left(\frac{p-1}{p+1}\right)^2\right).\]

It can be extended to an endomorphism of $\CP$ given in
homogeneous coordinates by the formula
\[F[z:p:u]=[(2z-p-u)^2:(p-u)^2:(p+u)^2].\]
Note that this map has no points of indeterminacy, since
$(2z-p-u)^2=(p-u)^2=(p+u)^2=0$ implies $p=u=z=0$. The Jacobian of
the map $F$ is
\[\left|\begin{array}{ccc}4(2z-p-u) & 0 & 0\\ -2(2z-p-u)^2 & 2(p-u) & 2(p+u)\\
-2(2z-p-u) & -2(p-u) &
2(p+u)\end{array}\right|=32(2z-p-u)(p-u)(p+u),\] hence the
critical locus consists of three lines $2z-p-u=0$, $p=u$, and
$p+u=0$. Their orbits under the action of $F$ are
\[\{2z-p-u=0\}\mapsto\{z=0\}\mapsto\{z=u\}\mapsto\{z=p\}\mapsto\{z=p\},\]
\[\{p=-u\}\mapsto\{u=0\}\mapsto\{p=u\}\mapsto\{p=0\}\mapsto\{p=u\}.\]
We see that the post-critical set of $F$ is the union of the six
lines $z=0$, $z=u$, $z=p$, $p=0$, $p=u$, and $u=0$. (Or, in affine
coordinates, $z=0$, $z=1$, $z=p$, $p=0$, $p=1$, and the line at
infinity.)

The map $F$ is a particular case of a general class of post-critically finite
skew-product maps related to the Teichm\"uller theory of
post-critically finite branched self-coverings of the
sphere (\emph{Thurston maps}). See~\cite{bartnek:rabbit}, where the map $F$ was (somewhat
implicitly) constructed,
and~\cite{koch:french,buffepsteinkochpilgrim:pullback} where other different
classes of similar examples are studied.

Denote by $J_2$ the Julia set of $F$, i.e., the set of points without
neighborhoods on which the sequence $F^{\circ n}$ is normal. Denote by
$J_1$ the support of the measure of maximal entropy of $F$, which
coincides with the attractor of backward iterations of $F$. Both sets
are completely $F$-invariant and we have $J_1\subset J_2$, see more in~\cite{fornsibon:higher}.

\begin{proposition}
The limit dynamical system of the iterated monodromy group of $F$
is topologically conjugate with the action of $F$ on $J_1$.
\end{proposition}

\begin{proof}
By~\cite[Theorem~5.5.3]{nek:book} (see also
Theorem~\ref{th:limspcontracting} in our paper),
it is enough to construct an orbifold metric
on a neighborhood of $J_1$ with respect to which $F$ is expanding.
Let $U$ be an open subset of
$\C\setminus\{0, 1\}$ containing the Julia set of the rational
function $f(p)=\left(\frac{p-1}{p+1}\right)^2$ and such
that $f^{-1}(U)\subset U$. Consider the inverse image $W$ of $U$ in $\CP$ under the
projection map $(z, p)\mapsto p$ of $\C^2\subset\CP$ onto $\C$. Note
that the lines $p=0$ and $p=1$ are disjoint from $W$, hence the
intersection points of the lines $z=p$, $z=0$, $z=1$, $p=0$
and $p=1$ do not belong to $W$. Consider the orbifold with the
underlying space $W$, where the lines $z=0$, $z=1$ and $z=p$
are singular with the isotropy
groups of order 2 uniformized in an atlas of the orbifold as rotations
by $180^\circ$ in the $z$-planes (that are projected to the identity map on the
$p$-plane).
The function $F$ can be realized as a covering $F:W_1\arr W$ of a
sub-orbispace $W_1$ of $W$, where $W_1$ is the orbifold with the
underlying space $F^{-1}(W)$ and singular lines $z=0$, $z=1$, $z=p$,
$z=p+1$ of order two (also uniformized by a rotation in the $z$-planes).

The orbifold $W$ is a locally
trivial bundle over the set $U$ with hyperbolic fibers (as the
fundamental group of every fiber is free product of three copies of the
group of order two). It follows from Proposition~3.2.2 and
Theorem~3.2.15 of~\cite{kobayashi:spaces} that the orbifold $W$ is Kobayashi
hyperbolic (i.e., that its universal covering is Kobayashi
hyperbolic). Consequently, the embedding of orbifolds $\iota:W_1\arr
W$ is contracting with respect to the Kobayashi metrics on $W_1$ and
$W$, while the map $F:W_1\arr W$ is a local
isometry. Theorem~\ref{th:limspcontracting} finishes the proof.
\end{proof}

An important property of the map $F$, greatly facilitating its
study is a skew-product structure: the second coordinate
$\left(\frac{p-1}{p+1}\right)^2$ of $F(z, p)$ depends only on $p$.
See Figure~\ref{fig:pjul} for the Julia set of
$f(p)=\left(\frac{p-1}{p+1}\right)^2$ together with marked
post-critical points $0$
and $1$.

\begin{figure}[h]
\centering
\includegraphics{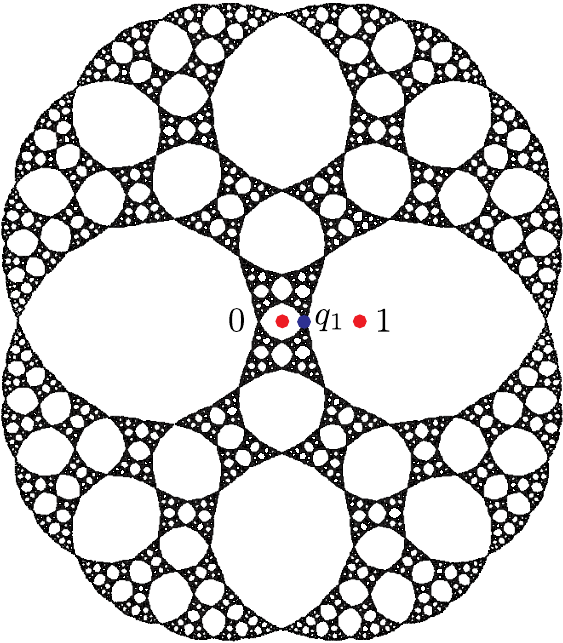}
\caption{Julia set of $\left(\frac{p-1}{p+1}\right)^2$}
\label{fig:pjul}
\end{figure}

On the first coordinate we have a quadratic polynomial
$h_p(z)=(2z/(p+1)-1)^2$ in $z$, depending on the parameter $p$. This
makes it possible, in particular, to draw the intersections of the
Julia set of $F$ with the $z$-lines 
$p=p_0$. See Figures~\ref{fig:pjuliasets} and~\ref{fig:pjuliav},
where some slices of the Julia set of $F$ are shown. We will denote by
$J_1(q)$ the intersection of the Julia set $J_1$ of $F$ with the line $p=q$.

\begin{figure}[h]
\centering
\includegraphics{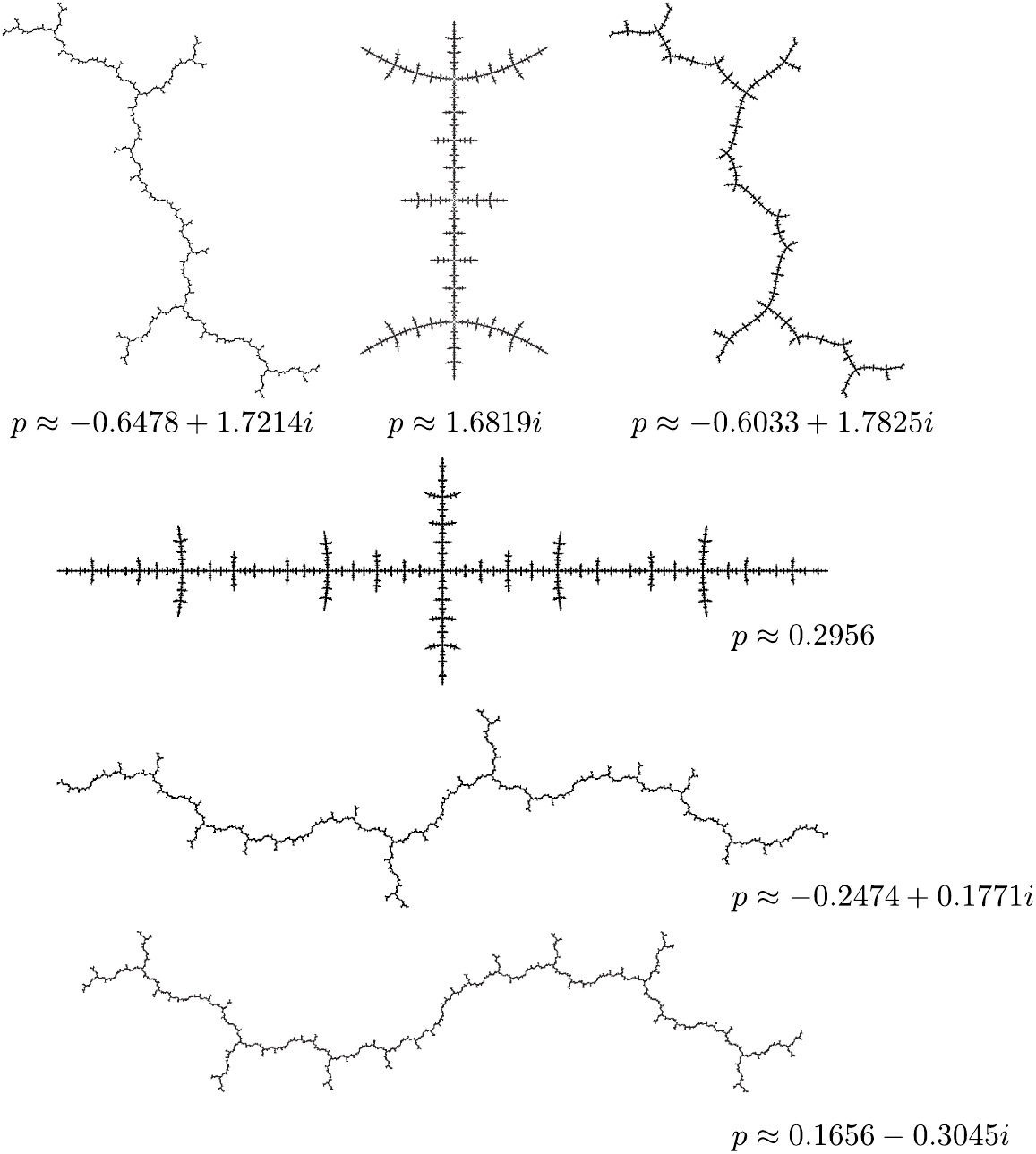}
\caption{Slices of the Julia set $J_1$ of
$F$}\label{fig:pjuliasets}
\end{figure}

\begin{figure}[h]
\centering
\includegraphics{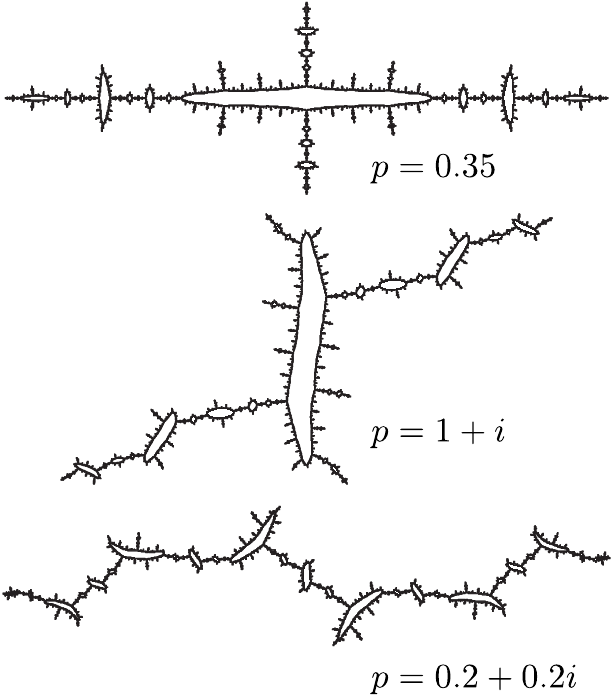}
\caption{Slices of the Julia set $J_2$ of $F$}\label{fig:pjuliav}
\end{figure}

The rational function $\left(\frac{p-1}{p+1}\right)^2$ has three
fixed points \[q_0\approx -0.6478 + 1.7214i,\quad \overline{q_0}\approx-0.6478-1.7214i, \quad
q_1\approx 0.2956.\] The corresponding polynomials
$h_q(z)=\left(\frac{2z}{q+1}-1\right)^2$, for $q\in\{q_0, \overline{q_0},
q_1\}$, are post-critically finite, with
the dynamics
\[0\mapsto 1\mapsto q\mapsto q\]
on the post-critical set. These polynomials and their iterated
monodromy groups were studied in~\cite{bartnek:rabbit}.

Consider the polynomial for $q_0\approx -0.6478 + 1.7214i$. Its
iterated monodromy group is generated by
\begin{eqnarray}
\label{eq:abc1}
\alpha &=& \sigma,\\
\label{eq:abc2}
\beta &=& (1, \alpha),\\
\label{eq:abc3}
\gamma &=& (\gamma, \beta),
\end{eqnarray}
where $\alpha$, $\beta$ and $\gamma$ are loops around $0, 1$ and
$p$, respectively (see a general formula for iterated monodromy
groups of quadratic polynomials in~\cite{bartnek:mand}).

We can interpret the complement $\M$ of the post-critical set of $F$ as
the configuration space of pairs of complex numbers $(z, p)$ which are
different from $\infty, 0, 1$ and from each other.
We can identify the loops $\alpha,
\beta, \gamma$ with the loops in the configuration space $\M$ coming from
$p$ staying and $z$ traveling along the loops $\alpha, \beta, \gamma$
in the line $p=q_0$.

Let $S$ and $P$ be the elements of the fundamental group of $\M$ uniquely
determined by the following relations (see~\cite{nek:dendrites} for details)
\begin{align}\label{eq:conj1}
P\alpha P^{-1}= & \beta\alpha\beta^{-1},
&\quad S\alpha S^{-1}= & \alpha\gamma\alpha\gamma^{-1}\alpha^{-1},\\
\label{eq:conj2} P\beta P^{-1}= & \beta\alpha\beta\alpha^{-1}\beta^{-1},&\quad
S\beta S^{-1}= & \beta,\\
\label{eq:conj3} P\gamma P^{-1}= &
\gamma,&\quad S\gamma S^{-1}= &
\alpha\gamma\alpha^{-1}.
\end{align}

Denote also $T=\gamma S^{-1}P^{-1}\beta\alpha$. We have then:
\begin{align}T\alpha T^{-1}= & \alpha,\\
T\beta
T^{-1} = & \gamma \beta \gamma^{-1},\\
T\gamma T^{-1}= & \gamma\beta\gamma\beta^{-1}\gamma^{-1}.
\end{align}

\begin{proposition}
The virtual endomorphism associated with $F$ is given by
\begin{gather*}\phi(\alpha^2)=1,\quad\phi(\beta)=1,\quad\phi(\gamma)=\gamma,\\
\phi(\alpha^{-1}\beta\alpha)=\alpha,\quad\phi(\alpha^{-1}\gamma\alpha)=\beta,\\
\phi(S^2)=\beta\alpha\gamma S^{-1}P^{-1},\quad\phi(T)=P,\quad\phi(P^2)=1.
\end{gather*}
\end{proposition}

\begin{proof}
The elements of the fundamental group of
$\M$ are uniquely determined by their action by conjugation on the
normal subgroup generated by $\alpha, \beta$ and $\gamma$. Let us use
this fact (following~\cite{bartnek:rabbit} and~\cite{nek:dendrites}) to compute the virtual endomorphism of
$\pi_1(\M)$ associated with the partial self-covering $F$.

The domain of the restriction of the virtual endomorphism onto
$\langle\alpha, \beta, \gamma\rangle$ (associated with the first
coordinate of the recursion~\eqref{eq:abc1}--\eqref{eq:abc3})
is generated by $\alpha^2, \beta, \gamma,
\alpha^{-1}\beta\alpha$ and $\alpha^{-1}\gamma\alpha$. The action of
the virtual endomorphism is given by
\begin{gather*}\phi(\alpha^2)=1,\quad\phi(\beta)=1,\quad\phi(\gamma)=\gamma,\\
\phi(\alpha^{-1}\beta\alpha)=\alpha,\quad
\phi(\alpha^{-1}\gamma\alpha)=\beta.
\end{gather*}

The domain of the induced virtual endomorphism is generated by the
above generators of the domain of $\phi$ and the automorphisms $S^2$,
$P^2$ and $T=\gamma S^{-1}P^{-1}\beta\alpha$.

Denote $\tau=\gamma^{-1}\alpha^{-1}\beta^{-1}$. A direct computation shows that $\tau$
commutes with $S$ and with $P$.
We also have
\[\tau^{-1}S^{-1}P^{-1}\alpha
PS\tau=\beta\gamma\beta^{-1}\gamma^{-1}\alpha\gamma\beta
\gamma^{-1}\beta,\]
\[\tau^{-1}S^{-1}P^{-1}\beta
PS\tau=\beta\gamma\beta\gamma^{-1}\beta^{-1},\]
and
\[\tau^{-1}S^{-1}P^{-1}\gamma PS\tau=\beta\gamma\beta^{-1}.\]

We have
\begin{multline*}
\phi(S^2)\alpha\phi(S^{-2})=\phi(S^2\alpha^{-1}\beta\alpha
S^{-2})=\\
\phi((\alpha\gamma)^2\alpha^{-1}(\alpha\gamma)^{-2}
\beta(\alpha\gamma)^2\alpha(\alpha\gamma)^{-2})=\\
\beta\gamma\beta^{-1}\gamma^{-1}\alpha\gamma\beta\gamma^{-1}\beta^{-1},
\end{multline*}
\[\phi(S^2)\beta\phi(S^{-2})=\phi(S^2\alpha^{-1}\gamma\alpha S^{-2})=
\phi((\alpha\gamma)^2\alpha^{-1}\gamma\alpha(\alpha\gamma)^{-2})=
\beta\gamma\beta\gamma^{-1}\beta^{-1},\]
\[\phi(S^2)\gamma\phi(S^{-2})=\phi(S^2\gamma S^{-2})=
\phi(\alpha\gamma\alpha\gamma\alpha^{-1}\gamma^{-1}\alpha^{-1})=
\beta\gamma\beta^{-1}.\] We get that
\[\phi(S^2)=\tau^{-1}S^{-1}P^{-1}.\]

We have
\[
\phi(T)\alpha\phi(T^{-1})=\phi(T\alpha^{-1}\beta\alpha T^{-1})=\\
\phi(\alpha^{-1}\gamma\beta\gamma^{-1}\alpha^{-1})=\beta\alpha\beta^{-1},
\]
\[
\phi(T)\beta\phi(T^{-1})=\phi(T\alpha^{-1}\gamma\alpha T^{-1})=\\
\phi(\alpha^{-1}\gamma\beta\gamma\beta^{-1}\gamma^{-1}\alpha)=
\beta\alpha\beta\alpha^{-1}\beta^{-1},
\]
and
\[
\phi(T)\gamma\phi(T^{-1})=\phi(T\gamma T^{-1})=
\phi(\gamma\beta\gamma\beta^{-1}\gamma^{-1})=\gamma,
\]
which implies that $\phi(T)=P$.

It remains to compute $\phi(P^2)$.
\[
\phi(P^2)\alpha\phi(P^{-2})=\phi(P^2\alpha^{-1}\beta\alpha
P^{-2})=\phi(\beta\alpha\beta\alpha^{-1}\beta^{-1})=\alpha,
\]
\begin{multline*}
\phi(P^2)\beta\phi(P^{-2})=\phi(P^2\alpha^{-1}\gamma\alpha P^{-2})=\\
\phi(\beta\alpha\beta\alpha^{-1}\beta^{-1}\alpha^{-1}\beta^{-1}\gamma
\beta\alpha\beta\alpha\beta^{-1}\alpha^{-1}\beta^{-1})=\beta,
\end{multline*}
\[
\phi(P^2)\gamma\phi(P^{-2})=\phi(P^2\gamma P^{-2})=\phi(\gamma)=\gamma,
\]
which implies that $\phi(P^2)=1$.
\end{proof}

\begin{theorem}
\label{th:imgF} The iterated monodromy group $\img{F}$ is generated by
the wreath recursion
\begin{eqnarray*}
\alpha &=& \sigma(\beta, \beta^{-1}, \beta\alpha,
\alpha^{-1}\beta^{-1}),\\
\beta &=& (1, \beta\alpha\beta^{-1}, \alpha, 1),\\
\gamma &=& (\gamma, \beta, \gamma, \beta),\\
P &=& \pi,\\
S &=& \sigma\pi(P^{-1}\tau^{-1}, P^{-1}, S^{-1}\tau^{-1}, S^{-1}),
\end{eqnarray*}
where $\sigma=(12)(34)$, $\pi=(13)(24)$, and $\tau=\gamma^{-1}\alpha^{-1}\beta^{-1}$.
\end{theorem}

Note that it follows from the recursions that the elements
$\alpha, \beta, \gamma, P$ of the iterated monodromy group are
involutions, hence the wreath recursion can be written as
\begin{eqnarray*}
\alpha &=& \sigma(\beta, \beta, \beta\alpha,
\alpha\beta),\\
\beta &=& (1, \beta\alpha\beta, \alpha, 1),\\
\gamma &=& (\gamma, \beta, \gamma, \beta),\\
P &=& \pi,\\
S &=& \sigma\pi(P\tau^{-1}, P, S^{-1}\tau^{-1}, S^{-1}),
\end{eqnarray*}
where $\tau=\gamma\alpha\beta$.

Note that the subgroup $\langle\alpha, \beta, \gamma\rangle$ of
$\img{F}$ is self-similar. It is the subgroup generated by loops of
the form $\ell\gamma\ell^{-1}$, where $\ell$ is a path, and $\gamma$
is a loop inside the plane $p=q$ for some $q\in\C\setminus\{0, 1\}$.
We denote this subgroup $\group{G}$.

\begin{proof}
Denote by $\bel_0$ the element $[\phi(1)1]$ of the biset $\bim_\phi$. Denote then
\[\bel_1=\tau\cdot\bel_0,\quad\ber_0=P\cdot\bel_0,\quad\ber_1=P\tau\cdot\bel_0,\]
and order the basis of the biset associated with $\phi$ in the
sequence $(\bel_0, \bel_1, \ber_0, \ber_1)$.

We have then
\[\alpha\cdot\bel_0=\tau\cdot\bel_0\cdot\phi(\tau^{-1}\alpha)=
\bel_1\cdot\phi(\beta\alpha\gamma\alpha)=\bel_1\cdot\beta,\]
\[\alpha\cdot\bel_1=\bel_0\cdot\phi(\alpha\tau)=
\bel_0\cdot\phi(\alpha\gamma^{-1}\alpha^{-1}\beta^{-1})=\bel_0\cdot\beta^{-1},\]
\[\alpha\cdot\ber_0=P\tau\cdot\bel_0\cdot\phi(\tau^{-1}P^{-1}\alpha
P)=\bel_1\cdot\phi(\beta\alpha\gamma\alpha^{-1}\beta^{-1}\alpha\beta\alpha)=
\bel_1\cdot \beta\alpha,\] and
\[\alpha\cdot\ber_1=P\cdot\bel_0\cdot\phi(P^{-1}\alpha
P\tau)=\ber_0\cdot\phi(\alpha^{-1}\beta^{-1}\alpha\beta\alpha\gamma^{-1}\alpha^{-1}\beta^{-1})=
\ber_0\cdot \alpha^{-1}\beta^{-1}.\] Consequently,
\[\alpha=\sigma(\beta, \beta^{-1}, \beta\alpha,
\alpha^{-1}\beta^{-1}).\]

We have
\[\beta\cdot\bel_0=\bel_0\cdot\phi(\beta)=\bel_0\cdot 1,\]
\[\beta\cdot\bel_1=\tau\cdot\bel_0\cdot\phi(\tau^{-1}\beta\tau)=
\bel_1\cdot\phi(\beta\alpha\gamma\beta\gamma^{-1}\alpha^{-1}\beta^{-1})=
\bel_1\cdot\beta\alpha\beta^{-1},\]
\[\beta\cdot\ber_0=P\cdot\bel_0\cdot\phi(P^{-1}\beta
P)=\ber_0\cdot\phi(\alpha^{-1}\beta\alpha)=\ber_0\cdot\alpha,\]
and
\[\beta\cdot\ber_1=P\tau\cdot\bel_0\cdot\phi(\tau^{-1}P^{-1}\beta
P\tau)=\ber_1\cdot\phi(\beta\alpha\gamma\alpha^{-1}\beta\alpha\gamma^{-1}\alpha^{-1}\beta^{-1})=
\ber_1\cdot 1,\] hence
\[\beta=(1, \beta\alpha\beta^{-1}, \alpha, 1).\]

We have
\[\gamma\cdot\bel_0=\bel_0\cdot\phi(\gamma)=\bel_0\cdot\gamma,\]
\[\gamma\cdot\bel_1=\tau\cdot\bel_0\cdot\phi(\tau^{-1}\gamma\tau)=
\bel_1\cdot\phi(\beta\alpha\gamma\alpha^{-1}\beta^{-1})=\bel_1\cdot\beta,\]
and also
\[\gamma\cdot\ber_0=\ber_0\cdot\gamma,\quad\gamma\cdot\ber_1=\ber_1\cdot\beta,\]
since $P$ commutes with $\gamma$.

Since $\phi(P^2)=1$ and $P$ commutes with $\tau$, we have
\[P=\pi.\]

Finally,
\begin{multline*}S\cdot\bel_0=P\tau\cdot\bel_0\cdot\phi(\tau^{-1}P^{-1}S)=
\ber_1\cdot\phi(\beta\alpha\gamma \cdot P^{-2}\beta\alpha \cdot
\alpha^{-1}\beta^{-1}PS\gamma^{-1}\cdot\gamma)=\\
\ber_1\cdot\phi(P^{-2}\cdot\beta\alpha\gamma\beta\alpha\cdot
T^{-1}\cdot\gamma)=\ber_1\cdot\beta\alpha P^{-1}
\gamma=\ber_1\cdot\tau^{-1}P^{-1}.
\end{multline*}
Since $T$ commutes with $\alpha$, we have:
\begin{multline*}
S\cdot\bel_1=P\cdot\bel_0\cdot\phi(P^{-1}S\tau)=\\
\ber_0\cdot\phi(P^{-2}\beta\alpha\cdot
\alpha^{-1}\beta^{-1}PS\gamma^{-1}\cdot\alpha^{-1}\beta^{-1})=
\ber_0\cdot\phi(P^{-2}\beta\alpha T^{-1}\alpha^{-1}\beta^{-1})=\\
\ber_0\cdot\phi(\beta T^{-1}\beta^{-1})=\ber_0\cdot P^{-1}.
\end{multline*}
Since $S$ commutes with $\alpha\gamma$ and $T$ commutes with
$\alpha$, we have:
\begin{multline*}S\cdot\ber_0=\tau\cdot\bel_0\cdot\phi(\tau^{-1}SP)=\\
\bel_1\cdot\phi(\beta\alpha\gamma S^2\gamma^{-1}\cdot \gamma
S^{-1}P^{-1}\beta\alpha\cdot \alpha^{-1}\beta^{-1} P^2)=
\bel_1\cdot\phi(\beta S^2\alpha\gamma\alpha^{-1} T\beta^{-1})=\\
\bel_1\cdot \tau^{-1}S^{-1}P^{-1} P=\bel_1\cdot\tau^{-1} S^{-1},
\end{multline*}
and, since $P$ and $S$ commute with $\tau$ and $\gamma$ commutes
with $P$:
\begin{multline*}
S\cdot\ber_1=\bel_0\cdot\phi(SP\tau)=
\bel_0\cdot\phi(S^2\gamma^{-1}\cdot\gamma
S^{-1}P^{-1}\beta\alpha\cdot\alpha^{-1}\beta^{-1}
P^2\gamma^{-1}\alpha^{-1}\beta^{-1})=\\
\bel_0\cdot\phi(S^2\gamma^{-1}
T\alpha^{-1}\beta^{-1}\gamma^{-1}\alpha^{-1}\beta^{-1} P^2)=
\bel_0\cdot \tau^{-1}S^{-1}P^{-1}\gamma^{-1} P
\alpha^{-1}\beta^{-1}=\\ \bel_0\cdot \tau^{-1}S^{-1}\tau=\bel_0\cdot
S^{-1},
\end{multline*}
which implies that
\[S=\sigma\pi(P^{-1}\tau^{-1}, P^{-1}, S^{-1}\tau^{-1}, S^{-1}),\]
which finishes the proof.
\end{proof}

Computation (for example using the GAP packages~\cite{barth:GAP}
or~\cite{muntyansavchuk:gap})
gives the following nucleus of $\img{F}$:
\begin{multline}\label{eq:nucleusimg}\{1, \alpha, \beta, \gamma, \alpha^\beta,
 \beta^\alpha, \gamma^\alpha, \gamma^\beta, \gamma^{\alpha\beta},
P, \gamma P, \gamma^{\alpha\beta} P\}
\cup\\
\{\alpha\beta, \alpha\gamma, \beta\gamma, \tau, \alpha\tau, \beta\tau,\\
S, \alpha S, S\beta, S\gamma,  \alpha\beta S, S\alpha\beta, S\gamma\beta,
S\beta\gamma,  S\gamma\alpha, S\alpha^\beta,
\gamma^{\alpha\beta} S, \tau S, \beta\tau S, \alpha S\alpha\beta,\\
P\alpha, P\beta, P\alpha\beta, P\tau
\}^{\pm 1}\end{multline}
consisting of 60 elements.

\subsection{An index 2 extension of $\img{F}$}
\label{ss:2extension}

\begin{defi}
\label{def:overlineG}
Denote by $\overline{\group{G}}$ the group generated by the elements
\begin{alignat*}{3}
\alpha &=\sigma,&\quad \beta&=(1, \alpha, \alpha, 1),&\quad\gamma&=(\gamma,
\beta, \gamma, \beta),\\
a&=\pi,&\quad b&=(a, a, \alpha a, \alpha a),&\quad c&=(\beta b,
\beta b,
\gamma c, \gamma c),
\end{alignat*}
where $\sigma=(12)(34)$ and $\pi=(13)(24)$.
\end{defi}

It is easy to check that $a^2=b^2=c^2=1$,
\begin{equation}\label{eq:commutationrel}
a\beta a=\alpha\beta\alpha,\quad b\gamma b=\beta\gamma\beta,\quad
c\beta c=\gamma\beta\gamma,
\end{equation} and that $a, b, c$ commute with the
remaining generators $\alpha, \beta, \gamma$.

Denote by $(\bee_{00}, \bee_{01}, \bee_{10}, \bee_{11})$ the ordered
basis of the biset in the definition of the group
$\overline{\group{G}}$.
Then $\sigma(\bee_{i, j})=\bee_{i, (1-j)}$ and $\pi(\bee_{i,
  j})=\bee_{(1-i), j}$.

Recall that by $\group{G}$ we denote the subgroup of $\img{F}$
generated by $\alpha, \beta, \gamma$. We will see below that it is
equivalent as a self-similar group to the subgroup of
$\overline{\group{G}}$ generated by $\alpha, \beta, \gamma$.

\begin{proposition}
\label{pr:imgind2} The group $\img{F}$ is isomorphic (as a
self-similar group) to an index two subgroup of the group
$\overline{\group{G}}$. The
isomorphism maps $\alpha, \beta, \gamma$ to the corresponding
generators of $\group{G}$ and maps $S$ and $P$ to $ac\alpha\gamma$ and
$\beta ba$, respectively. The bisets of the wreath recursions
for $\img{F}$ (as in Theorem~\ref{th:imgF}) and $\overline{\group{G}}$
(as in Definition~\ref{def:overlineG}) are identified with each other
using the equalities
\[\{\bel_0=\bee_{00},\quad
\bel_1=\bee_{01}\cdot\beta,\quad \ber_0=\bee_{10}\cdot a,\quad
\ber_1=\bee_{11}\cdot\beta\alpha a\}.\]
\end{proposition}

\begin{proof}
Recall, that the group $\overline{\group{G}}$ is generated by
\[\alpha=\sigma,\quad\beta=(1, \alpha, \alpha, 1),\quad\gamma=(\gamma,
\beta, \gamma, \beta),\]
\[a=\pi,\quad b=(a, a, \alpha a, \alpha a),\quad
c=(\beta b, \beta b, \gamma c, \gamma c).\]

We have
\[\beta^a=\alpha\beta\alpha,\quad \gamma^b=\beta\gamma\beta,\quad
\beta^c=\gamma\beta\gamma,\] and $a, b, c$ commute with the
remaining generators $\alpha, \beta, \gamma$.

Conjugating the right hand side of the recursion by $(1, \beta, a,
\beta\alpha a)$, we get
\[\alpha=\sigma(\beta, \beta, a\alpha\beta a,
a\beta\alpha a)=\sigma(\beta, \beta, \beta\alpha, \alpha\beta),\]
\[\beta=(1, \beta\alpha\beta, a\alpha a, 1)=(1, \beta\alpha\beta,
\alpha, 1),\]
\[\gamma=(\gamma, \beta, a\gamma a, a\alpha\beta\alpha
a)=(\gamma, \beta, \gamma, \beta),\]
\[a=\pi(a, \alpha a, a, \alpha a),\]
\[b=(a, \beta a\beta , a\alpha aa, a\alpha\beta \alpha a\beta\alpha
a)=(a, \beta\alpha\beta\alpha a, \alpha a, \alpha a),\]
\[c=(\beta b, \beta \beta b\beta, a\gamma c a, a\alpha\beta\gamma c\beta\alpha
a)=(\beta b, \beta b, \gamma aca, \alpha \gamma \alpha aca).\]

Let us show that $S=ac\alpha\gamma$, and $P=\beta ba$ satisfy the
wreath recursion of Proposition~\ref{pr:imgind2}.

We have, using commutation of the involutions $a$ and $\alpha$:
\[\beta ba=\pi(a\alpha a\alpha aa, \alpha a\alpha a, aa,
\beta\alpha\beta\beta\alpha\beta\alpha a\alpha a)=\pi,\]
which agrees with the condition $P=\pi$.

We have
\begin{multline*}ac\alpha\gamma=\pi(a, \alpha a, a, \alpha
a)(\beta b, \beta b, \gamma aca, \alpha\gamma\alpha
aca)\sigma(\beta, \beta, \beta\alpha, \alpha\beta)(\gamma, \beta,
\gamma, \beta)=\\
\pi\sigma(\alpha a b \gamma, a b\beta, \gamma\alpha
ca\beta\alpha\gamma, \gamma ca\alpha)=
\pi\sigma(ab\beta\cdot\beta\alpha\gamma, ab\beta, \gamma\alpha
ca\beta\alpha\gamma, \gamma\alpha ca),
\end{multline*}
which also agrees with
\[S=\pi\sigma(P^{-1}\tau^{-1}, P^{-1}, S^{-1}\tau^{-1}, S^{-1}),\]
and finishes the proof.
\end{proof}

\subsection{Properties of the groups $\group{G},
  \overline{\group{G}}$, and $\img{F}$}
\label{ss:propertiesofgroups}

Note that for every element $g\in\group{G}$ and for every $\bee_{ij}$ we
have $g\cdot\bee_{ij}=\bee_{ik}\cdot h$ for some $h\in\group{G}$ and
$k\in\{0, 1\}$. In other words, the self-similarity biset of
$\group{G}$ is a disjoint union (``direct sum'') of the
bisets $\bim_0=\{\bee_{00}, \bee_{01}\}\cdot\group{G}$ and $\bim_1=\{\bee_{10},
\bee_{11}\}\cdot\group{G}$.

Let us denote $E_i=\{\bee_{i0}, \bee_{i1}\}$ for $i\in\{0, 1\}$.
We will also denote $E_\emptyset=\{\emptyset\}$ and
\[E_{i_1i_2\ldots i_n}=E_{i_1}E_{i_2}\ldots E_{i_n}\subset\{\bee_{00},
\bee_{01}, \bee_{10}, \bee_{11}\}^n.\]

Then for every sequence $w=i_1i_2\ldots\{0, 1\}^\omega$ the subtree
\[T_w=\bigcup_{n\ge 0}E_{i_1i_2\ldots i_n}\]
of the tree $T=\{\bee_{00}, \bee_{01}, \bee_{10}, \bee_{11}\}^*$
is invariant under the action of the group $\group{G}$.

Let us identify $T_w$, for $w=x_1x_2\ldots$, with the binary tree $\{0, 1\}^*$ by the
isomorphism
\[\bee_{x_1i_1}\bee_{x_2i_2}\ldots\bee_{x_ni_n}\mapsto i_1i_2\ldots
i_n.\]
Denote by $\group{G}_w$ the restriction of the action of $\group{G}$
onto the subtree $T_w$, seen as an automorphism group of the binary
tree.

Then it follows directly from the wreath recursion for $\group{G}$
that the group $\group{G}_w$ is generated by automorphisms $\alpha_w,
\beta_w, \gamma_w$ (images of $\alpha, \beta, \gamma$)
which are defined by the following recursions.
\[\alpha_w=\sigma,\quad\gamma_w=(\gamma_{\overline{w}},
\beta_{\overline{w}}),\] and
\[\beta_w=\left\{\begin{array}{ll}(1, \alpha_{\overline{w}}) &
\text{if $x_1=0$,}\\
(\alpha_{\overline{w}}, 1) & \text{if $x_1=1$.}\end{array}\right.
\]
Here $\overline w=x_2x_3\ldots$ is the shift of $w$.

The group $\group{G}$ is the universal group of the family
$\{\group{G}_w\;:\;w\in\{0, 1\}^\omega\}$, i.e.,
$\group{G}$ is the quotient of the free group
$\langle\alpha, \beta, \gamma\;|\;\emptyset\rangle$ by the normal subgroup
$R=\bigcap_{w\in\{0, 1\}^\omega}R_w$, where $R_w$ is the kernel of the
natural epimorphism $\alpha\mapsto\alpha_w$, $\beta\mapsto\beta_w$,
$\gamma\mapsto\gamma_w$ of the free group
$\langle\alpha, \beta, \gamma\;|\;\emptyset\rangle$ onto the group $\group{G}_w$.
This follows from the fact that the subtrees $T_w$ cover the tree $T$.

\begin{proposition}
The group $\group{G}$ is contracting with the nucleus $\nuke=\{1, \alpha, \beta,
\gamma\}$.
\end{proposition}

\begin{proof}
By~\cite[Lemma~2.11.2]{nek:book} we have to show that sections of
$\nuke\cdot\{\alpha, \beta,
\gamma\}$ eventually belong to $\nuke$. But sections of the elements
$\alpha, \beta$ in words of length more than one are trivial, while
$\gamma$ is of order two. Hence, sections of the elements of
$\nuke\cdot\{\alpha, \beta, \gamma\}$ in words of length two belong to
$\nuke$.
\end{proof}

The following proposition is a direct corollary of the wreath
recursion defining the groups $\overline{\group{G}}$ and $\group{K}$.

\begin{proposition}
\label{pr:groupK}
Denote by $\group{K}$ the self-similar group generated by
\[\wt a=\sigma,\quad \wt b=(\wt a, \wt a),\quad \wt c=(\wt b, \wt c),\]
where $\sigma$ is the transposition. The map $g\mapsto\wt g$ from
$\overline{\group{G}}$ to $\group{K}$ defined
by
\[a\mapsto\wt a,\quad
b\mapsto\wt b,\quad c\mapsto\wt c\]
and $g\mapsto 1$ for $g\in\group{G}$ extends to an epimorphism
$\overline{\group{G}}\arr\group{K}$. Together with the map
$\bee_{ij}\mapsto j$ it generates an epimorphism of bisets.

The image of the subtree $T_w$ under the action of an element
$h\in\overline{\group{G}}$ is the subtree $T_{\wt h(w)}$.
\end{proposition}

\begin{proposition}
\label{pr:kernel}
For any element $g$ of the kernel of the epimorphism
$\overline{\group{G}}\arr\group{K}$ there exists $n$ such that
$g|_v\in\group{G}$ for all words $v$ of length greater than $n$.
\end{proposition}

\begin{proof}
It is easy to check that the wreath recursion
\[\wt a=\sigma,\quad \wt b=(\wt a, \wt a),\quad \wt c=(\wt b, \wt c)\]
is contracting on the group given by the presentation $\wt{\group{K}}=\langle\wt a,
\wt b, \wt c\;|\; (\wt a)^2=(\wt b)^2=(\wt c)^2=1\rangle$ with the
nucleus $\{1, \wt a, \wt b, \wt c\}$.

It follows from~\cite[Proposition~2.13.2]{nek:book} that a product $\wt g$ of the
generators $\wt a, \wt b, \wt c$ is trivial in $\group{K}$ if and only
if there exists $n$ such that $\wt g$ belongs to the kernel of the $n$th
iterate of the wreath recursion on $\wt{\group{K}}$. Let $g$ be a
product of the generators of $\overline{\group{K}}$ and let $\wt g$ be
the word obtained from $g$ by removing all generators $\alpha, \beta,
\gamma$ and applying the homomorphism $h\mapsto\wt h$ to every
letter $a, b, c$. Then the word $\wt g$ represents a trivial element
of $\overline{K}$. Let $n$ by such that $\wt g$ belongs to the kernel
of the $n$th iterate of the wreath recursion on $\wt{\group{K}}$. Then
it follows from the wreath recursion defining $\overline{\group{G}}$ and normality of
$\group{G}$ in $\overline{\group{G}}$ that the sections of $g$ in all
words of length $v$ belong to $\group{G}$.
\end{proof}

It follows directly from the interpretation of the generators $S$ and
$P$ of $\img{F}$ that the image of the subgroup
$\img{F}<\overline{\group{G}}$ in the quotient $\group{K}$ is
isomorphic the iterated monodromy group of $f(p)=\left(\frac{p-1}{p+1}\right)^2$.
Consequently, $\img{f}$ is isomorphic to the self-similar group generated by
\begin{equation}\label{eq:imgf}
S=\sigma(P, S^{-1}),\qquad P=\sigma.
\end{equation}
The corresponding epimorphism of the self-similarity
bisets acts by the rule $\bel_i\mapsto\bel$ and
$\ber_i\mapsto\ber$, where $(\bel, \ber)$ is the ordered basis
associated with the above recursion for $\img{f}$.

\begin{proposition}
\label{pr:quotSP}
The transformation $\kappa$ of
the space $\{\bel, \ber\}^{-\omega}$ changing in every sequence
$w\in\{\bel, \ber\}^{-\omega}$ each letter $\bel$ to $\ber$ and vice
versa induces a homeomorphism of the limit space of $\img{f}$, corresponding
to the complex conjugation on the Julia set of $f$.
\end{proposition}

\begin{proof}
Let us compute the iterated monodromy group $\img{f}$ directly, in
order to understand the geometric meaning of the elements $\bel$ and $\ber$.

The post-critical set of $f$ is
$\{\infty, 0, 1\}$. Take $-1$ as the basepoint. Let $S$ and $P$ be
the loops going in the positive direction around $0$ and around
both $0$ and $1$, respectively, as it is shown on the top part
of Figure~\ref{fig:img}. Connect the basepoint $1$ with its
preimages $\pm i$ by straight segments.

\begin{figure}[h]
\centering
\includegraphics{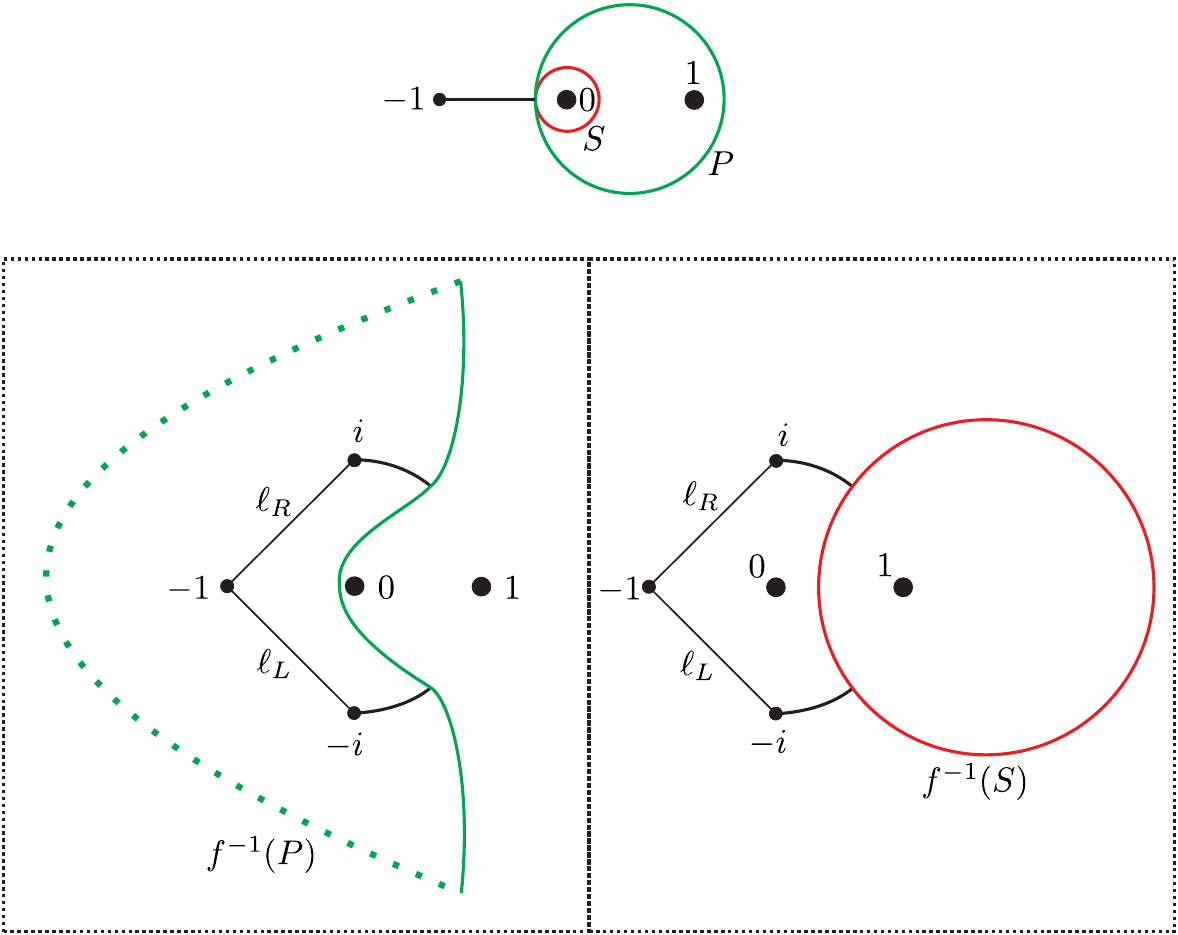}
\caption{Computation of $\img{f}$} \label{fig:img}
\end{figure}

The bottom part of
Figure~\ref{fig:img} shows the inverse images of the generators
under the action of the rational function. We see that if we label
the path connecting the basepoint $-1$ to $i$ by $\ell_{\ber}$ and
the path connecting $-1$ to $-i$ by $\ell_{\bel}$, then the associated
biset is defined by
\[P\cdot\bel=\ber,\quad P\cdot\ber=\bel,\]
and
\[S\cdot\ber=\bel\cdot
P,\quad S\cdot\bel=\ber\cdot S^{-1},\] which agrees with the wreath
recursion~\eqref{eq:imgf}.

Recall that a sequence $\ldots X^{(1)}X^{(0)}\in\{\bel,
\ber\}^{-\omega}$ represents the point of the Julia set equal to
the limit of the path $\ell_0\ell_1\cdots,$ where $\ell_k$ is a
continuation of the path $\ell_{k-1}$ and is a lift of the path
$\ell_{X^{(k)}}$ by the $k$th iteration of the rational function.
Since the rational function $\left(\frac{p-1}{p+1}\right)^2$ has
real coefficients, the basepoint $-1$ is real, and complex
conjugation permutes the paths $\ell_{\ber}$ and $\ell_{\bel}$,
the transformation $\kappa$ maps a sequence corresponding to a
point $z$ to the sequence corresponding to the conjugate point
$\overline z$.
\end{proof}

Recall that $J_1(q)$ denotes the intersection of the Julia set $J_1$ with
the $z$-line $p=q$.

\begin{proposition}
\label{pr:slices}
Each connected component of the limit space $\lims[\group{G}]$ of the group $\group{G}$
consists of points represented by the sequences of the form $\ldots
X^{(2)}_{i_2}X^{(1)}_{i_1}$, where $w=\ldots X^{(2)}X^{(1)}\in\{\ber,
\bel\}^{-\omega}$ is fixed and $i_k\in\{0, 1\}$ are arbitrary.
The connected component corresponding to $w\in\{\ber,
\bel\}^{-\omega}$ is homeomorphic $J_1(p_0)$, where $p_0$ is the
point of the Julia set of $f$ encoded by the sequence $w$.
\end{proposition}

\begin{proof}
Since $\group{G}$ is a self-similar subgroup of $\img{F}$, the
equivalence relation associated with $\group{G}$ is a sub-relation of
the asymptotic equivalence relation associated with $\img{F}$. The group
$\group{G}$ changes only the indices of the symbols $X^{(k)}_{i_k}$,
hence $\group{G}$-equivalent sequences are of the form  $\ldots
X^{(2)}_{i_2}X^{(1)}_{i_1}$, $\ldots
X^{(2)}_{j_2}X^{(1)}_{j_1}$ for some $i_k, j_k\in\{0, 1\}$ and
$X^{(k)}\in\{\bel, \ber\}$. Since $\group{G}$ is level-transitive on each of the
subtrees, the image of the set $\{\ldots
X^{(2)}_{i_2}X^{(1)}_{i_1}\;:\;\ldots i_2i_1\in\{0, 1\}^{-\omega}\}$
in the limit space of $\group{G}$ is connected (by the argument
similar to that of~\cite[Section~3.5]{nek:book}), hence is a connected
component.

It remains to show that the equivalence relation associated with
$\img{F}$ restricted to the set $\{\ldots
X^{(2)}_{i_2}X^{(1)}_{i_1}\;:\;\ldots i_2i_1\in\{0, 1\}^{-\omega}\}$
coincides with the restriction of the equivalence associated with
$\group{G}$, i.e., that the group $\img{F}$ does not
introduce new identifications inside the connected components of the
limit space of $\group{G}$. Suppose that the sequences
$\ldots X^{(2)}_{i_2}X^{(1)}_{i_1}$, $\ldots
X^{(2)}_{j_2}X^{(1)}_{j_1}$ are equivalent with respect to the action
of $\img{F}$. It means that there exists a sequence $g_k$ of elements
of $\img{F}$ assuming a finite set of values and such that
$g_k\cdot X^{(k)}_{i_k}=X^{(k)}_{j_k}\cdot g_{k-1}$ for all $k\ge
1$. The limit space of $\img{f}$ has no singular points, since the
rational function $f(p)=\left(\frac{p-1}{p+1}\right)^2$ is
hyperbolic. It follows that the images of $g_k$ in $\group{K}$ are
trivial. But this implies, by Proposition~\ref{pr:kernel} that the
elements $g_k$ belong to $\group{G}$, i.e., that the sequences are
equivalent with respect to the action of the group $\group{G}$.
\end{proof}

\subsection{The Schreier graphs of the groups $\group{G}_w$}
\label{ss:schreiergraphs}

Recall that for $v=i_1i_2\ldots i_n\in\{0, 1\}^n$, we denoted by $E_v$ the set of words
of the form $\bee_{i_1j_1}\bee_{i_2j_2}\ldots\bee_{i_nj_n}$, where
$j_1j_2\ldots j_n\in\{0, 1\}^n$. We will denote the word
$\bee_{i_1j_1}\bee_{i_2j_2}\ldots\bee_{i_nj_n}$ just
$j_1j_2\ldots j_n$, for simplicity of notation. This notation agrees
with the interpretation of $E_v$ as the $n$th level of the tree on
which the group $\group{G}_w$ acts.

Let $v\in\{0, 1\}^*$. Denote by $\Gamma_v$ the Schreier graph of the
action of $\group{G}$ on the set $E_v$, i.e., the graph with the set of
vertices $E_v$ in which to elements $w_1, w_2\in X_v$ are adjacent if
and only if $g(w_1)=w_2$ for some $g\in\{\alpha, \beta, \gamma\}$. We
label the corresponding edge of $\Gamma_v$ by $g$.

Note that the graph $\Gamma_v$ is the Schreier graph of the action of
the group $\group{G}_w$ on the $n$th level of the tree, where $w$ is any
infinite word starting with $v$, and $n$ is the length of $v$.

Note that for every generator $g\in\{\alpha, \beta, \gamma\}$ of $\group{G}$
there exists a unique word $z_{g, v}\in E_v$ of length $n$ and a
generator $h\in\{\alpha, \beta, \gamma\}$ such that $h|_{z_{g,
v}}=g$. For the remaining pairs $h\in\{\alpha, \beta, \gamma\}$
and $u\in E_v$ we have $h|_u=1$.

Namely, we have, for $v\in\{0, 1\}^n$:
\[z_{\alpha, v}=\underbrace{00\ldots 0}_{\text{$n-2$ times}}1x',\quad
z_{\beta, v}=\underbrace{00\ldots 0}_{\text{$n-1$ times}}1, \quad
z_{\gamma, v}=\underbrace{00\ldots 0}_{\text{$n$ times}},\] where
$x'=1-x$ is the letter different from the last letter $x$ of $v$.
If the word $v$ has length less than $2$, one has to take the
endings of length $|v|$ in the right hand sides of the equalities.

\begin{proposition}
Let $v, u\in\{0, 1\}^*$ be arbitrary finite words. Consider for each
word $w\in\{0, 1\}^{|v|}$ a copy
$\Gamma_{u, w}$ of the edge-labeled graph $\Gamma_u$. Connect, for
each $g\in\{\alpha, \beta, \gamma\}$ and $w\in\{0, 1\}^{|v|}$ the
copy of $z_{g, u}$ in $\Gamma_{u, w}$ with the copy of $z_{g, u}$
in $\Gamma_{u, g(w)}$ by an edge labeled by the element
$h\in\{\alpha, \beta, \gamma\}$ such that $h|_{z_{g, u}}=g$. The
obtained graph is isomorphic to $\Gamma_{uv}$.

In this graph the vertex $z_{g, uv}$ is the copy of $z_{h, u}$ in
$\Gamma_{u, w}$ for $h\in\{\alpha, \beta, \gamma\}$ and $w\in\{0,
1\}^{|v|}$ such that $h|_w=g$.
\end{proposition}

Note that the copies $\Gamma_{u, w}$ of $\Gamma_u$ are connected in
$\Gamma_{uv}$ in the same way as the vertices $w$ are connected in the
graph $\Gamma_v$.

\begin{proof}
Let $w_1w_2\in X_{uv}$ and $|w_1|=|u|, |w_2|=|v|$. It follows from the
definition of the words $z_{\alpha, u}, z_{\beta, u}, z_{\gamma, u}$
that a generator $g\in\{\alpha, \beta, \gamma\}$ changes the end of
length $|w_2|$ in the word $w_1w_2$ only when $w_1=z_{h, u}$ for some
$h\in\{\alpha, \beta, \gamma\}$, and then we have
$g(w_1w_2)=w_1h(w_2)$. In all the other cases $g(w_1w_2)=g(w_1)w_2$,
since $g|_{w_1}=1$.
\end{proof}

In the case $|v|=1$ we get the following inductive rule of
constructing the graphs $\Gamma_u$.

\begin{corollary}
\label{cor:rule1} In order to get $\Gamma_{ux}$ one has to
take two copies $\Gamma_u^{(0)}$ and $\Gamma_u^{(1)}$ of $\Gamma_u$
and connect by an edge the copies of
the vertices $z_{\alpha, u}$. The obtained graph is $\Gamma_{ux}$.
The vertex $z_{\alpha, ux}$ is the copy of $z_{\beta, u}$ in
$\Gamma_u^{(1-x)}$. The vertex $z_{\beta, ux}$ is the copy of
$z_{\gamma, u}$ in $\Gamma_u^{(1)}$. The vertex $z_{\gamma, ux}$
is the copy of $z_{\gamma, u}$ in $\Gamma_u^{(0)}$.
\end{corollary}

In the opposite case (when $|u|=1$) we get the
following rule.

\begin{corollary}
\label{cor:rule2} In order to get $\Gamma_{xv}$ one has to replace
in $\Gamma_v$ each vertex $w$ by a pair of vertices $0w$ and $1w$,
connected by an edge (labeled by $\alpha$), connect $(1-x)w$ to
$(1-x)\alpha(w)$ by an edge (labeled by $\beta$), connect $0w$ to
$0\gamma(w)$ and $1w$ to $1\beta(w)$ by edges labeled by
$\gamma$.
\end{corollary}

We get nice pictures of the graphs $\Gamma_v$ when we draw
the edges labeled by $\alpha$, $\beta$, and $\gamma$ in such a way
that they have equal length and for every vertex $w$ the edge
labeled by $\beta$ incident with $w$ (if it exists) is obtained
from the edge labeled by $\alpha$ by rotation by $\pi/2$ around
$w$, while the edge labeled by $\gamma$ is obtained from the edge
labeled by $\alpha$ by rotation by $-\pi/2$. We can also use the
opposite agreement (changing the signs of $\pi/2$ and $-\pi/2$).
See, for instance, Figure~\ref{fig:gammav}, where some graphs
$\Gamma_v$ are constructed in this way. Figure~\ref{fig:schr2} shows different graphs $\Gamma_v$ for $|v|=6$.

\begin{figure}[h]
\centering
\includegraphics{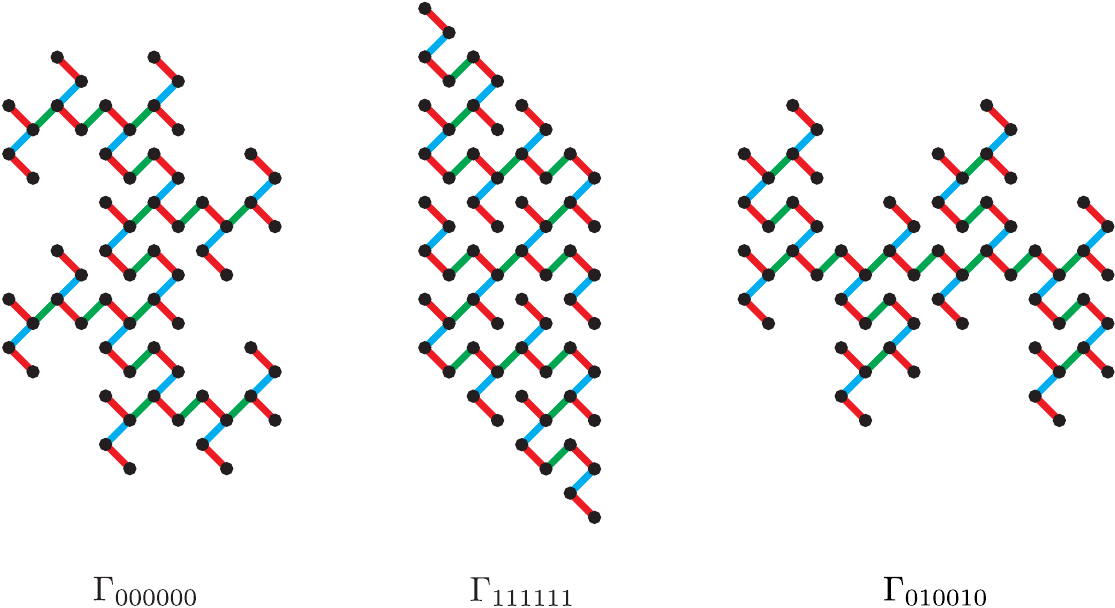}
\caption{Graphs $\Gamma_v$}\label{fig:gammav}
\end{figure}

\begin{figure}[h]
\centering
\includegraphics{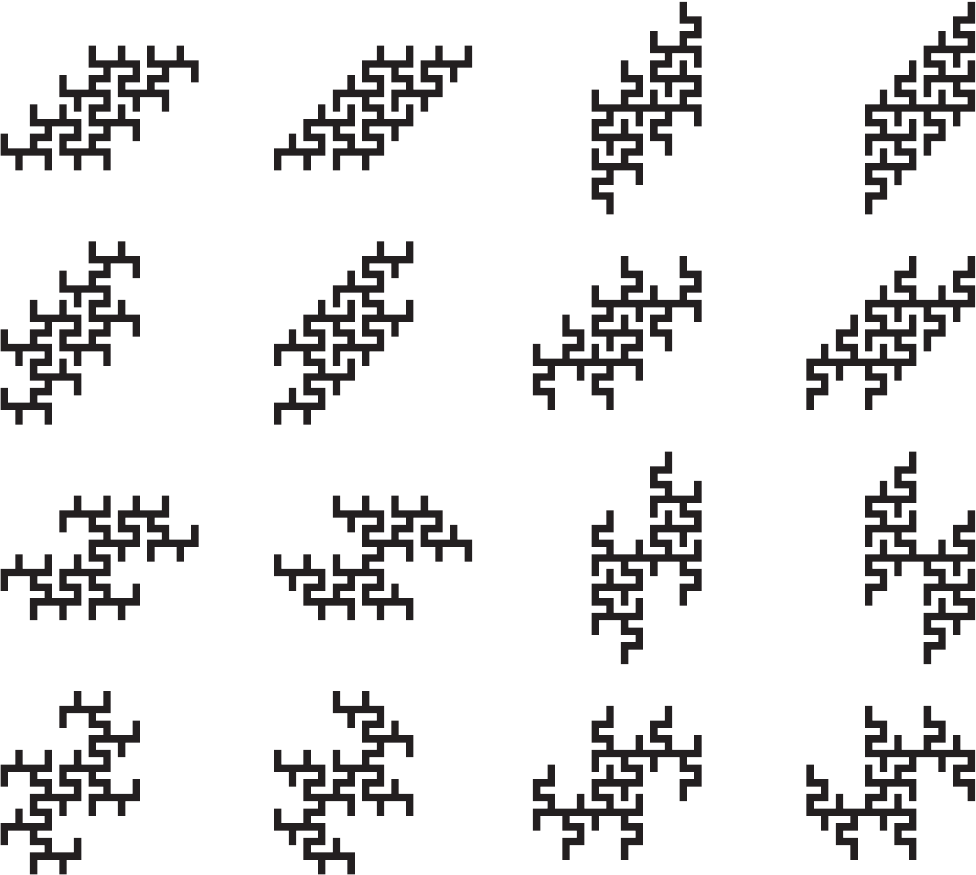}
\caption{Graphs $\Gamma_v$ for $|v|=6$}\label{fig:schr2}
\end{figure}

The rule from Corollary~\ref{cor:rule2} is shown then on
Figure~\ref{fig:rule2}.
Note that in the transition from $\Gamma_v$ to $\Gamma_{1v}$ the
relative position of the edges labeled by $\alpha$ (connecting
$0w$ with $1w$), $\beta$ (connecting $0w$ with $0\alpha(w)$) and
$\gamma$ (connecting $0w$ with $0\gamma(w)$ and $1w$ with
$1\beta(w)$) is inverted. This can be corrected by taking mirror image
of $\Gamma_{1v}$.

\begin{figure}[h]
\centering
\includegraphics{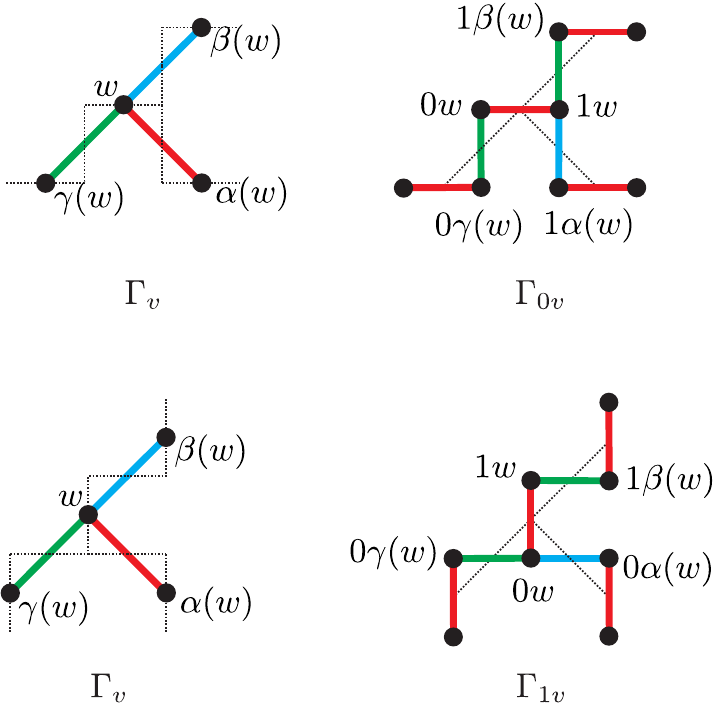}
\caption{Inductive construction of $\Gamma_v$}\label{fig:rule2}
\end{figure}

The inductive rule, shown on Figure~\ref{fig:rule2} can be used to
prove many properties of the graphs $\Gamma_v$, but we will use a
more unified approach later.

\subsection{External angles}

The group generated by the binary adding machine (odometer) $\tau=\sigma(1, \tau)$ is
the iterated monodromy group of the polynomial $z^2$. For every
quadratic polynomial $h(z)$ the loop around infinity generates a
self-similar cyclic subgroup of $\img{h}$ equivalent as a
self-similar group to the group generated by the adding machine. The
obtained embedding $\img{z^2}\arr\img{h}$ induces a surjection
from the circle (the Julia set of $z^2$) onto the Julia set of
$h$, which agrees with the dynamics (i.e., is a
semiconjugacy). This semiconjugacy coincides with the classical
\emph{Caratheodory loop}, see~\cite{milnor:dragons}, i.e., to the
extension to the boundary of the biholomorphic conjugacy from the
action of $z^2$ on the complement of the unit disc to the action of $h$ on
the complement of its filled Julia set (i.e., of the set of points that
have bounded $h$-orbits). 

By analogy with the Caratheodory loop,
let us consider the subgroup $\group{R}=\langle P, S, \tau\rangle<\img{F}$ and the
semiconjugacy of the corresponding limit spaces.

The generators of the subgroup $\group{R}$ are given by the wreath recursion
\begin{eqnarray*}
P &=& \pi,\\
S &=& \pi\sigma(P\tau^{-1}, P, S^{-1}\tau^{-1}, S^{-1}),\\
\tau &=& \sigma(1, \tau, 1, \tau).
\end{eqnarray*}

\begin{proposition}
The nucleus of the group $\group{R}=\langle P, S, \tau\rangle$ is the set
\[\nuke=\{1, S, S^{-1}, P, \tau, \tau^{-1}, S\tau, S^{-1}\tau^{-1}, P\tau, P\tau^{-1}\}.\]
\end{proposition}

\begin{proof}
It follows directly from the recursion (and the fact that $\tau$
commutes with $P$ and $S$) that $\nuke$ is a symmetric state-closed
set (a subset $A$ of a self-similar group is called
\emph{state-closed} if for every $g\in A$ and $x\in X$ we have
$g|_x\in A$).
We have to prove that the sections of the elements \[\{S, S^{-1},
\tau, \tau^{-1}, S\tau, S^{-1}\tau^{-1}\}\cdot\{S, \tau\}\] eventually belong to $\nuke$
(sections of $P$ are trivial in non-empty
words, so we do not have to consider it).
But this follows from the equalities
\begin{align*}S^2&=(PS^{-1}\tau^{-1}, PS^{-1}\tau^{-1}, S^{-1}P\tau, S^{-1}P\tau),\\
S\tau^{-1}&=\pi(P\tau^{-1}, P\tau^{-1}, S^{-1}\tau^{-1}, S^{-1}\tau^{-1}),\\
S^2\tau&=\sigma(PS^{-1}\tau^{-1}, PS^{-1}, S^{-1}P\tau, S^{-1}P),\\
\tau^2&=(\tau, \tau, \tau, \tau),\\
S\tau^2&=\pi\sigma(P, P\tau, S^{-1}, S^{-1}\tau).
\end{align*}
\end{proof}

See the Moore diagram of the nucleus on Figure~\ref{fig:moore}.

\begin{figure}[h]
\centering
\includegraphics{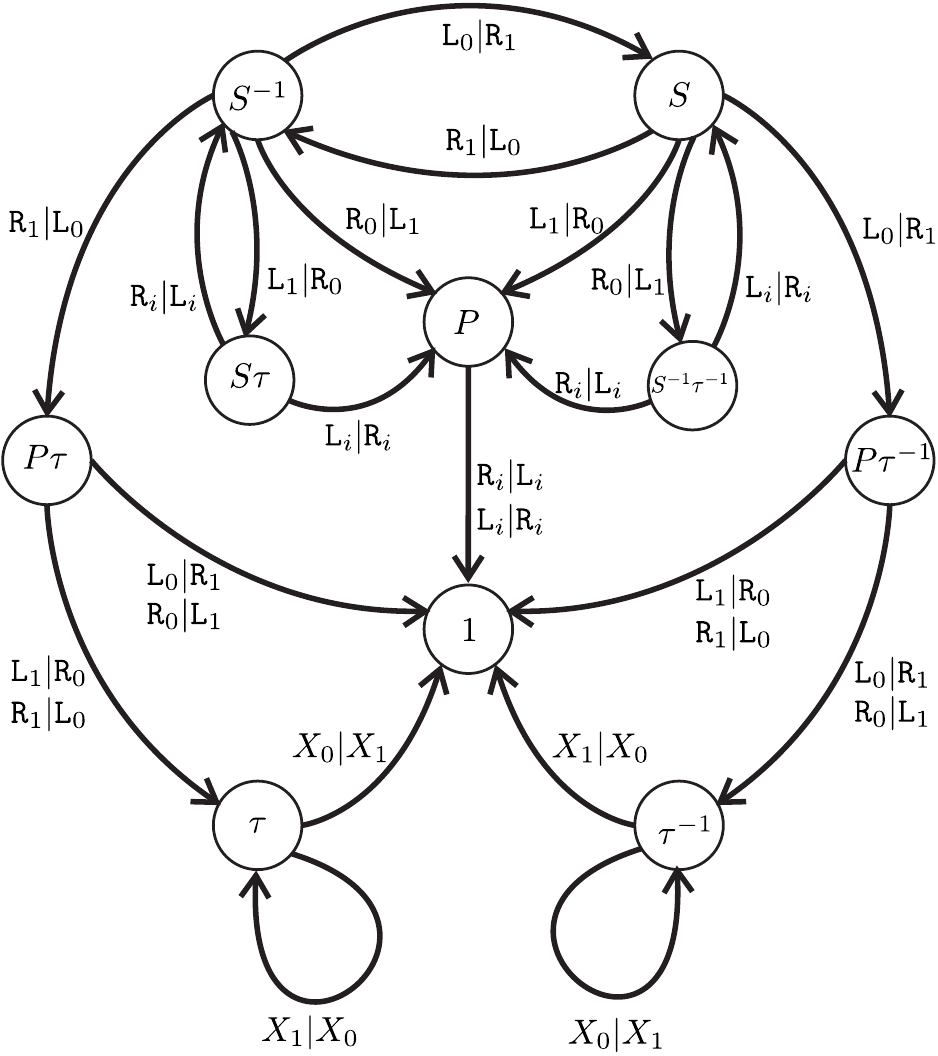}
\caption{Moore diagram of the nucleus}\label{fig:moore}
\end{figure}

Let $w=\ldots X^{(3)}_{i_3}X^{(2)}_{i_2}X^{(1)}_{i_1}$ be an element of
$\{\bel_0, \bel_1, \ber_0, \ber_1\}^{-\omega}$, where
$X^{(k)}\in\{\bel, \ber\}$ and $i_k\in\{0, 1\}$. Denote then
\[p(w)=\ldots X^{(3)}X^{(2)}X^{(1)}\in\{\bel, \ber\}^{-\omega},\]
and
\[\theta(w)=\sum_{k=1}^{\infty}\frac{i_k}{2^k}\in\R/\Z.\]

\begin{proposition}
\label{pr:asequivangles}
Two sequences are asymptotically equivalent with respect to
$\group{R}$ if and only if they are equal to sequences $w_1, w_2\in\{\bel_0, \bel_1, \ber_0,
\ber_1\}^{-\omega}$ such that one of the conditions is satisfied
\begin{enumerate}
\item \[p(w_1)=p(w_2),\quad\theta(w_1)=\theta(w_2),\]
\item 
\[p(w_1)=(\ber\bel)^{-\omega},\quad p(w_2)=(\bel\ber)^{-\omega},\quad
\theta(w_1)=\theta(w_2)+\frac 23,\]
\item 
\[p(w_1)=(\ber\bel)^{-\omega}\bel,\quad
p(w_2)=(\bel\ber)^{-\omega}\ber,\quad\theta(w_1)=\theta(w_2)+\frac 13,\]
\item there exists a non-empty word $v\in\{\bel, \ber\}^*$ such that
\[p(w_1)=(\ber\bel)^{-\omega}\bel v,\quad
p(w_2)=(\bel\ber)^{-\omega}\ber v',\quad\theta(w_1)=\theta(w_2)+\frac 1{2^{|v|}3},\]
where $v'$ is obtained from $v$ by changing the first letter.
\end{enumerate}
\end{proposition}

\begin{proof}
Note that the wreath recursion for $\tau$ and $P$ is written in terms
of the self-similarity biset as
\[\tau\cdot X_0=X_1,\quad\tau\cdot X_1=X_0\cdot\tau,\]
\[P\cdot\ber_i=\bel_i,\quad P\cdot\bel_i=\ber_i,\]
where $X$ is one of the symbols $\bel, \ber$ and $i$ is one of the
symbols $0, 1$.

Note also that $\tau S=\pi(P, P, S^{-1}, S^{-1})$, hence
\[S\cdot\bel_i=\tau^{-1}\cdot\ber_i\cdot P,\quad
S\cdot\ber_i=\tau^{-1}\cdot\bel_i\cdot S^{-1},\]
which implies
\[S^{-1}\cdot\bel_i=\tau\cdot\ber_i\cdot S,\quad
S^{-1}\cdot\ber_i=\tau\cdot\bel_i\cdot P,\]
since $S$ and $\tau$ commute.

Examining the nucleus of the group on Figure~\ref{fig:moore}, we see
that left-infinite (infinite in the past) paths labeled by $(w_1, w_2)$
in its Moore diagram belong to one of the following types:

\textbf{(I)} The path travels inside the set $\{1, \tau, \tau^{-1}\}$. In this case
the we have $p(w_1)=p(w_2)$ and $\theta(w_1)=\theta(w_2)$, and every
pair $(w_1, w_2)$ satisfying these two equalities can be obtained in
this way.

\textbf{(II)} Its vertices alternatively belongs to the sets $\{S, S\tau\}$
  and $\{S^{-1}, S^{-1}\tau^{-1}\}$.

If the last vertex of the path belongs to $\{S, S\tau\}$, then
$p(w_1)=(\ber\bel)^{-\omega}$ and $p(w_2)=(\bel\ber)^{-\omega}$.
If the last vertex belongs to $\{S^{-1}, S^{-1}\tau^{-1}\}$, then
$p(w_1)=(\bel\ber)^{-\omega}$ and $p(w_2)=(\ber\bel)^{-\omega}$, which
is symmetric with the first case.

\textbf{(III)} The last vertex of the path belongs to $\{P, P\tau,
  P\tau^{-1}\}$. Then either
 $p(w_1)=(\ber\bel)^{-\omega}\bel$ and
  $p(w_2)=(\bel\ber)^{-\omega}\ber$ (if the previous vertex belongs to $\{S,
  S\tau\}$), or $p(w_1)=(\bel\ber)^{-\omega}\ber$ and $p(w_2)=(\ber\bel)^{-\omega}\bel$
  (otherwise).

\textbf{(IV)} One of the vertices of the path (but not the last one) belong to
  $\{P, P\tau, P\tau^{-1}\}$. Then $p(w_1)=(\ber\bel)^{-\omega}\bel v$ and
  $p(w_2)=(\bel\ber)^{-\omega}\ber v'$, or
  $p(w_1)=(\bel\ber)^{-\omega}\ber v$ and
  $p(w_2)=(\ber\bel)^{-\omega}\bel v'$,
where $v'$ is obtained from $v$ by changing the first letter.

In the first case of \textbf{(II)}, if
$w_1=\ldots\ber_{i_4}\bel_{i_3}\ber_{i_2}\bel_{i_1}$ and
$w_2=\ldots\bel_{j_4}\ber_{j_3}\bel_{j_2}\ber_{j_1}$, then for any $n$
either
\[S\cdot\ber_{i_{2n}}\bel_{i_{2n-1}}\ldots\ber_{i_2}\bel_{i_1}=
\bel_{j_{2n}}\ber_{j_{2n-1}}\ldots\bel_{j_2}\ber_{j_1}\cdot S,\]
or
\[S\cdot\ber_{i_{2n}}\bel_{i_{2n-1}}\ldots\ber_{i_2}\bel_{i_1}=
\bel_{j_{2n}}\ber_{j_{2n-1}}\ldots\bel_{j_2}\ber_{j_1}\cdot S\tau,\]
or
\[\tau S\cdot\ber_{i_{2n}}\bel_{i_{2n-1}}\ldots\ber_{i_2}\bel_{i_1}=
\bel_{j_{2n}}\ber_{j_{2n-1}}\ldots\bel_{j_2}\ber_{j_1}\cdot S,\]
or
\[\tau S\cdot \ber_{i_{2n}}\bel_{i_{2n-1}}\ldots\ber_{i_2}\bel_{i_1}=
\bel_{j_{2n}}\ber_{j_{2n-1}}\ldots\bel_{j_2}\ber_{j_1}\cdot S\tau.\]

This implies that either
\[\tau^{-1}\cdot\bel_{i_{2n}}\cdot \tau\cdot\ber_{i_{2n-1}}
\ldots\tau^{-1}\cdot\bel_{i_2}\cdot
\tau\cdot \ber_{i_1}\cdot S=
\bel_{j_{2n}}\ber_{j_{2n-1}}\ldots\bel_{j_2}\ber_{j_1}\cdot S,\]
or
\[\tau^{-1}\cdot\bel_{i_{2n}}\cdot\tau\cdot\ber_{i_{2n-1}}\ldots\tau^{-1}\cdot\bel_{i_2}\cdot\tau\cdot
\ber_{i_1}\cdot S=
\bel_{j_{2n}}\ber_{j_{2n-1}}\ldots\bel_{j_2}\ber_{j_1}\cdot S\tau,\]
or
\[\bel_{i_{2n}}\cdot \tau\cdot\ber_{i_{2n-1}}\ldots\tau^{-1}\cdot
\bel_{i_2}\cdot \tau\cdot \ber_{i_1}\cdot S=
\bel_{j_{2n}}\ber_{j_{2n-1}}\ldots\bel_{j_2}\ber_{j_1}\cdot S,\]
or
\[\bel_{i_{2n}}\cdot \tau\cdot
\ber_{i_{2n-1}}\ldots\tau^{-1}\cdot\bel_{i_2}\cdot \tau\cdot
\ber_{i_1}\cdot S=
\bel_{j_{2n}}\ber_{j_{2n-1}}\ldots\bel_{j_2}\ber_{j_1}\cdot S\tau.\]

In all cases, as $n\to\infty$ we get
\[\theta(w_2)=\theta(w_1)+1/2-1/4+1/8-1/16+\cdots=\theta(w_1)+1/3\pmod{1},\]
i.e., $\theta(w_1)=\theta(w_2)+2/3$.

In the first case of \textbf{(III)} we have $w_1=\ldots
\ber_{i_5}\bel_{i_4}\ber_{i_3}\bel_{i_2}\bel_{i_1}$, and
$w_2=\ldots\bel_{j_5}\ber_{j_4}\bel_{j_3}\ber_{j_2}\ber_{j_1}$, and we
have
\[\tau^{k_1}S\cdot\ber_{i_{2n+1}}\bel_{i_{2n}}\ldots\ber_{i_3}\bel_{i_2}\bel_{i_1}=
\bel_{j_{2n+1}}\ber_{j_{2n}}\ldots\bel_{j_3}\ber_{j_2}\ber_{j_1}\cdot
P\tau^{k_2}\]
for $k_1\in\{0, 1\}$ and $k_2\in\{0, -1\}$.
Then
\[\tau^{k_1-1}\cdot\bel_{i_{2n+1}}\cdot\tau\cdot\ber_{i_{2n}}\ldots\tau^{-1}\cdot\bel_{i_3}\cdot
\tau\cdot\ber_{i_2}\cdot\tau^{-1}\cdot\ber_{i_1}\cdot
P=\bel_{j_{2n+1}}\ber_{j_{2n}}\ldots\bel_{j_3}\ber_{j_2}\ber_{j_1}\cdot
P\tau^{k_2},\]
which implies
\[\theta(w_2)=\theta(w_1)-1/2+1/4-1/8+\cdots=\theta(w_1)-1/3\pmod{1},\]
which proves case \textbf{(III)} of the proposition.

Consider now case \textbf{(IV)}. We have $w_1=\ldots
\ber_{i_5}\bel_{i_4}\ber_{i_3}\bel_{i_2}\bel_{i_1}u$ and
$w_2=\ldots\bel_{j_5}\ber_{j_4}\bel_{j_3}\ber_{i_2}\ber_{i_1}u'$ for
some $u, u'\in\{\bel_0, \bel_1, \ber_0, \ber_1\}^*$, and for every $n$ we have
\[\tau^{k_1}S\cdot\ber_{i_{2n+1}}\bel_{i_{2n}}\ldots\ber_{i_3}\bel_{i_2}\bel_{i_1}u=
\bel_{j_{2n+1}}\ber_{j_{2n}}\ldots\bel_{j_3}\ber_{j_2}\ber_{j_1}u',\]
for some $k_1\in\{0, 1\}$,
hence
\[\tau^{k_1-1}\cdot\bel_{i_{2n+1}}\cdot\tau\cdot\ber_{i_{2n}}\ldots\tau^{-1}\cdot\bel_{i_3}\cdot
\tau\cdot\ber_{i_2}\cdot\tau^{-1}\cdot\ber_{i_1}\cdot P\cdot
u=\bel_{j_{2n+1}}\ber_{j_{2n}}\ldots\bel_{j_3}\ber_{j_2}\ber_{j_1}u'.\]
This implies that
\[\theta(w_2)=\theta(w_1)+\frac{1}{2^{|u|}}\left(-\frac 12+\frac
    14-\frac 18+\cdots\right)=\theta(w_1)-\frac{1}{2^{|u|}\cdot 3},\]
which finishes the proof.
\end{proof}

Recall that the iterated monodromy group $\img{f}$ of
$f(p)=\left(\frac{p-1}{p+1}\right)^2$ is the image of the group
$\group{G}=\langle\alpha, \beta, \gamma, R, S\rangle$ under the
natural epimorphism $\overline{\group{G}}\arr\group{K}$
of self-similar groups described in Proposition~\ref{pr:groupK}. Note
that the image of the group $\group{R}$ under this epimorphism is also
$\img{f}$.
The asymptotic equivalence relation defined by $\img{f}$  is
generated by the identifications
\[(\ber\bel)^{-\omega}\sim(\bel\ber)^{-\omega},\quad
(\ber\bel)^{-\omega}\bel v\sim(\bel\ber)^{-\omega}\ber v',\] where
$v\in\{\bel, \ber\}^*$ is arbitrary and $v'$ is obtained from $v$
by changing the first letter. This follows also from
Proposition~\ref{pr:asequivangles} just by ignoring the indices,
i.e., the map $\theta$.

The natural projection $p:\{\bel_0, \bel_1, \ber_0,
\ber_1\}^{-\omega}\arr\{\bel, \ber\}^{-\omega}$ agrees with the
equivalence relations defined by the group $\group{R}$
and its quotient $\img{f}$, so that $p$ induces a surjective
continuous map $\wt p:\lims[\group{R}]\arr
\lims[\img{f}]$ of the limit
spaces. The fibers of the map $\wt p$ are circles by
Proposition~\ref{pr:asequivangles}. It follows that the limit space of
$\group{R}$ can be interpreted as the bundle over the Julia set of $f$
of the Caratheodory loops around the $p$-slices $J_1(p)$ of the Julia
set of $F$.

Let us describe the limit space of the group $\img{f}$
following~\cite[Section~3.10]{nek:book}. As
the zero step approximation of the tile of the group take a
rectangle. The vertices of the rectangle (which will correspond to
the boundary points of the tile) are labeled by the sequences
$(\ber\bel)^{-\omega}\bel, (\bel\ber)^{-\omega}\ber,
(\bel\ber)^{-\omega}, (\ber\bel)^{-\omega}$ in the given cyclic
order counterclockwise. Hence, the zero step approximation of the
limit space will be the rectangle with two pairs of vertices
identified. In order to get the next approximation of the tile one
has to take two copies of the previous approximation, append
$\ber$ to end of the labels of one of them and append $\bel$ to
the labels of the other. After that one has to identify the point
labeled by $(\bel\ber)^{-\omega}\ber\ber$ with the point labeled
by $(\ber\bel)^{-\omega}\bel\bel$ and the point labeled by
$(\bel\ber)^{-\omega}\ber\bel$ with the point labeled by
$(\ber\bel)^{-\omega}\bel\ber$.

See the sixth approximation of the tile on the middle picture of
Figure~\ref{fig:pjulc}. The two pairs of the boundary points of
the tile, which are identified in the limit space, are drawn close
to each other, so that we get a picture approximating the limit
space. The left-hand side part of Figure~\ref{fig:pjulc} shows the
the identifications of the vertices of 64 rectangles made in the
process of construction the approximation of the limit space.
Compare the obtained pictures with the Julia set of the rational
function $u\mapsto\frac{u^2+1}{u^2-1}$, shown on the
right-hand side of Figure~\ref{fig:pjulc}. This rational function
is conjugate to $f:p\mapsto\left(\frac{1-p}{1+p}\right)^2$ via the
identification $p=\frac{u-1}{u+1}$.

\begin{figure}[h]
\centering
\includegraphics{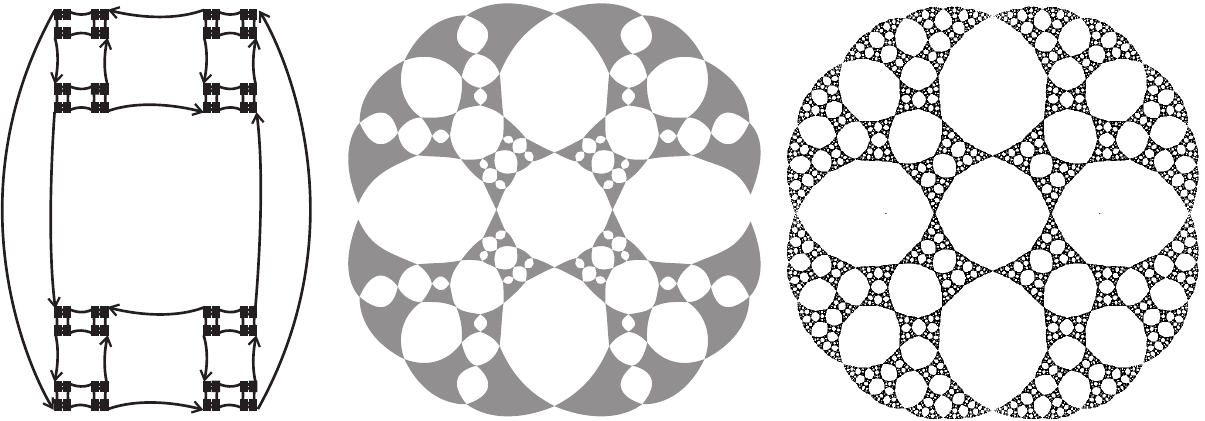}
\caption{The Julia set of $(p-1)^2/(p+1)^2$ and its combinatorial
model}\label{fig:pjulc}
\end{figure}

If we apply just the identifications (1) of
Proposition~\ref{pr:asequivangles} to the space $\{\bel_0, \bel_1,
\ber_0, \ber_1\}^{-\omega}$, i.e., if we consider the limit space
of $\langle\tau\rangle$, then we will get the direct product of
the Cantor set $\{\bel, \ber\}^{-\omega}$ with the circle $\R/\Z$.

Figure~\ref{fig:pjulc2} shows the remaining identifications producing
the limits space of $\group{R}$.
The arrows show which sequences $w\in\{\bel, \ber\}^{-\omega}$ are
identified, while the labels are the rotations applied to the
corresponding circles. Namely, if we have an arrow from $w_1$ to
$w_2$ labeled by $\theta_0$, then each point $\theta$ of the
circle above $w_1$ is identified with the point $\theta+\theta_0$
of the circle above $w_2$.
 The limit dynamical system acts on the Julia set of $f$ as
$f$ (equivalently, as the shift on $\{\bel, \ber\}^{-\omega}$), and on
the circles as the map $\theta\mapsto 2\theta$. Note that the
identifications described by Proposition~\ref{pr:asequivangles} and
Figure~\ref{fig:pjulc2} are such that the resulting map on the limit
space of $\group{R}$ is well defined.

\begin{figure}[h]
\centering
\includegraphics{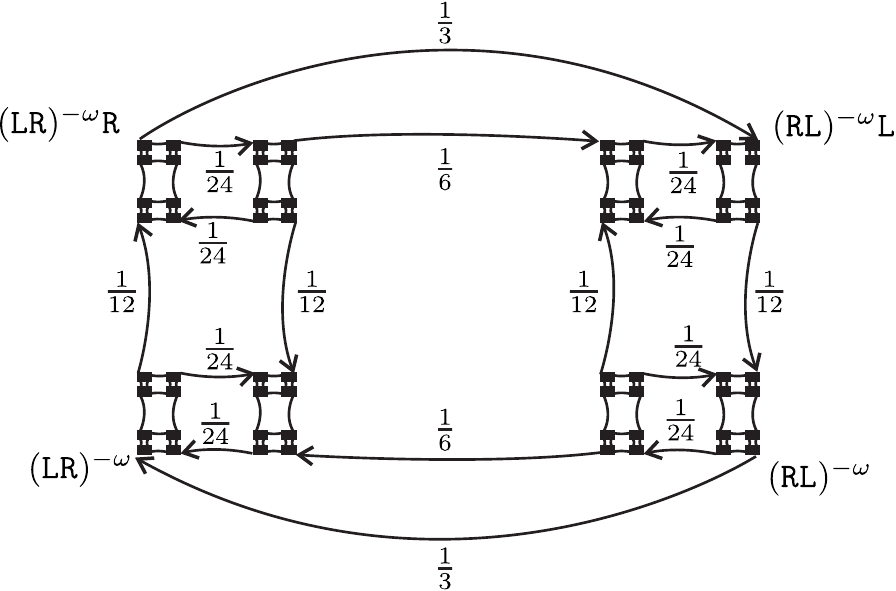}
\caption{Building the space of external angles}
\label{fig:pjulc2}
\end{figure}

The embedding $\group{R}<\img{F}$ induces a semiconjugacy of the
limit spaces \[\Phi:\lims[\group{R}]\arr \lims[\img{F}]=J_1.\] We call
points of $\lims[\group{R}]$ \emph{external rays}. We say that an
external ray $\zeta\in\lims[\group{R}]$ \emph{lands} on $(z, p)\in
J_1$ if $\Phi(\zeta)=(z, p)$. Points of $\lims[\group{R}]$ are encoded
by pairs $(\theta, w)$, where $\theta\in\R/\Z$ and $w\in\{\bel,
\ber\}^{-\omega}$ is a sequence representing a point $p\in J_1$. The
coordinate $\theta$ is called the \emph{angle} of the external
ray. Note that angle of an external ray may be not uniquely defined,
since a point of the Julia set of $f$ may be represented by different
sequences. On the other hand, difference between angles of two
external rays above the same point of the Julia set of $f$ is well
defined, since two circles are pasted to each other (in
Proposition~\ref{pr:asequivangles}) using a rotation.

\begin{proposition}
\label{pr:externalraysonpp}
Denote by $q_1$ the fixed point $\approx 0.2956$ of
$f(p)=\left(\frac{1-p}{1+p}\right)^2$. If $p_0$ belongs to the
backward orbit $\bigcup_{n\ge 0}f^{-n}(q_1)$ of $q_1$, then there are
two external rays landing on $(p_0, p_0)$. The difference of
angles of these external rays is equal to $\frac{1}{2^{k-1}3}$, where
$k$ is the smallest integer such that $f^k(p_0)=q_1$. In all
the other cases there is a unique ray landing on $(p_0, p_0)$.
\end{proposition}

Recall that the line $z=p$ is contained in $J_1$ and is an $F$-invariant subset of
the post-critical locus of $F$.

\begin{proof}
It follows from the description of the asymptotic equivalence
relation of the group $\img{f}$ that the fixed points of $f$ are
encoded in the limit space by the sequences $\ber^{-\omega},
\bel^{-\omega}$ and
$(\ber\bel)^{-\omega}\sim(\bel\ber)^{-\omega}$. The transformation
$\kappa$ permutes the first two sequences and fixes the last one.
Since $\kappa$ corresponds to complex conjugation (see
Proposition~\ref{pr:quotSP}), we conclude
that the real fixed point of $f$ is encoded by the sequences
$(\ber\bel)^{-\omega}\sim(\bel\ber)^{-\omega}$. Hence, the points of
the backward orbit of $q_1$ are the points encoded by the sequences of
the form $(\ber\bel)^{-\omega}v$, for
$v\in\{\ber, \bel\}^*$.

It follows from the dynamics on the post-critical set of $F$ that
the points of the line $z=p$ are singular with the isotropy group
a conjugate of $\langle\gamma\rangle$. The wreath recursion in
Theorem~\ref{th:imgF} implies that the points encoded by the
sequences $\ldots X^{(2)}_0X^{(1)}_0$ for $X^{(k)}\in\{\bel,
\ber\}$ have isotropy group $\langle\gamma\rangle$, hence these
sequences encode the points $z=p_0$, where $p_0$ is encoded by $\ldots
X^{(2)}X^{(1)}$ in the limit space of $\img{f}$.

Let us see to which sequences of the form $\ldots
X^{(2)}_{i_2}X^{(1)}_{i_1}$ the sequence $\ldots X^{(2)}_0X^{(1)}_0$
can be equivalent. By Proposition~\ref{pr:slices}, such two sequences,
if they are equivalent with respect to $\img{F}$, then they are
equivalent with respect to $\group{G}$.

The nucleus of $\group{G}$ for the wreath recursion of
Theorem~\ref{th:imgF}
is equal to \[\{\alpha, \beta, \gamma, (\alpha\beta)^{\pm 1},
\alpha^\beta, (\alpha\gamma)^{\pm 1}, \beta^\alpha,
(\beta\gamma)^{\pm 1}, \gamma^\alpha, \gamma^\beta,
\gamma^{\alpha\beta}, \tau^{\pm 1}, (\alpha\tau)^{\pm 1},
(\beta\tau)^{\pm 1}\},\]
see~\eqref{eq:nucleusimg} on page~\pageref{eq:nucleusimg}     .

We are interested in the left-infinite paths in the Moore diagram of
the nucleus with the arrows labeled by pairs of the form $(X_0, X_i)$
for $X\in\{\bel, \ber\}$ and $i\in\{0, 1\}$. Removing all the other
arrows and removing all arrows which do not belong to any
left-infinite path, we get the graph shown on Figure~\ref{fig:moore2}.

\begin{figure}[h]
\centering
\includegraphics{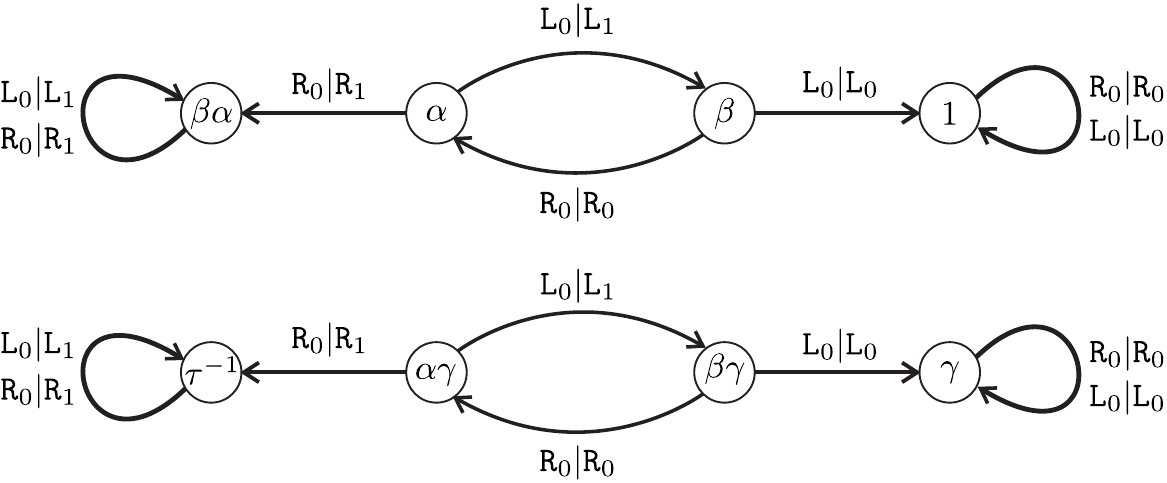}
\caption{}\label{fig:moore2}
\end{figure}

It follows that we have only the following non-trivial identifications
\[ (\bel_0\ber_0)^{-\omega}\ber_0X_0^{(n)}\ldots
X_0^{(2)}X_0^{(1)}\sim
(\bel_1\ber_0)^{-\omega}\ber_1X_1^{(n)}\ldots
X_1^{(2)}X_1^{(1)},\]
\[(\ber_0\bel_0)^{-\omega}\bel_0X_0^{(n)}\ldots
X_0^{(2)}X_0^{(1)}\sim (\ber_0\bel_1)^{-\omega}\bel_0X_0^{(n)}\ldots
X_0^{(2)}X_0^{(1)},\]
their shifts and the identification
\[\ldots X_0^{(2)}X_0^{(1)}\sim\ldots X_1^{(2)}X_1^{(1)}.\]
The last identification is trivial in terms of external angles.

It follows that the point $z=p_0$ is a landing point of one external
ray to the slice $p=p_0$ of the Julia set of $F$ except when $p_0$ is
in the backward orbit of the fixed point $q_1$, when it is a landing
point of exactly two external rays.

The remaining statements follow from
Proposition~\ref{pr:asequivangles}.
\end{proof}

Note that the points of the backward orbit of $q_1$ are
precisely the points where different Fatou components of $f$ touch
each other, i.e., the points
belonging to boundaries of two Fatou
components of $f$. This follows from the fact that the fixed point
$q_1$ belongs to the boundaries of the Fatou
components containing $0$ and $1$, and that every Fatou component of $f$ is mapped by some
iterations of $f$ onto the Fatou components containing $0$ and $1$.

See Figure~\ref{fig:rays}, where the external rays to the point $(q_1,
q_1)$ are shown.

\begin{figure}[h]
\centering
\includegraphics{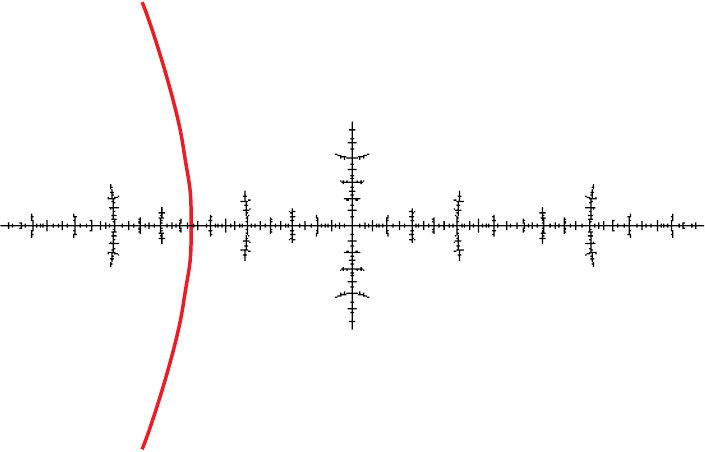}
\caption{External rays landing at $(q_1, q_1)$}
\label{fig:rays}
\end{figure}

\section{Matings}
\label{s:matings}

\subsection{An amalgam of $\group{G}$ with itself}

Consider a copy $\group{G}_1$ of the group $\group{G}$ generated by
\begin{eqnarray}
\label{eq:Ghat1}
\alpha_1 &=& \sigma(\beta_1, \beta_1, \beta_1\alpha_1,
\alpha_1\beta_1),\\
\label{eq:Ghat2}
\beta_1 &=& (1, \beta_1\alpha_1\beta_1, \alpha_1, 1),\\
\label{eq:Ghat3}
\gamma_1 &=& (\gamma_1, \beta_1, \gamma_1, \beta_1),
\end{eqnarray}
as in Theorem~\ref{th:imgF}.

Let us conjugate the right hand side of the recursion defining
$\alpha_1, \beta_1, \gamma_1$ by $\pi=(1,3)(2,4)$ (which corresponds
to changing each $\bel_i$ by $\ber_i$ and vice versa). We get then an
equivalent copy $\group{G}_2$ of $\group{G}$:
\begin{eqnarray}
\label{eq:Ghat4}
\alpha_2 &=& \sigma(\beta_2\alpha_2, \alpha_2\beta_2, \beta_2,
\beta_2),\\
\label{eq:Ghat5}
\beta_2 &=& (\alpha_2, 1, 1, \beta_2\alpha_2\beta_2),\\
\label{eq:Ghat6}
\gamma_2 &=& (\gamma_2, \beta_2, \gamma_2, \beta_2).
\end{eqnarray}

Note that
\[\gamma_1\alpha_1\beta_1=\sigma(1, \gamma_1\alpha_1\beta_1, 1, \gamma_1\alpha_1\beta_1).\]
Similarly,
\[\gamma_2\alpha_2\beta_2=\sigma(1, \gamma_2\alpha_2\beta_2, 1, \gamma_2\alpha_2\beta_2),\]
which implies that $\gamma_1\alpha_1\beta_1=\gamma_2\alpha_2\beta_2=\tau$.

Denote by $\widehat{\group{G}}$ the group generated by the set $\group{G}_1\cup\group{G}_2$.

\begin{lemma}
\label{lem:bgcommute}
The elements $\beta_1$ and $\beta_2$ act non-trivially on disjoint sets of
words and hence commute. The same is true for
$\gamma_1$ and $\gamma_2$.
\end{lemma}

\begin{proof}
The first statement follows directly from the wreath recursion. The
second statement follows from the first.
\end{proof}

Computer computation using GAP 
shows that $\widehat{\group{G}}$ is contracting with the
following nucleus of $122$ elements
\begin{multline*}\{1, \alpha_i, \beta_i, \gamma_i, \alpha_i^{\beta_i},
\gamma_i^{\alpha_i}, \gamma_i^{\beta_i}, \beta_i^{\alpha_i}, \\
\alpha_i^{\beta_j}, \beta_i^{\alpha_j},
\gamma_i^{\alpha_j}, \gamma_i^{\beta_j},
\alpha_i^{\alpha_j\beta_j}, \gamma_i^{\alpha_i\beta_i},
 \beta_i^{\alpha_j\beta_j},\\
 \alpha_i^B,  \alpha_i^C,  B, C, B^{\alpha_i}, B^{\alpha_i\beta_i},
 C^{\alpha_i}, C^{\beta_i}, C^\tau,
\alpha_i\beta_j\alpha_j\beta_j
\}\cup\\
\{\alpha_i\beta_i, \alpha_i\gamma_i, \beta_i\gamma_i, \beta_i\alpha_j,
\beta_i\gamma_j, \alpha_i\gamma_j, \\
\tau, \beta_i\alpha_i\beta_j, \beta_i\alpha_i\alpha_j,
\alpha_i\beta_i\alpha_j\beta_j, \beta_i\gamma_i\alpha_j,
\alpha_i\tau,
\beta_i\tau,
\alpha_iB,
\beta_iC,
\beta_i\alpha_iB,\\
\beta_i\tau\beta_j,
C\tau,
\alpha_iC\tau,
\beta_iC\tau\}^{\pm 1}.\end{multline*}
Here $\{i, j\}=\{0, 1\}$ and $B=\beta_1\beta_2$, $C=\gamma_1\gamma_2$.

For the definition of $J_1(p)$ see the remark before Proposition~\ref{pr:slices}.

\begin{proposition}
\label{pr:mating}
The connected components of the limit space of the group $\widehat{\group{G}}$
are obtained by taking the slices $J_1(p)$ and $J_1(\overline{p})$
of the Julia set of $F$ and gluing one to the other along the
Caratheodory loop, where the external ray landing on $(p, p)$ is identified
with the external ray landing on $(\overline p, \overline p)$. If there are two
external rays landing on $(p, p)$ (i.e., when $p$ belongs to the grand
orbit of the real fixed point $q_1$ of $f$), then the Caratheodory loops are aligned
in such a way that only one external ray landing on $(p, p)$ is identified
with the external ray landing on $(\overline p, \overline p)$.
\end{proposition}

Equivalently, the connected components of the limit space of
$\widehat{\group{G}}$ are obtained by taking two copies of $J_1(p)$,
and gluing the Caratheodory loop around one copy of $J_1(p)$ to
its mirror reflection along the diameter containing a ray landing on
$(p, p)$.

The image of the Caratheodory loop will be a curve going through every point of the
pillowcase $\C/\mathcal{H}$ induced by the embedding
$\langle\tau\rangle<\img{F}$.

\begin{proof}
It follows from Propositions~\ref{pr:slices} and~\ref{pr:quotSP} that the
connected components of the limit space of $\widehat{\group{G}}$ are
obtained by gluing together the slice $J_1(p)$ of the Julia set of $F$
with the slice $J_1(\overline{p})$. Since $\gamma_1\alpha_1\beta_1=\gamma_2\alpha_2\beta_2=\tau$, the
Caratheodory loop around $J_1(p)$ is identified with the Caratheodory
loop around $J_1(\overline{p})$ by the map induced by the map
$\ldots X^{(2)}_{i_2}X^{(1)}_{i_1}\mapsto\ldots
Y^{(2)}_{i_2}Y^{(1)}_{i_1}$ on the corresponding sets of
sequences. Here $\ldots i_2i_1\in\{0, 1\}^{-\omega}$ encodes the
points of the circle $\lims[\langle\tau\rangle]$ and $\ldots
X^{(2)}X^{(1)}=\kappa(\ldots Y^{(2)}Y^{(1)})$ is the sequence encoding
the point $p$. The identification rule of the circles of external rays
follows then from Proposition~\ref{pr:asequivangles}.
\end{proof}

Classically (see~\cite{milnor:dragons}) the identifications described in
Proposition~\ref{pr:mating} are called ``matings''. The only
difference is that in the case of the classical mating the polynomials
are monic, and the corresponding Caratheodory loops are reflected with
respect to the real axis (which corresponds to the angle 0 external
ray of a special fixed point of the polynomial). In our case we
reflect the Caratheodory loop with respect to the diameter containing
a ray landing on the points of the invariant line $(p, p)$. Since
there can be two rays landing on $(p, p)$, there are two possible
``rotated matings''.

As particular cases of components described in
Proposition~\ref{pr:mating} we get the mating of the polynomial
$h_{q_0}(z)=\left(\frac{2z}{q_0+1}-1\right)^2$, for $q_0\approx -0.6478 + 1.7214i$,
with itself (see a detailed analysis of this mating in~\cite{milnor:dragons}),
and two rotated matings of the polynomial
$h_{q_1}(z)=\left(\frac{2z}{q_1+1}-1\right)^2$, for $q_1\approx
0.2956$ with itself.

\subsection{The self-similarity biset of $\widehat{\group{G}}$}

Let $(\hat\bel_0, \hat\bel_1, \hat\ber_0, \hat\ber_1)$ be the ordered basis of the self-similarity
$\widehat{\group{G}}$-biset corresponding to the original wreath
recursion~\eqref{eq:Ghat1}--\eqref{eq:Ghat6}.

Then the $\group{G}_1$-biset $\{\hat\bel_0, \hat\bel_1, \hat\ber_0,
\hat\ber_1\}\cdot\group{G}_1$ is naturally isomorphic to the
self-similarity biset of $\group{G}$ (if we identify $\group{G}_1$ with $\group{G}$ in
the natural way). The isomorphism is given by the map
\[\hat\bel_0\mapsto\bel_0,\quad\hat\bel_1\mapsto\bel_1,\quad
\hat\ber_0\mapsto\ber_0,\quad\hat\ber_1\mapsto\ber_1,\]
where $\{\bel_0, \bel_1, \ber_0, \ber_1\}$ is the usual basis of the
self-similarity biset of $\group{G}$.

The $\group{G}_2$-biset $\{\hat\bel_0, \hat\bel_1, \hat\ber_0,
\hat\ber_1\}\cdot\group{G}_2$ is also isomorphic to the self-similarity biset
of $\group{G}$ via the mapping
\[\hat\bel_0\mapsto\ber_0,\quad\hat\bel_1\mapsto\ber_1,\quad
\hat\ber_0\mapsto\bel_0,\quad
\hat\ber_1\mapsto\bel_1.\]

The self-similarity biset of $\widehat{\group{G}}$ is a direct sum
(i.e., disjoint union) of the
biset $\bimL=\{\hat\bel_0, \hat\bel_1\}\cdot\widehat{\group{G}}$ and
$\bimR=\{\hat\ber_0, \hat\ber_1\}\cdot\widehat{\group{G}}$. Also denote for $i=1,2$
\[\bimL_i=\{\hat\bel_0, \hat\bel_1\}\cdot\group{G}_i,\qquad
\bimR_i=\{\hat\ber_0, \hat\ber_1\}\cdot\group{G}_i.\]
Let us identify $\group{G}_1$ and $\group{G}_2$ with $\group{G}$ in a
natural way, so that $\bimL_i$ and $\bimR_i$ become
$\group{G}$-bisets. Note that then $\bimL_1=\bimR_2$ and $\bimL_2=\bimR_1$.

Let $a=\pi(a, a\alpha, a, \alpha a)$, which is the element $a=\pi$
of $\overline{\group{G}}$ written with respect to the basis $\bel_0=\bee_{00}$,
$\bel_1=\bee_{01}\cdot\beta$, $\ber_0=\bee_{10}\cdot a$,
$\ber_1=\bee_{11}\cdot\beta\alpha_1a$, see
Proposition~\ref{pr:imgind2}. Then $a$ induces an
automorphism of $\group{G}$ by conjugation:
\[\alpha^a=\alpha,\quad\beta^a=\beta^\alpha,\quad\gamma^a=\gamma.\]

Let $\bim_0=\{\bee_{00}, \bee_{01}\}\cdot\group{G}$ and
$\bim_1=\{\bee_{10}, \bee_{11}\}\cdot\group{G}$ be the natural
$\group{G}$-bisets, see Subsection~\ref{ss:propertiesofgroups}.

\begin{proposition}
\label{pr:Ghatbiset}
Let $v=X^{(1)}X^{(2)}\ldots X^{(n)}\in\{\bimL, \bimR\}^n$, denote
\[x_i=\left\{\begin{array}{ll} 0 & \text{if $X^{(i)}=\bimL$,}\\ 1 &
    \text{if $X^{(i)}=\bimR$.}\end{array}\right.\]
Then the biset $X^{(1)}_1\otimes X^{(2)}_1\otimes\cdots\otimes
X^{(n)}_1$ is isomorphic to the biset
\[\bim_{x_1}\otimes\bim_{x_1+x_2}\otimes\bim_{x_2+x_3}\otimes\cdots\otimes
\bim_{x_{n-1}+x_n}\cdot a^{x_n},\]
where addition of indices is modulo two.

The $\group{G}$-biset $X^{(1)}_2\otimes X^{(2)}_2\otimes\cdots\otimes
X^{(n)}_2$ is isomorphic to the biset
\[\bim_{1+x_1}\otimes\bim_{x_1+x_2}\otimes\bim_{x_2+x_3}\otimes\cdots\otimes
\bim_{x_{n-1}+x_n}\cdot a^{1+x_n}.\]
\end{proposition}

\begin{proof}
We have $\bel_0=\bee_{00}$, $\bel_1=\bee_{01}\cdot\beta_1$, so $\bimL_1=\bim_0$.
We have $\ber_0=\bee_{10}\cdot a$ and
$\bel_1=\bee_{11}\cdot\beta\alpha a$, hence $\bimR_1$ is identified
with $\bim_1\cdot a$.
Consequently, $\bimL_2$ is isomorphic to $\bim_1\cdot a$, and  $\bimR_2$
is isomorphic to $\bim_0$.

Note that it follows from the wreath recursions defining
$\overline{\group{G}}$ that $a\cdot\bim_0\cong\bim_1$ and
$a\cdot\bim_1\cong\bim_0$.

Consequently, the biset $X^{(1)}_1\otimes X^{(2)}_1\otimes\cdots\otimes
X^{(n)}_1$ is isomorphic to the biset
$\bim_{x_1}\cdot a^{x_1}\otimes\bim_{x_2}\cdot
a^{x_2}\otimes\cdots\otimes\bim_{x_n}\cdot
a^{x_n}$, which is isomorphic to $\bim_{x_1}\otimes\bim_{x_1+x_2}\otimes\bim_{x_2+x_3}\otimes\cdots\otimes
\bim_{x_{n-1}+x_n}\cdot a^{x_n}$.

Similarly, the $\group{G}$-biset $X^{(1)}_2\otimes X^{(2)}_2\otimes\cdots\otimes
X^{(n)}_2$ is isomorphic to $\bim_{1+x_1}\cdot a^{1+x_1}\otimes\bim_{1+x_2}\cdot
a^{1+x_2}\otimes\cdots\otimes\bim_{x_n}\cdot
a^{1+x_n}$, which is isomorphic to
$\bim_{1+x_1}\otimes\bim_{x_1+x_2}\otimes\bim_{x_2+x_3}\otimes\cdots\otimes
\bim_{x_{n-1}+x_n}\cdot a^{x_n}$.
\end{proof}

\subsection{A virtually abelian subgroup of $\widehat{\group{G}}$}

Denote $A=\alpha_2$, $B=\beta_1\beta_2$, and $C=\gamma_1\gamma_2$.
We have then $\alpha_1=CAB=BAC$.

Let us pass to the basis
\[\bex_1=\hat\bel_0,\quad\bex_2=\hat\bel_1\cdot\beta_1,\quad
\bex_3=\hat\ber_0,\quad\bex_4=\hat\ber_1\cdot\beta_2,
\] i.e., conjugate the wreath recursion defining $\widehat{\group{G}}$ by $(1,
\beta_1, 1, \beta_2)$. We get then
\begin{eqnarray*}
CAB=\alpha_1 &=& \sigma(1, 1, AC, CA),\\
\beta_1 &=& (1, CAB, CAB, 1),\\
\gamma_1 &=& (\gamma_1, \beta_1, \gamma_1, \beta_1),\\
A=\alpha_2 &=& \sigma(BA, AB, 1, 1),\\
\beta_2 &=& (A, 1, 1, A),\\
\gamma_2 &=& (\gamma_2, \beta_2, \gamma_2, \beta_2).
\end{eqnarray*}

It follows that
\begin{eqnarray*}
A &=& \sigma(BA,
AB, 1, 1),\\
B &=& (A, CAB, CAB, A),\\
C &=& (C, B, C, B).
\end{eqnarray*}

\begin{proposition}
\label{pr:lattes} The subgroup $\group{H}=\langle A, B, C\rangle$
of $\widehat{\group{G}}$ is
isomorphic as a self-similar group to the group of affine transformations of $\C$ of the
form $z\mapsto \pm z+q$, where $q\in\Z[i]$. The isomorphism
identifies $A, B$ and $C$ with the affine transformations
\[z\cdot A=-z+1,\quad z\cdot B=-z+1+i,\quad z\cdot C=-z,\]
the basis of the self-similarity biset is identified with
the affine transformations
\begin{eqnarray*}
z\cdot\bex_1 &=& \frac 1{1+i}z=\frac{1-i}2z,\\
z\cdot\bex_2 &=& \frac{1}{1+i}(-z+i)=-\frac{1-i}2z+\frac{1+i}2,\\
z\cdot\bex_3 &=& \frac 1{1-i}z=\frac{1+i}2z,\\
z\cdot\bex_4 &=& \frac 1{1-i}(-z+1)=-\frac{1+i}2z+\frac{1+i}2.
\end{eqnarray*}
\end{proposition}

For identification of permutational bisets with sets of (partial)
transformations, see Subsection~\ref{ss:wreathvirtend}. 
The biset structure comes from pre- and post-composition
with the group action.

\begin{proof}
We have $\bex_4=A\cdot\bex_3$ and $A\cdot\bex_4=\bex_3$,
\[z\cdot
A\cdot\bex_1=-\frac{1-i}2z+\frac{1-i}2=-\left(-\left(-\frac{1-i}2z+
\frac{1+i}2\right)+1+i\right)+1=
\bex_2\cdot BA,\] hence $A\cdot\bex_1=\bex_2\cdot BA$ and $A\cdot\bex_2=\bex_1\cdot AB$, which
agrees with the wreath recursion.

We have
\[z\cdot B\cdot\bex_1=\frac{1-i}2(-z+1+i)=-\frac{1-i}2z+1=z\cdot\bex_1\cdot A,\]
\[z\cdot B\cdot\bex_2=-\frac{1-i}2(-z+1+i)+\frac{1+i}2=\frac{1-i}2z-\frac{1-i}2=
z\cdot\bex_2\cdot CAB,\] since $z\cdot CAB=-z+i$,
\[z\cdot B\cdot\bex_3=\frac{1+i}2(-z+1+i)=-\frac{1+i}2z+i=z\cdot\bex_3\cdot CAB,\]
and
\[z\cdot B\cdot\bex_4=-\frac{1+i}2(-z+1+i)+\frac{1+i}2=\frac{1+i}2z+\frac{1-i}2=
z\cdot\bex_4\cdot A,\] which also agrees with the wreath
recursion.

Finally, it is easy to check that $z\cdot C\cdot\bex_1=z\cdot
\bex_1\cdot C$, $z\cdot C\cdot\bex_3=z\cdot\bex_3\cdot C$ and
\[z\cdot C\cdot\bex_2=\frac{1-i}2z+\frac{1+i}2=z\cdot\bex_2\cdot B,\]
and
\[z\cdot C\cdot\bex_4=\frac{1+i}2z+\frac{1+i}2=z\cdot\bex_4\cdot B.\]
\end{proof}

\subsection{The limit dynamical system of $\group{H}$}

\begin{proposition}
\label{pr:Hspacelim}
The limit $\group{H}$-space $\limg[\group{H}]$ is homeomorphic to the direct
product of $\C$ with the Cantor set $\{\bel, \ber\}^{-\omega}$ with the
natural (right) action of $\group{H}$ on $\C$ and trivial action on $\{\bel,
\ber\}^{-\omega}$. The self-similarity structure is given by
\begin{eqnarray*}
(z, \ldots y_2y_1)\otimes\bex_1 &=& \left(\frac{1-i}2z, \ldots y_2y_1\bel\right),\\
(z, \ldots y_2y_1)\otimes\bex_2 &=& \left(-\frac{1-i}2z+\frac{1+i}2, \ldots y_2y_1\bel\right),\\
(z, \ldots y_2y_1)\otimes\bex_3 &=& \left(\frac{1+i}2z, \ldots y_2y_1\ber\right),\\
(z, \ldots y_2y_1)\otimes\bex_4 &=& \left(-\frac{1+i}2z+\frac{1+i}2, \ldots y_2y_1\ber\right).\\
\end{eqnarray*}
\end{proposition}

\begin{proof}
Direct corollary of Proposition~\ref{pr:lattes} and Theorem~\ref{th:uniquenesslimg}.
\end{proof}

The orbispace $\C/\group{H}$ is a flat surface homeomorphic to the sphere
with four singular points, which are the images of the fixed
points $1/2, (1+i)/2, 0$, and $i/2$ of the transformations $A, B,
C$ and $CAB$, respectively. Let us denote these singular points by
$Z_A, Z_B, Z_C$ and $Z_{CAB}$, respectively. A fundamental domain
$D$ of $\group{H}$ is the rectangle with the vertices $i/2, 0, 1$
and $1+i/2$.

The natural map $D\arr\C/\group{H}$ folds this rectangle along the
segment connecting $1/2$ and $(1+i)/2$ in
two, so that we get a ``pillowcase'', whose vertices are the
points $Z_A, Z_B, Z_C$ and $Z_{CAB}$, see Figure~\ref{fig:fundamental}.

\begin{figure}[h]
\centering
\includegraphics{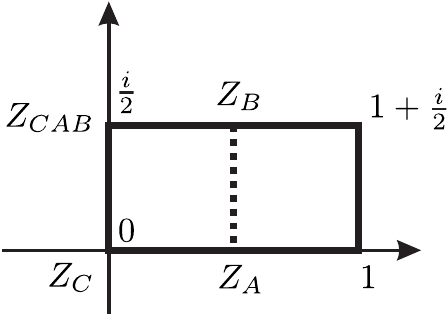}
\caption{Fundamental domain of $\group{H}$}
\label{fig:fundamental}
\end{figure}

The following is a direct corollary of the description of the limit
$\group{H}$-space given in Proposition~\ref{pr:Hspacelim}.

\begin{corollary}
The limit space $\lims[\group{H}]$ is homeomorphic to the direct product
$\C/\group{H}\times\{\bel, \ber\}^{-\omega}$. The shift map
$\si:\lims[\group{H}]\arr\lims[\group{H}]$ acts by the rule
\[\si(z, \ldots y_2y_1)=\left\{\begin{array}{ll} ((1-i)z, \ldots y_3y_2), & \text{if $y_1=\bel$},\\
((1+i)z, \ldots y_3y_2), & \text{if $y_1=\ber$.}\end{array}\right.\]
Here $z, (1-i)z$, and $(1+i)z$ are complex numbers representing the
corresponding points of the orbispace $\C/\group{H}$.
\end{corollary}

\subsection{Schreier graphs of $\group{H}$ and $\widehat{\group{G}}$}

Consider the natural (right) action of $\group{H}$ on $\C$. Take the
basepoint $\xi=\frac{1+i}4$. It has trivial stabilizer in $\group{H}$,
hence we can consider the orbit of $\xi$ as a vertex set of the left
Cayley graph $\Gamma_{\group{H}}$ of $\group{H}$ with respect to the generating set
$A, B, C, CAB$. See the Cayley graph on Figure~\ref{fig:caley}. Here
edges corresponding to the generators $A, B, C$, and $CAB$ are orange,
blue, green, and red, respectively.

\begin{figure}[h]
\centering
\includegraphics{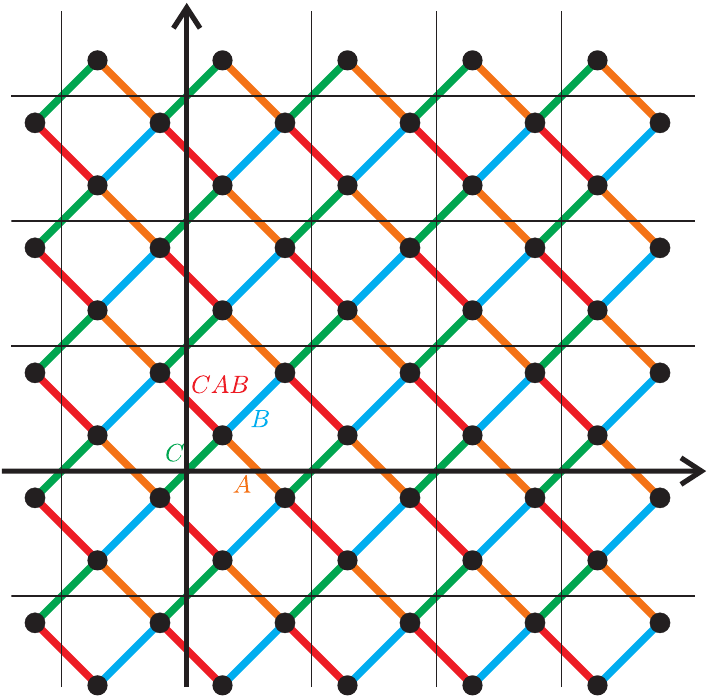}
\caption{Cayley graph of $\group{H}$}
\label{fig:caley}
\end{figure}

Denote by $\group{H}_n$ the subgroup of $\group{H}$ consisting of the affine
transformations of the form $z\mapsto \pm z+q$, where $q\in\Z[i]$
is a Gaussian integer divisible by $(1+i)^n$.

It follows then from Proposition~\ref{pr:lattes}, that the
Schreier graph of the action of $\group{H}$ on the $n$th
level of the tree consists of $2^n$ copies of the graph
$\Gamma_n(\group{H}):=\Gamma_{\group{H}}/\group{H}_n$.
A fundamental domain of $\group{H}_n$ is the rectangle with vertices
$0$, $(1+i)^n$, $i(1+i)^n/2$, and $(1+i/2)(1+i)^n$. Note that its
sides are either parallel to the real and imaginary axis (for even
$n$) or parallel to the diagonals $\Re(z)=\Im(z)$ and $\Re(z)=-\Im(z)$.

See Figure~\ref{fig:lattes} for the Schreier graphs
$\Gamma_5(\group{H})$ and $\Gamma_4(\group{H})$. Note that they have
four loops, which we will usually omit in the sequel, since we will
consider simplicial Schreier graphs only.

\begin{figure}[h]
\centering
  \includegraphics{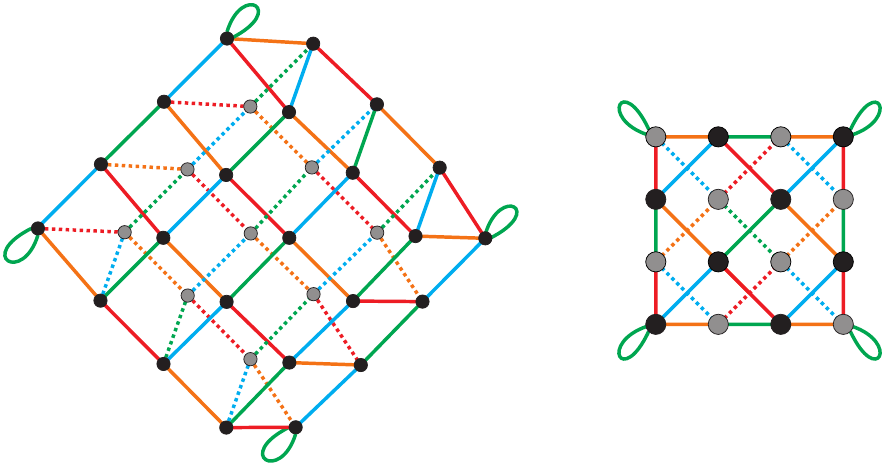}\\
  \caption{Schreier graphs $\Gamma_n(\group{H})$}\label{fig:lattes}
\end{figure}

Figure~\ref{fig:lattes2} shows a more convenient way of drawing
the Schreier graphs $\Gamma_n(\group{H})$ and their subgraphs. Here the
graphs $\Gamma_5(\group{H})$ and $\Gamma_4(\group{H})$ are drawn inside
fundamental domains of the action of $\group{H}_n$ on $\C$. In
order to get the Schreier graphs one has to fold the rectangle
into a square pillowcase (which corresponds to taking the quotient
$\C/\group{H}_n$).

\begin{figure}[h]
\centering
\includegraphics{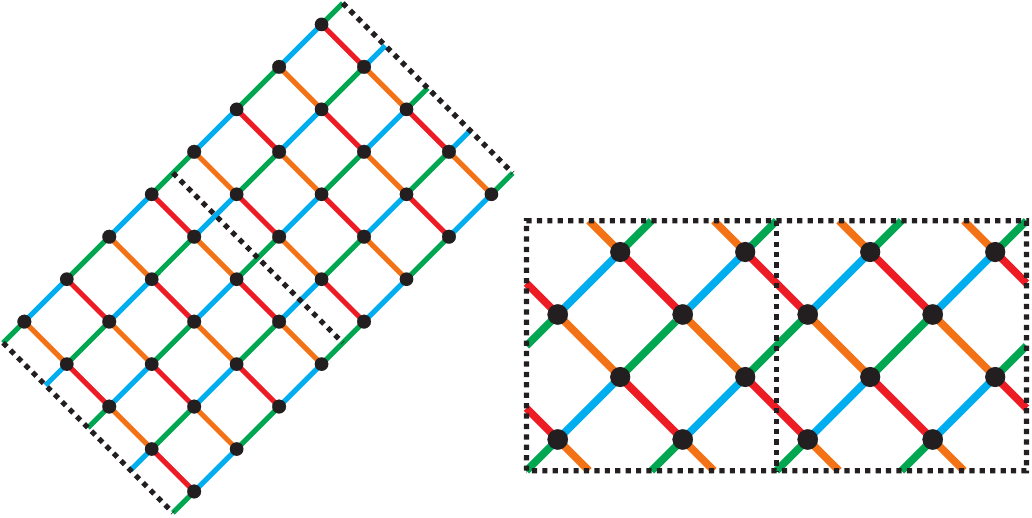}\\
\caption{Unfolded Schreier graphs
$\Gamma_{\group{H}_n}$}\label{fig:lattes2} \end{figure}

Denote for $v=X^{(1)}X^{(2)}\ldots X^{(n)}\in\{\bel, \ber\}^n$ by
$\Gamma_{1, v}$
and $\Gamma_{2, v}$ the corresponding connected components of the
Schreier graphs of the actions of $\group{G}_1$ and $\group{G}_2$ on
the $n$th level of the tree. More precisely, they are
the Schreier graphs of the actions of $\group{G}_i$
on the spaces of right orbits of the bisets
\[\{X_0^{(1)}, X_1^{(1)}\}\otimes\{X_0^{(2)},
X_1^{(2)}\}\otimes\cdots\otimes
\{X_0^{(n)}, X_1^{(n)}\}\cdot\group{G}_i.\]

The graphs $\Gamma_{1, v}$ and $\Gamma_{2, v}$ are isomorphic to the
Schreier graphs $\Gamma_{w_1}$ and $\Gamma_{w_2}$ of the group
$\group{G}$, where $w_1, w_2\in\{0, 1\}^*$ are determined by the rules
given in Proposition~\ref{pr:Ghatbiset}. Note that $w_2$ is obtained from $w_1$ by
changing the first letter.

By Lemma~\ref{lem:bgcommute}, the graphs $\Gamma_{1, v}$ and
$\Gamma_{2, v}$ have disjoint sets of edges such that their
union is the set of edges of $\Gamma_n(\group{H})$ (if we ignore the
loops). Namely, the red
edges of Figure~\ref{fig:caley} correspond to $\alpha_1$, the orange
ones to $\alpha_2$; each blue edge corresponds either to $\beta_1$ or
to $\beta_2$, each green edge either to $\gamma_1$ or to $\gamma_2$.

See Figure~\ref{fig:schrdouble} for an example of the subgraphs
$\Gamma_{1, v}$ and $\Gamma_{2, v}$ (colored red and black)
of $\Gamma_6(\group{H})$. We removed the edges corresponding
to loops in $\Gamma_6(\group{H})$.

\begin{figure}[h]
\centering
\includegraphics{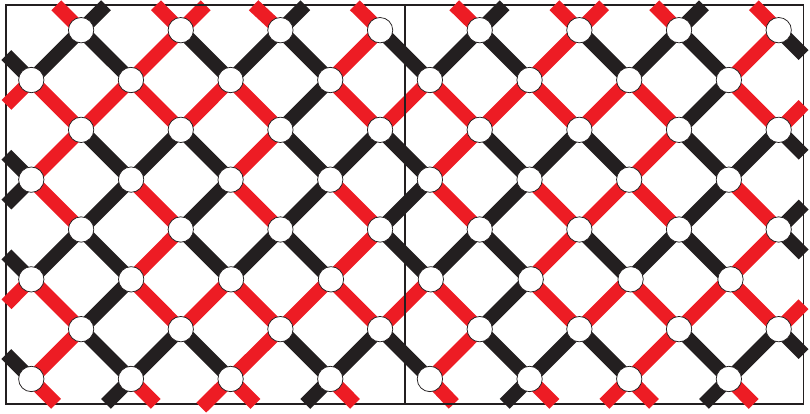}
\caption{A component of a Schreier graph of $\widehat{\group{G}}$}\label{fig:schrdouble}
\end{figure}

Note that each edge of $\Gamma_{\group{H}}$ is a diagonal of a
square with sides of length $1/2$ parallel to the real and
imaginary axes. These squares tile the plane and the pillowcases
$\C/\group{H}_n$ and each square of the tiling has precisely one diagonal
belonging to the Cayley graph $\Gamma_{\group{H}}$. By coloring the squares
containing the edges of $\Gamma_{1, v}$ and
$\Gamma_{2, v}$ in different colors (e.g., black and white), we
get a nice visualization of the partition of $\Gamma_n(\group{H})$ into
the trees $\Gamma_{1, v}$ and $\Gamma_{2, v}$, see
Figure~\ref{fig:schrshaded}. Here the squares whose
diagonals are loops of $\Gamma_n(\group{H})$ are colored blue. We will
call them \emph{singular}.

Let us denote by $K_{1, v}$ the
union of the squares whose diagonals belong to $\Gamma_{1, v}$ and by
$K_{2, v}$ the union of the squares whose diagonals
belong to $\Gamma_{2, v}$ and of the singular squares.

\begin{figure}[h]
\centering
\includegraphics{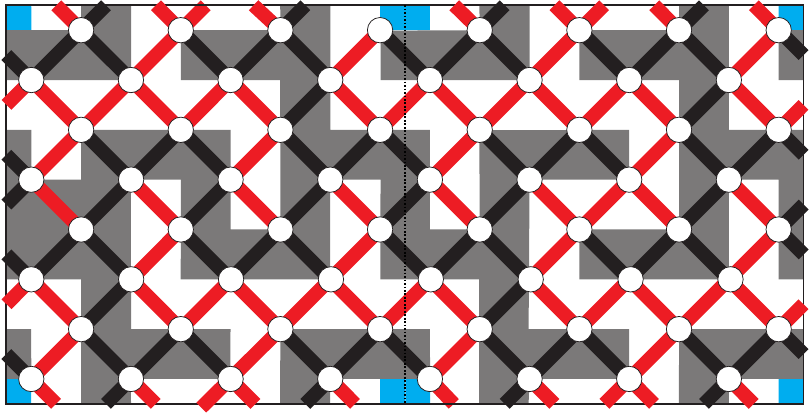}
\caption{Partition into sets $K_{i, v}$}\label{fig:schrshaded}
\end{figure}

\begin{proposition}
\label{pr:lambda}
Boundary between $K_{1, v}$ and $K_{2, v}$ is a closed broken
line $\lambda_v$ describing the action of $\tau$ on the vertex set of
$\Gamma_n(\group{H})$. Namely, for every vertex $u$ there are no vertices
of $\Gamma_n(\group{H})$ on $\lambda_v$ between $u$ and $\tau(u)$.
\end{proposition}

Note that our choice to include the singular squares in $K_{2, v}$ is not
very important. It will change only the side on which the path
$\lambda_v$ goes around the singular point of $\C/\group{H}_n$.

\begin{proof}
We have
$\tau=\gamma_1\alpha_1\beta_1=\gamma_2\alpha_2\beta_2$. Consider the
little squares of $\Gamma_n(\group{H})$. Two of their sides (opposite
to each other) correspond to $\alpha_1=CAB$ and $\alpha_2=A$, the other two
sides correspond to $B$ and $C$. Each of the latter two edges may
belong either to $\Gamma_{1, v}$ or to $\Gamma_{2, v}$.
Figure~\ref{fig:taug} shows all four
possible cases. Note that if an edge corresponding to $B$ or $C$ belongs to
$\Gamma_{i, v}$, then its endpoints are fixed under the action of
$\beta_{1-i}$, $\gamma_{1-i}$, respectively. This information makes it
possible to determine for one of the pairs of vertices of the square
that one is the image of the other under the action of $\tau$, as it
is shown by black arrows on Figure~\ref{fig:taug}. If one of the edges
of the squares is a loop of $\Gamma_n(\group{H})$, then we assume that
it belongs to $\Gamma_{2, v}$ (according to our convention about the
set $K_{2, v}$). Note that different agreement does not change the
order in which $\lambda_v$ connects the vertices of $\Gamma_n(\group{H})$.

\begin{figure}[h]
\centering
\includegraphics{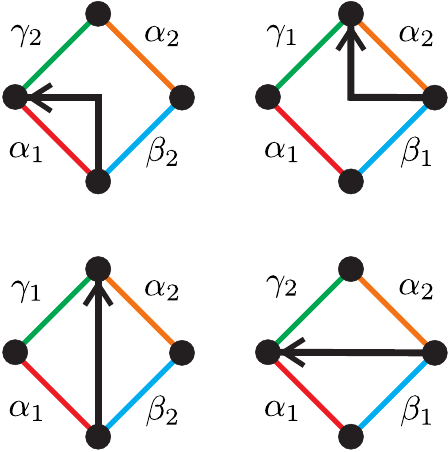}
\caption{Action of $\tau$}
\label{fig:taug}
\end{figure}

We see that the arrows describing the action of $\tau$ belong to the
boundary $\lambda_v$ between the sets $K_{1, v}$ and $K_{2, v}$.
\end{proof}

Consequently, partition of the pillowcase $\C/\group{H}_n$ into the
sets $K_{1, v}$ and $K_{2, v}$ is an approximation of the mating
described in Proposition~\ref{pr:mating}. The boundary $\lambda_v$
between the sets converges (as $v$ converges to a left-infinite sequence
$w\in\{\bel, \ber\}^{-\omega}$) to the map from the circle to
a connected component of $\lims[\widehat{\group{G}}]$
induced by the inclusion of $\langle\tau\rangle<\widehat{\group{G}}$.

See Figure~\ref{fig:pillow2s}, where two examples of the partition are
given.

\begin{figure}[h]
\centering
\includegraphics[width=5in]{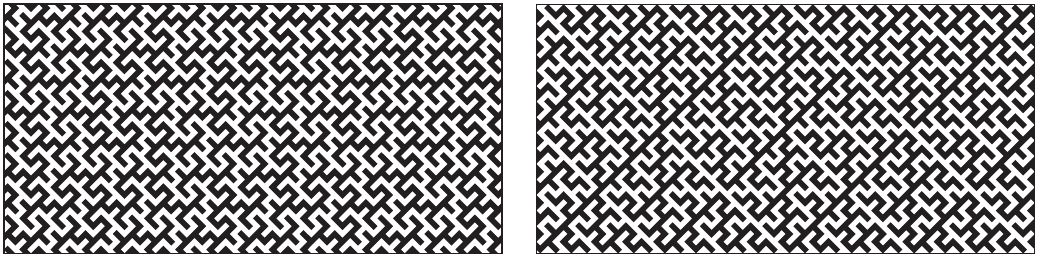}
\caption{``Pillowcase ornaments''}\label{fig:pillow2s}
\end{figure}

We orient $\lambda_v$ according to the action of $\tau$, so that
oriented segments go from $u$ to $\tau(u)$. Then $\lambda_v$ goes
around $\Gamma_{1, v}$ in positive direction, and around $\Gamma_{2,
  v}$ in the negative direction (if we orient $\C$ in the standard
way), see Figure~\ref{fig:taug}.

\subsection{Recursive rule of constructing $\Gamma_{1, v}$ and
  $\Gamma_{2, v}$}

It follows from Proposition~\ref{pr:Ghatbiset} that the graphs
$\Gamma_{1, v}$ and $\Gamma_{2, v}$ for $w=z_1z_2\ldots z_n\in\{\bel, \ber\}^n$ are isomorphic to the
Schreier graphs $\Gamma_{w_1}$ and $\Gamma_{w_2}$ of $\group{G}$, where
$w_1=x_1x_2\ldots x_n$, $w_2=y_1y_2\ldots y_n\in\{0, 1\}^n$ are
defined by the rule
\[x_k=y_k=\left\{\begin{array}{ll}0 & \text{if $z_{k-1}=z_k$,}\\
1 & \text{otherwise,} \end{array}\right.\] for $k\ge 2$, while
$x_1$ and $y_1$ are defined by the same rule with the assumption
$z_0=\bel$ and $z_0=\ber$, respectively.

Let us translate now the recursive rule from Corollary~\ref{cor:rule1}
of construction of the graphs $\Gamma_v$ in terms of the sequences
over the alphabet $\{\bel, \ber\}$ and subgraphs of the graph $\Gamma_{\group{H}}$.

Note that since the graphs $\Gamma_{i, v}$  are trees (as they are isomorphic to
$\Gamma_w$ for some $w$), they can be
lifted by the natural quotient map
$\Gamma_{\group{H}}\arr\Gamma_{\group{H}}/\group{H}_n$
to a subgraph of the Cayley graph $\Gamma_{\group{H}}$ of $\group{H}$.

On the initial step (for the empty word $v$) the graphs
$\Gamma_{i, \emptyset}$ consist of one vertex only, which is
marked by $z_{\alpha, \emptyset}, z_{\beta, \emptyset}$, and
$z_{\gamma, \emptyset}$ simultaneously. Choose a point in
$\Gamma_{\group{H}}$, which will be the lift of the graphs
$\Gamma_{\emptyset, i}$. We will add, for convenience, halves of the incident
edges of $\Gamma_{\group{H}}$, corresponding to $CAB, B, C$ for
$\Gamma_{1, \emptyset}$ and $A, B, C$ for $\Gamma_{2, \emptyset}$. Let us
denote the obtained graphs by $\Delta_{1, \emptyset}$ and
$\Delta_{2, \emptyset}$, respectively.

Suppose that we have constructed the graphs $\Delta_{1, v}$ and
$\Delta_{2, v}$, which are lifts of the graphs $\Gamma_{1, v}$
and $\Gamma_{2, v}$, respectively, with three marked vertices
$z_{\alpha, v}$, $z_{\beta, v}$, and $z_{\gamma, v}$ and halves of some
edges of $\Gamma_{\group{H}}$ attached to the marked vertices. Let $z_{\alpha,
  v}'$, $z_{\beta, v}'$, and $z_{\gamma, v}'$ be the other
(``hanging'') vertices of the half-edges.

Then the graphs $\Delta_{i, vx}$ for $i\in\{1, 2\}$, $x\in\{\bel,
\ber\}$ together with the marking are obtained by the following rule.

\medskip
{\it
Denote $\Delta_{i, v, 0}=\Delta_{i, v}$ and let $\Delta_{i, v, 1}$
be $\Delta_{i, v}$ rotated by $180^\circ$ around $z_{\alpha, v}'$.
Take the union of $\Delta_{i, v, 0}$ with the $\Delta_{i, v, 1}$,
connecting in this way the respective copies of $z_{\alpha, v}$ by
an edge.

The copy of $z_{\gamma, v}$ in $\Delta_{i, v, 0}$ is the vertex
$z_{\gamma, vx}$. The copy of $z_{\gamma, v}$ in $\Delta_{i, v, 1}$ is
the vertex $z_{\beta, vx}$. If the last letter of $v$ coincides with
$x$ (or if $v=\emptyset$, $x=\bel$, $i=1$, or $v=\emptyset, x=\ber, i=2$),
then the copy of $z_{\beta, v}$ in $\Delta_{i, v, 1}$ is
$z_{\alpha, vx}$, otherwise $z_{\alpha, vx}$ is the copy of $z_{\beta,
  v}$ in $\Delta_{i, v, 0}$. Remove the half-edge attached to the
other (unmarked) copy of $z_{\beta, v}$. The obtained graph is
$\Delta_{i, vx}$. The graph $\Gamma_{i, vx}$ is obtained from it
by removing the three half-edges attached to the marked vertices.
}
\medskip

See the first three steps of this recursion (for $i=1$) on
Figure~\ref{fig:LRL}.

\begin{figure}[h]
\centering
\includegraphics{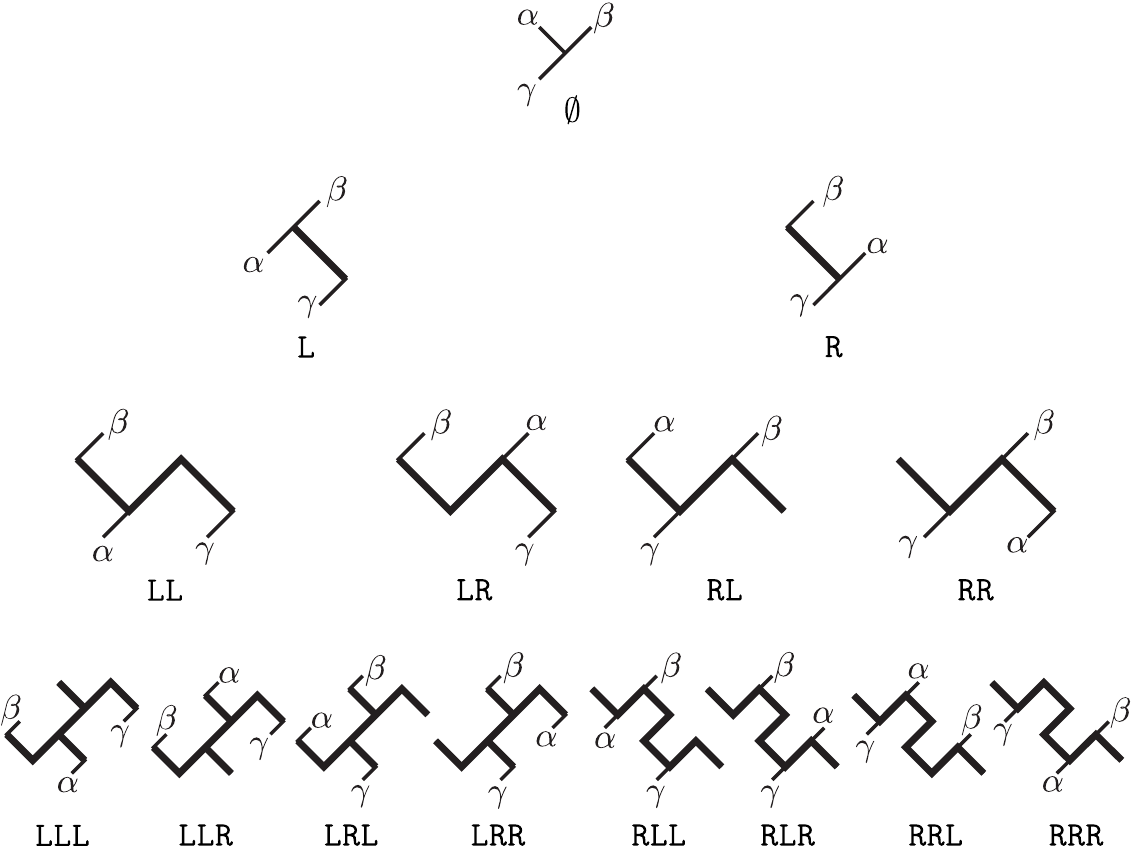}
\caption{Graphs $\Gamma_{1, v}$} \label{fig:LRL}
\end{figure}

Note that it follows directly from the construction, that the
points $z_{\alpha,
  v}'$, $z_{\beta, v}'$ and $z_{\gamma, v}'$ are vertices of a right
isosceles triangle. Orientation of the triangles $z_{\alpha,
  v}'z_{\beta, v}'z_{\gamma, v}'$ depends on the last letter of $v$:
it is counterclockwise if it is $\bel$ and clockwise if it is $\ber$.

\section{Paper-folding curves}
\label{s:paperfolding}
\subsection{Mazes associated with graphs $\Gamma_{1, v}$}

Consider again the Cayley graph $\group{H}$ drawn in $\C$, as on
Figure~\ref{fig:caley}. Consider the half-integral grid on $\C$, i.e.,
the tiling of the plane by the parallel translations
by the elements of $\Z[i]/2$ of the square with the vertices $0, 1/2,
i/2$, and $1/2+i/2$. The group $\group{H}$ acts freely on the set of these
squares with two orbits (corresponding to the two
colors of the checkerboard coloring). We get in this way a checkerboard
coloring of the pillowcases $\C/\group{H}_n$. Let $Q_n$ be the graph
consisting of the sides of the half-integral grid on $\C/\group{H}_n$.

The vertices of the graph $\Gamma_{1, v}$ are centers of squares of one
color in the checkerboard coloring of $\C/\group{H}_n$.
Since $\Gamma_{1, v}$ is a tree, there is a closed
Eulerian path $\rho_v$ in $Q_n$ without transversal
self-intersections, which goes around $\Gamma_{1, v}$, i.e., does not
intersect it transversally, see Figure~\ref{fig:euler}, where the
squares containing $\Gamma_{1, v}$ are colored red, and the other
squares are white. The path $\rho_v$ is the boundary of the white
region (after we glue the picture into a pillowcase).

Note that, unlike for the path $\lambda_v$, there are no problems in the
definition of $\rho_v$ concerning the singular points.

\begin{figure}[h]
\centering
\includegraphics[width=7cm]{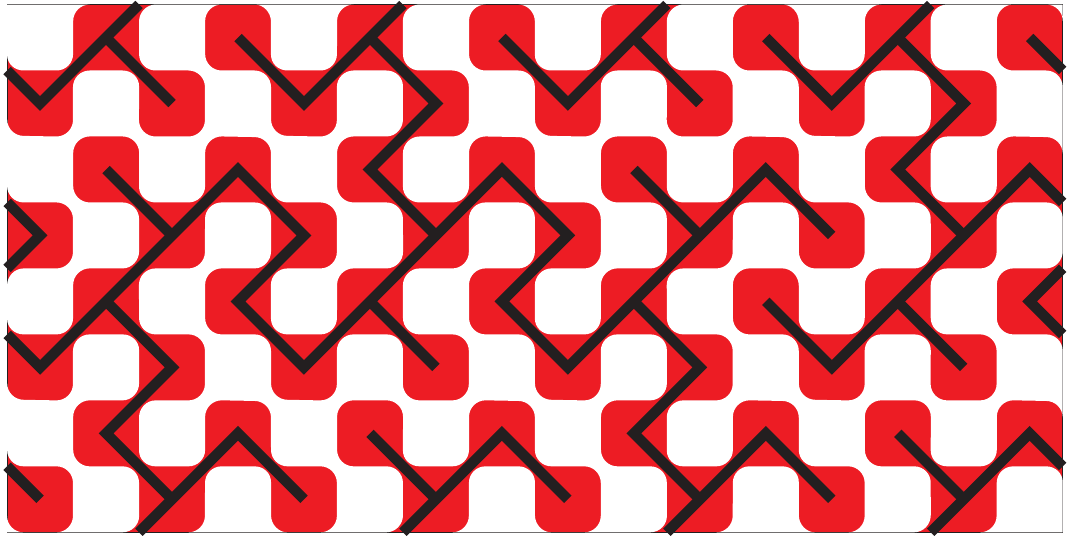}
\caption{Path $\rho_v$}
\label{fig:euler}
\end{figure}

See more examples of the paths $\rho_v$ on Figure~\ref{fig:labyrinth},
where their connected lifts to $\C$ are shown.

\begin{figure}[h]
\centering
\includegraphics[width=7cm]{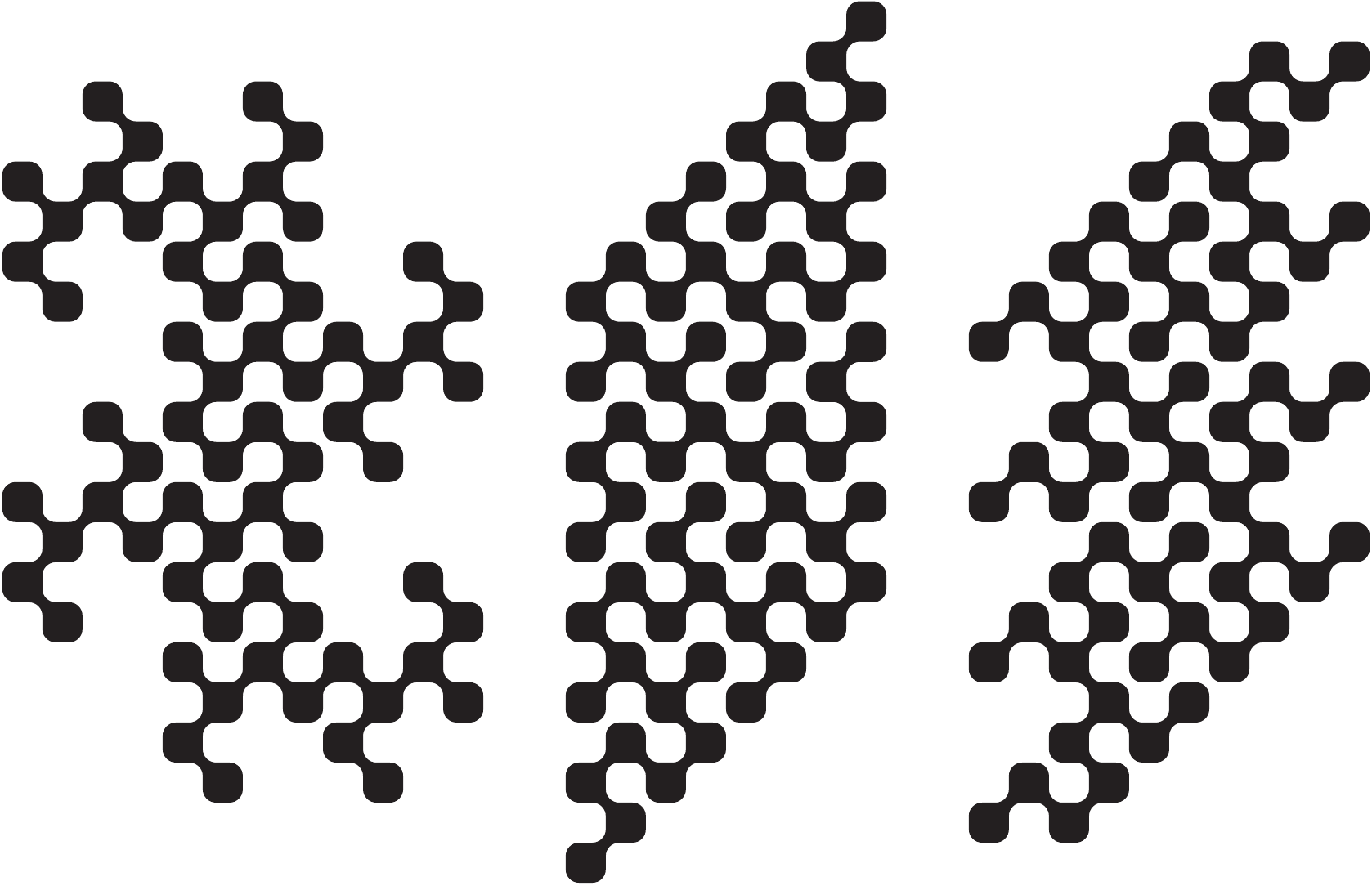}
\caption{Paths $\rho_v$ around $\Gamma_v$}
\label{fig:labyrinth}
\end{figure}

Recall that the curve $\lambda_v$ describes the action of $\tau$ on
the vertices of $\Gamma_{1, v}$. It is easy to check that $\rho_v$ is
obtained from $\lambda_v$ by the replacements of segments of
$\lambda_v$ between vertices of $\Gamma_{1, v}$ by curves shown on
Figure~\ref{fig:taup}. In particular, the curves $\rho_v$ also approximate the
plane-filling curves coming from the matings described in
Proposition~\ref{pr:mating}.

\begin{figure}[h]
\centering
\includegraphics{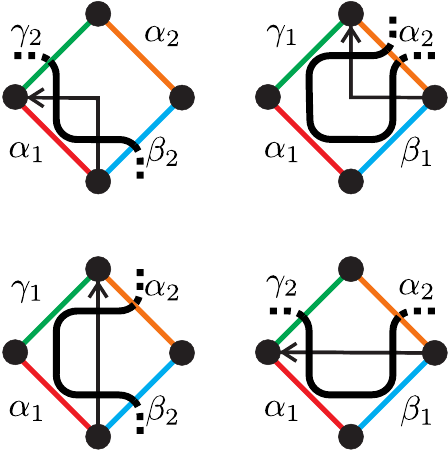}
\caption{Changing $\lambda_v$ to $\rho_v$}
\label{fig:taup}
\end{figure}

\subsection{Another pair of Schreier graphs}

On Figure~\ref{fig:euler}, the curve $\rho_n$ goes around $\Gamma_{1, v}$ bounding the red
cells of the checkerboard tiling of the pillowcase
$\C/\group{H}_n$. It also goes around the white cells, and these cells
are arranged into a tree around which $\rho_v$ travels. Let us try to
interpret this ``white'' tree in terms of the group $\group{G}$.

Let $\zeta_1=1/4+i/4$ and $\zeta_2=1/4-i/4$.
Let $\Sigma_1=\Gamma_{\group{H}}$ be the left Cayley graph
of $\group{H}$ with the set of vertices $\zeta_1\cdot\group{H}$ and the edges
corresponding to the generators $A, B, C, CAB$. Let $\Sigma_2$
be the left Cayley graph of $\group{H}$ with the set of vertices
$\zeta_2\cdot\group{H}$
and the edges corresponding to the generators $A, ABA, C, ABC$.

See Figure~\ref{fig:sigmai} for the graphs $\Sigma_1$, $\Sigma_2$.
Note that each edge of $\Sigma_1$ intersects exactly one edge of
$\Sigma_2$ and vice versa. Namely, the edges of $\Sigma_1$
corresponding to $A, B, C$, and $CAB$ intersect the edges of
$\Sigma_2$ corresponding to $A, ABA, C$, and $ABC$, respectively.

\begin{figure}[h]
\centering
\includegraphics{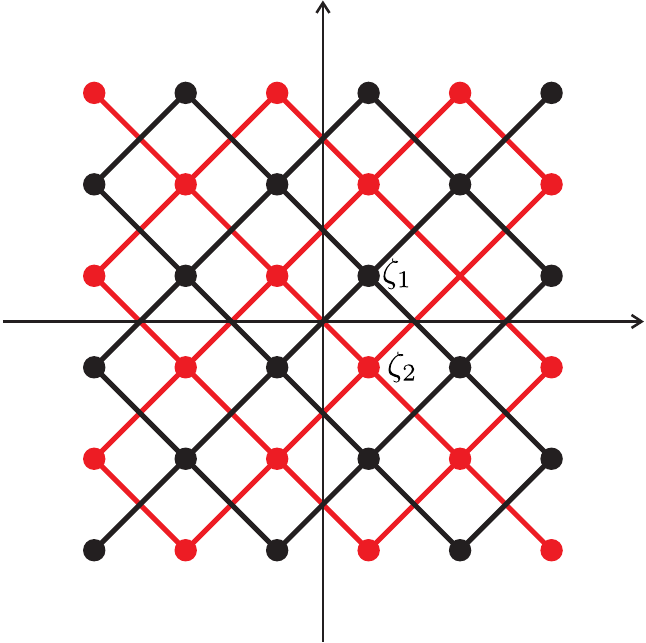}
\caption{}\label{fig:sigmai}
\end{figure}

For a given word $v=X^{(1)}X^{(2)}\ldots X^{(n)}\in\{\bel, \ber\}^*$ of
length $n$, denote by $\Sigma_{1, v}$ and $\Sigma_{2, v}$ the
Schreier graphs of the groups $\group{G}_1$ and $\group{G}_2$ acting on the set
$X^{(1)}\times X^{(2)}\times\cdots\times X^{(n)}$ (i.e., on the
set of the right orbits of the bisets $X^{(1)}_1\otimes
X^{(2)}_1\otimes\cdots\otimes X^{(n)}_1$ and $X^{(1)}_2\otimes
X^{(2)}_2\otimes\cdots\otimes X^{(n)}_2$, respectively) defined
with respect to the generating sets $\{\alpha_1=CAB,
\beta_1, \gamma_1\}$ and $\{\alpha_2=A, \alpha_2\beta_2\alpha_2,
\gamma_2\}$. We identify them with the corresponding sub-graphs of
the graphs $\Sigma_1/\group{H}_n$ and $\Sigma_2/\group{H}_n$, respectively.

\begin{proposition}
\label{pr:corridors}
The graphs $\Sigma_{1, v}$ and $\Sigma_{2, v}$ do
not intersect (as subsets of $\C/\group{H}_n$). In each pair of intersecting
edges of $\Sigma_1/\group{H}_n$ and $\Sigma_2/\group{H}_n$ one
edge belongs to one of the graphs
$\Sigma_{1, v}$ and $\Sigma_{2, v}$ and the other edge does
not belong to either graphs.
\end{proposition}

It follows that $\rho_v$ separates the trees $\Sigma_{1, v}=\Gamma_{1,
  v}$ and $\Sigma_{2, v}$, and that the graphs $\Sigma_{1, v}$ and
$\Sigma_{2, v}$ describe adjacency of the cells
on the corresponding side of the curve $\rho_v$.

\begin{proof}
Connect the basepoints $\zeta_1$ and $\zeta_2$ by a straight
segment $\ell$. The points $\zeta_1$ and $\zeta_2$ as vertices of
the Cayley graphs $\Sigma_1$ and $\Sigma_2$ correspond to the
identity in the group $\group{H}$. It follows that the images of $\ell$
under the action of $\group{H}$ connect the vertices of $\Sigma_1$ and
$\Sigma_2$ corresponding to the same elements of $\group{H}$. It follows
now from the construction of the graphs $\Sigma_i$ (see
Figure~\ref{fig:crossings}) that the edge connecting $h\in\group{H}$ to
$Ch$ in $\Sigma_1$ intersects with the edge connecting the
corresponding vertices in $\Sigma_2$. The same statement for the
edges connecting $h$ to $Ah$ is true. The edge in $\Sigma_1$
connecting $Ah$ to $BAh$ intersects the edge in $\Sigma_2$
connecting $h$ to $ABAh$.

\begin{figure}[h]
\centering
\includegraphics{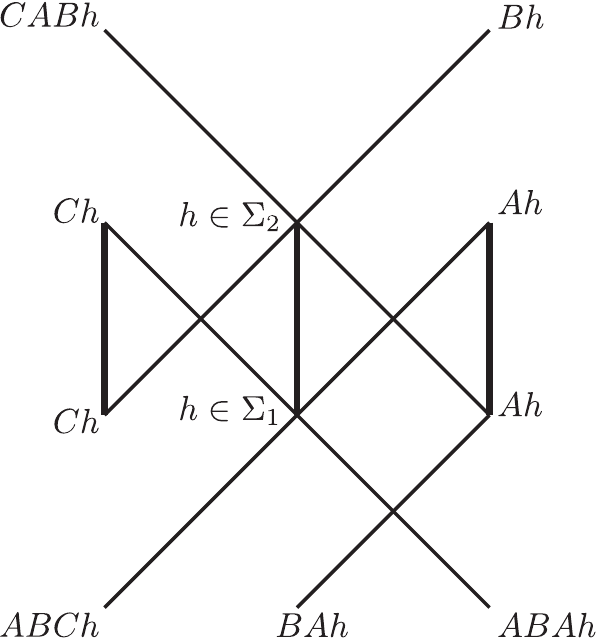}
\label{fig:crossings} \caption{Proof of
Proposition~\ref{pr:corridors}}
\end{figure}

If the graph $\Sigma_{1, v}$ contains an edge from $w$ to
$\gamma_1(w)=C(w)$, then the graph $\Sigma_{2, v}$ does not
contain an edge connecting $w$ to $\gamma_2(w)$, and vice versa.
The edge $w$ to $\gamma_1(w)$ coincides with the edge from some
$h$ to $Ch$ in $\Sigma_1$, while the edge from $w$ to
$\gamma_2(w)$ coincides with the edge from the same element $h$ to
$Ch$ in $\Sigma_2$. This settles the statement for the edges $(w,
\gamma_i(w))$.

The edge from $w$ to $\alpha_2(w)$ corresponds to the edge from
$h$ to $Ah$, which is not included into $\Sigma_1$. Similarly, the
edge from $w$ to $\alpha_1(w)=CAB(w)$ is not included into
$\Sigma_2$.

If $\Sigma_2(w)$ contains an edge from $w$ to
$\alpha_2\beta_2\alpha_2(w)=ABA(w)$, then
$w\ne\alpha_2\beta_2\alpha_2(w)$, i.e., $\alpha_2(w)\ne
\beta_2\alpha_2(w)$, which is equivalent to the condition that the
graph $\Sigma_1(w)$ does not have an edge from $A(w)$ to $BA(w)$.
\end{proof}

\begin{proposition}
The graphs $\Sigma_{1, v}$ and $\Sigma_{2, v}$ are isomorphic for
every word $v$.
\end{proposition}

\begin{proof}
We know, see ~\ref{pr:Ghatbiset}, that for every finite sequences $v=X_1X_2\ldots X_n\in\{\bel,
\ber\}^*$ the $\group{G}_1\cong\group{G}$-biset $X_1\otimes
X_2\otimes\cdots\otimes X_n\cdot\group{G}_1$ is isomorphic to the
$\group{G}_2\cong\group{G}$-biset $a\cdot X_1\otimes
X_2\otimes\cdots\otimes X_n\cdot\group{G}_2\cdot a$. We have
$\alpha=a\alpha a, \alpha\beta\alpha=a\beta a$ and $\gamma=a\gamma
a$, which implies that the map
\[w\mapsto a(w)\]
is an isomorphism of the Schreier graphs $\Sigma_{2, v}\arr\Sigma_{1, v}$.
\end{proof}

\subsection{Paper-folding}

Consider a lift $\Sigma_{1, v}'$ of the graph $\Sigma_{1, v}$ to $\C$,
and
let $\rho_v'$ be the corresponding lift of the path $\rho_v$
(i.e., a closed path going around the lift of the tree $\Sigma_{1,
v}$).

Since the singular points $Z_A, Z_B, Z_C$, and $Z_{CAB}$ of the
orbifold $\C/\group{H}_n$ belong to the graph $Q_n$ for every $n$, the
path $\rho_v$ also passes through these points. The lift $\rho_v'$
will pass through preimages $Z_A', Z_B', Z_C'$ and $Z_{CAB}'$
of the singular points.

\begin{proposition}
\label{pr:foldrule}
The path $\rho_v'$ consists of four pieces: $\rho_{v}(A, B)$ from
$Z_A'$ to $Z_B'$, $\rho_v(B, CAB)$ from $Z_B'$ to $Z_{CAB}'$,
$\rho_v(CAB, C)$ from $Z_{CAB}'$ to $Z_C'$, and $\rho_v(C, A)$
from $Z_C'$ to $Z_A'$.

If the last letter of $v$ is $\bel$, or if $v$ is empty (resp., if
the last letter of $v$ is $\ber$), then the path, going consecutively
through $\rho_v(A, B), \rho_v(B, CAB)$, $\rho_v(CAB, C)$, and
$\rho_v(C, A)$ goes in the positive (resp., negative) direction
around $\Sigma_{1, v}'$. Each path $\rho_v(X_1, X_2)$ is equal to
the image of the previous path $\rho_v(X_0, X_1)$ under rotation
by $-\pi/2$ (resp.\ $\pi/2$) around $Z_{X_1}'$.

The path $\rho_{v\bel}(C, A)$ (resp.\ $\rho_{v\ber}(C, A)$) can be taken
equal to the union of the path $\rho_v(C, A)\cup\rho_v(A, B)$ with
its image under rotation by $-\pi/2$ (resp.\ $\pi/2$) around
$Z_B'$.
\end{proposition}

\begin{proof}
Straightforward induction, using the inductive rule of
constructing the graphs $\Gamma_{1, v}$. See Figure~\ref{fig:foldrule}
for the inductive step. The points $z_{\alpha, v}'$,
$z_{\beta, v}'$ and $z_{\gamma, v}'$ will coincide with the points
$Z_{CAB}'$, $Z_B'$ and $Z_C'$ (which are denoted $CAB, B$ and $C$ on
the Figure~\ref{fig:foldrule}).
\begin{figure}[h]
\centering
\includegraphics{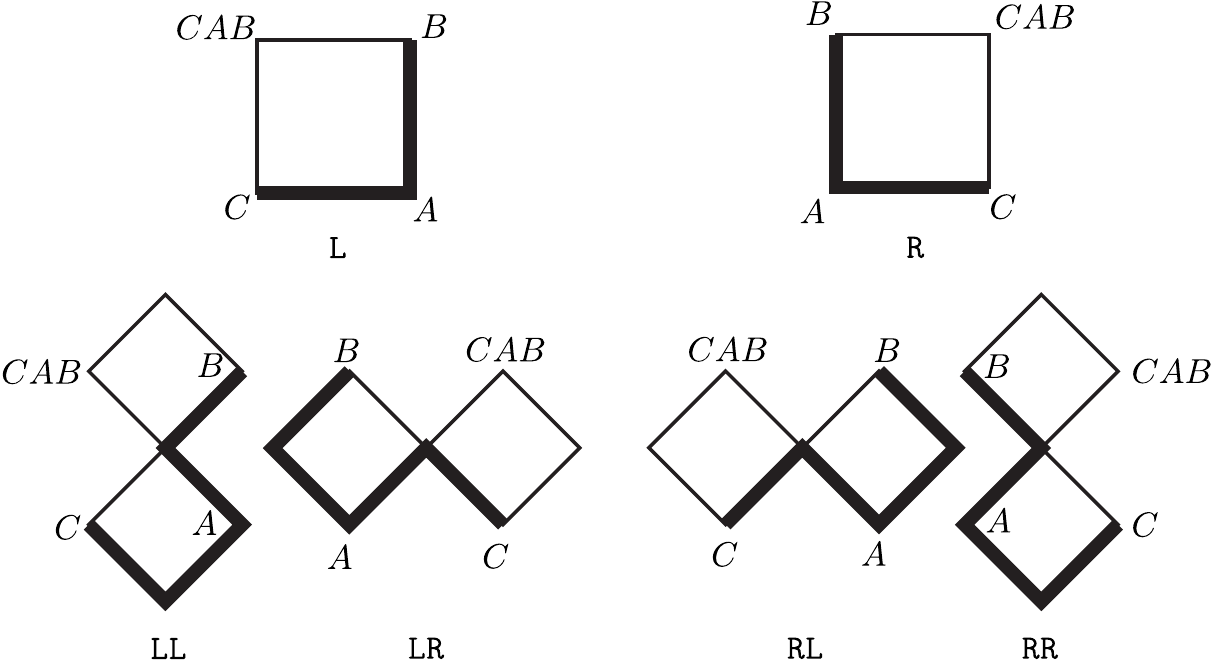}
\caption{Inductive proof of
Proposition~\ref{pr:foldrule}}\label{fig:foldrule}
\end{figure}
\end{proof}

Consider a strip of paper of length $2^{n-1}$. Let
us denote one end of the strip by $C$. For a given word
$v=X_1X_2\ldots X_n$ of letters $X_i\in\{\bel, \ber\}$ fold the strip in
two, fixing $C$ and moving the other end of the strip to $C$ on
the left side, if $X_n=\bel$, or on the right side, if $X_n=\ber$ (see
Figure~\ref{fig:paper}). Repeat now the procedure for the word $X_1X_2\ldots
X_{n-1}$. After $n$ steps unfold the strip so that all bends are at right
angles. We get in this way a broken line $P_v$. Take a copy of $P_v$,
rotate it by $180^\circ$ and connect its endpoints with $P_v$. We get
a closed broken line $\overline P_v$.

\begin{figure}[h]
\centering
\includegraphics{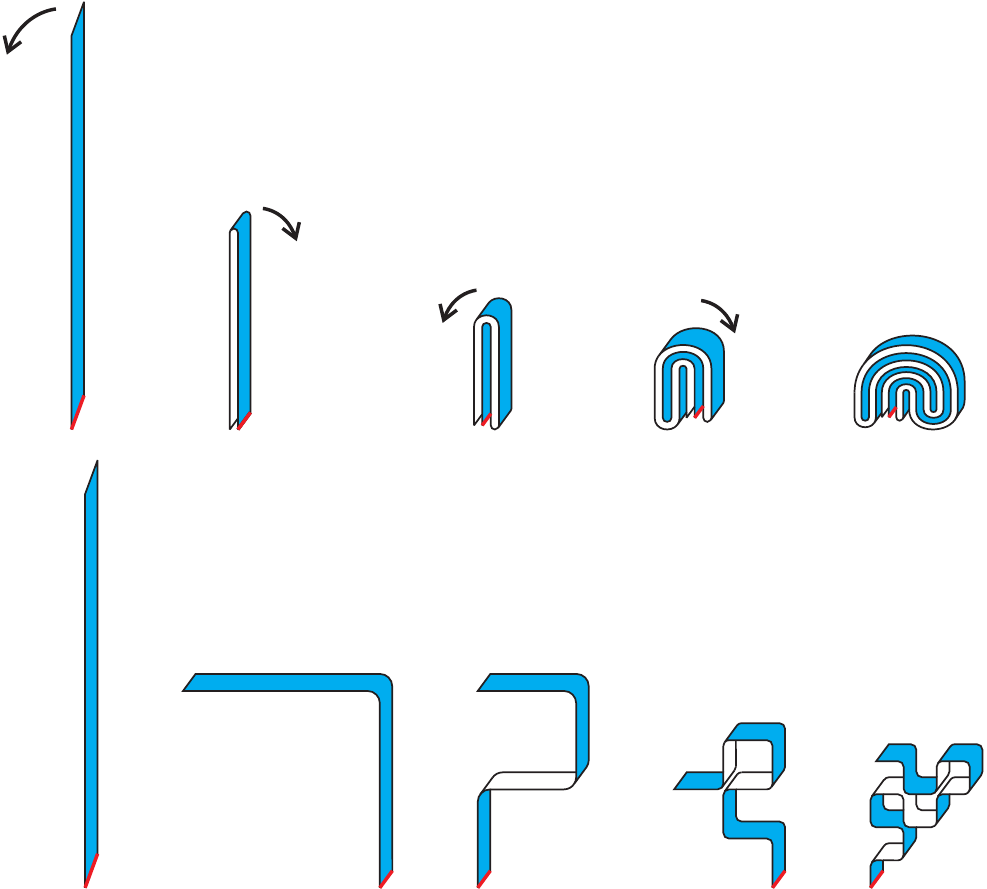}
\caption{Folding paper}
\label{fig:paper}
\end{figure}

The following statement is a direct corollary of Proposition~\ref{pr:foldrule}.

\begin{corollary}
The broken line $\overline P_v$ is isometric to the path $\rho_v'$
going around the graph $\Sigma_{1, v}'$.
\end{corollary}

\section{Boundaries of Fatou components of
  $f$ and rotated tunings}
\label{s:tuning}

Consider the group $\group{B}=\langle S, \gamma\alpha, \beta\rangle=\langle
S\rangle\times
\langle\gamma\alpha, \beta\rangle$. Let us
write the images of the generators under the wreath recursion.
\begin{eqnarray*}S &=& \sigma\pi(P\beta\alpha\gamma, P, S^{-1}\beta\alpha\gamma,
S^{-1}),\\
\gamma\alpha &=& \sigma(1, \gamma\beta, \alpha, \gamma\alpha\beta),\\
\beta &=& (1, \beta\alpha\beta, \alpha, 1).
\end{eqnarray*}

Conjugate the right hand side of the recursion by
$(34)(\beta, \beta, P\beta\alpha\gamma\beta, P\beta)$:
\begin{eqnarray*}S &=& (13)(24)(1, 1, T,
T),\\
\gamma\alpha &=& \sigma(1, \beta\gamma, 1, \beta\gamma),\\
\beta &=& (1, \alpha, 1, \alpha).
\end{eqnarray*}

We see that in the right hand side of the recursion we get elements of
the group $\group{A}=\langle T, \beta\gamma, \alpha\rangle=\langle
T\rangle\times
\langle\beta\gamma, \alpha\rangle$. Moreover, the corresponding
$(\group{B}-\group{A})$-biset $\bim_{\group{B},\group{A}}$ is the direct product of the
$(\langle S\rangle-\langle T\rangle)$-biset given by the binary
recursion
\[S=\sigma(1, T),\]
with the
$(\langle\gamma\alpha, \beta\rangle-\langle\beta\gamma,
\alpha\rangle)$-biset given by the recursion
\begin{eqnarray}\label{eq:BA1}\gamma\alpha &=& \sigma(1, \beta\gamma),\\
\label{eq:BA2}\beta &=& (1, \alpha).
\end{eqnarray}

The following proposition is proved by direct computation.

\begin{proposition}
\label{pr:BApol}
The biset over free groups of rank two given by the wreath
recursion~\eqref{eq:BA1}--\eqref{eq:BA2} is isomorphic to the biset associated with
the partial covering
\[\C\setminus\{0, 1\}\supset\C\setminus\{1, 0, 2\}\arr\C\setminus\{0, 1\}\]
defined by the polynomial $(1-z)^2$, where $\gamma\alpha$ and $\beta$ correspond
to the loops around punctures $0$ and $1$, respectively, and the
generators $\beta\gamma$ and $\alpha$ correspond to the loops around
$1$ and $0$, respectively.
\end{proposition}

The generators of
the group $\group{A}$ are decomposed (in the original recursion) as follows:
\begin{eqnarray*}T &=& (P, \beta
P\beta, \gamma S\gamma, S),\\
\beta\gamma &=& (\gamma, \beta\alpha, \alpha\gamma, \beta),\\
\alpha &=& \sigma(\beta, \beta, \beta\alpha, \alpha\beta).
\end{eqnarray*}

Conjugate the wreath recursion by $(1, \beta, \alpha, \beta)$:
\begin{eqnarray*}T &=& (P, P, S, S),\\
\beta\gamma &=& (\gamma, \alpha\beta, \gamma\alpha, \beta),\\
\alpha &=& \sigma.\end{eqnarray*}

Restricting to the last two coordinates of the wreath recursion (i.e.,
to the biset $\bimR$) we get a biset $\bim_{\group{A}, \group{B}}$, which is a
direct product of the biset defined by the isomorphism
$T\mapsto S$ and the $(\langle\beta\gamma,
\alpha\rangle-\langle\gamma\alpha, \beta\rangle)$-biset given by
the recursion
\begin{eqnarray}\label{eq:AB1}
\beta\gamma &=& (\gamma, \alpha\beta),\\
\label{eq:AB2}\alpha &=& \sigma.
\end{eqnarray}

The proof of the following proposition is also straightforward.

\begin{proposition}
\label{pr:ABpol}
  The biset over the free groups defined by the wreath
  recursion~\eqref{eq:AB1}--\eqref{eq:AB2} is isomorphic to the biset associated
  with the partial covering
\[\C\setminus\{0, 1\}\supset\C\setminus\{0, 1, 1/2\}\arr\C\setminus\{0, 1\},\]
given by the polynomial $(2z-1)^2$, where the generators $\alpha$ and
$\beta\gamma$ correspond to the loops around $0$ and $1$, while the
generators $\gamma$ and $\alpha\beta$ correspond to the loops around
$1$ and $0$, respectively.
\end{proposition}

Taking tensor products $\bim_{\group{B}, \group{A}}\otimes_{\group{A}}\bim_{\group{A}, \group{B}}$ and
$\bim_{\group{A}, \group{B}}\otimes_{\group{B}}\bim_{\group{B},
 \group{A}}$ we see that $\group{A}$ and $\group{B}$ are self-similar
subgroups of $\img{F^{\circ 2}}$.

\begin{corollary}
\label{cor:tensorproducts}
  The bisets $\bim_{\group{B},
    \group{A}}\otimes_{\group{A}}\bim_{\group{A}, \group{B}}$
and $\bim_{\group{A}, \group{B}}\otimes_{\group{B}}\bim_{\group{B}, \group{A}}$ are isomorphic to the
bisets associated with
the post-critically finite polynomials $(2(1-z)^2-1)^2=(2z^2-4z+1)^2$ and
$(1-(2z-1)^2)^2=16z^2(1-z)^2$, respectively.
\end{corollary}

\begin{proof}
It follows from Propositions~\ref{pr:BApol},~\ref{pr:ABpol}.
\end{proof}

Note that the polynomials $16z^2(1-z)^2$ and $(2z^2-4z+1)^2$ coincide
with the restrictions of the second iteration of the endomorphism $F$
of $\CP$ to the post-critical lines $p=1$ and $p=0$,
respectively. The action of $F$ is written in homogeneous coordinates
as
\[[z:p:u]\mapsto [(2z-p-u)^2:(p-u)^2:(p+u)^2],\]
hence its restriction to the line $p=u$ is
\[[z:p:p]\mapsto [(2z-2p)^2:0:4p^2],\]
so that it acts on the first coordinate (in non-homogeneous
coordinates) as
\[z\mapsto (z-1)^2.\]
Restriction of $F$ onto the line $p=0$ is
\[[z:0:u]\mapsto [(2z-u)^2:u^2:u^2],\]
i.e.,
\[z\mapsto (2z-1)^2.\]

\begin{proposition}
The limit spaces of $(\group{A}, \bim_{\group{A}, \group{B}}\otimes\bim_{\group{B}, \group{A}})$
and $(\group{B}, \bim_{\group{B}, \group{A}}\otimes\bim_{\group{A},
\group{B}})$ are direct products of circles with
the Julia sets of the polynomials $16z^2(1-z)^2$ and $(2z^2-4z+1)^2$, respectively. The
images of $\lims[\group{A}]$ and $\lims[\group{B}]$ in $\lims[\img{F}]$ are identified
with the subsets of the Julia set projected by $(z, p)\mapsto p$ to the boundaries of the Fatou
components of $f(p)=\left(\frac{p-1}{p+1}\right)^2$, containing 1 and 0, respectively.
\end{proposition}

\begin{proof}
Description of the limit space follows directly from the structure of
the wreath recursion and Corollary~\ref{cor:tensorproducts}. The
groups $\langle\alpha, \beta\gamma\rangle$ and $\langle\beta,
\gamma\alpha\rangle$ are level-transitive, which implies that
the map from the limit spaces of $\group{A}$ and $\group{B}$ to the
limit space of $\img{F}$ is surjective on the fibers of the natural projection
$\lims[\img{F}]\arr\lims[\img{f}]$.

It follows from the post-critical dynamics of $f$ and the
interpretation of the maps $S$ and $P$ as loops in the space
$\C\setminus\{0, 1\}$ (see Proposition~\ref{pr:quotSP})
that the images of the limit spaces
$\lims[\group{A}]$ and $\lims[\group{B}]$ in the Julia set of $f$
are the boundaries of the Fatou components of $1$ and $0$, respectively.
\end{proof}

See Figure~\ref{fig:pjuliat}, where the Julia sets of the
polynomials $16z^2(1-z)^2$ and $(2z^2-4z+1)^2$ are shown.

\begin{figure}[h]
\centering
\includegraphics{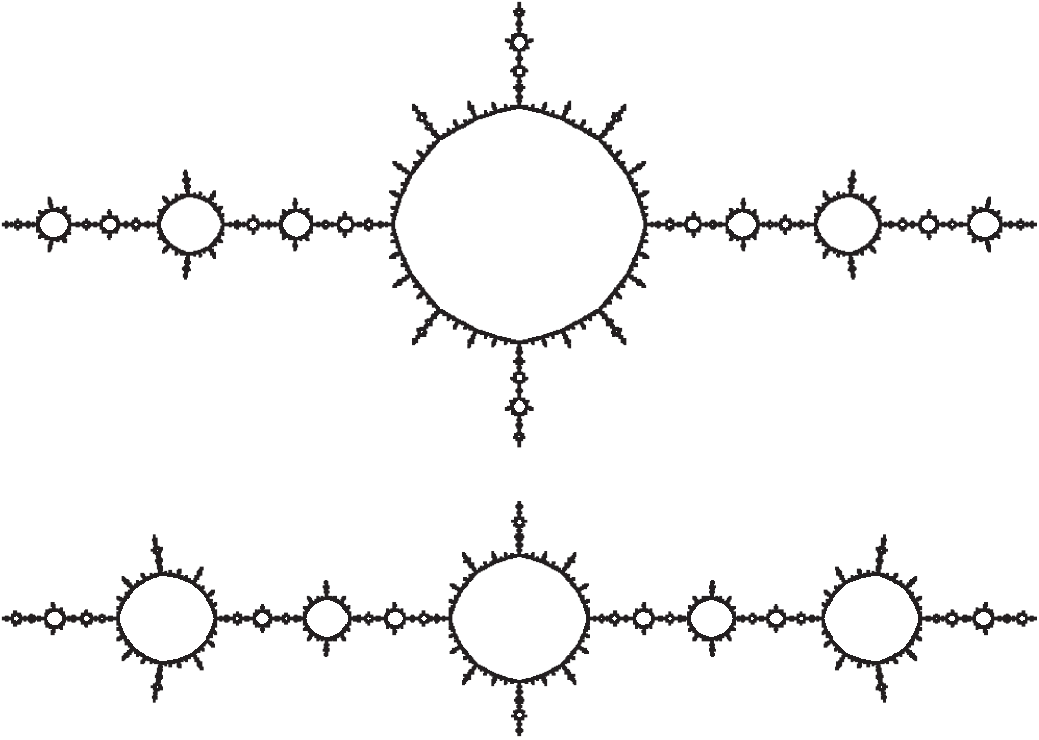}
\label{fig:pjuliat}\caption{The Julia sets of $16z^2(1-z)^2$ and
$(2z^2-4z+1)^2$}
\end{figure}

The natural map from $\lims[\group{A}]$ and $\lims[\group{B}]$ to the Julia set of
$F$ are not injective. Let us see which points of the limit space
are identified under these maps.

Consider the group $\group{B}_1=\langle\gamma, \alpha, S\rangle$. Its
generators are written as
\begin{eqnarray*}
\alpha &=& \sigma(\beta, \beta, \beta\alpha, \alpha\beta),\\
\gamma &=& (\gamma, \beta, \gamma, \beta),\\
S &=& \pi\sigma(P\beta\alpha\gamma, P, S^{-1}\beta\alpha\gamma,
S^{-1}),\end{eqnarray*}

Conjugating the right-hand side by $(1, \beta, P, P\beta\alpha)$
we get
\begin{eqnarray*}
\alpha &=& \sigma,\\
\gamma &=& (\gamma, \beta, \gamma, \beta),\\
S\gamma\alpha &=& \pi(1, 1, T\beta\gamma, T\beta\gamma).\end{eqnarray*}

Consider now the group $\group{A}_1=\langle\gamma, \beta, T\rangle$. We
have
\begin{eqnarray*}\beta &=& (1, \beta\alpha\beta, \alpha, 1),\\
\gamma &=& (\gamma, \beta, \gamma, \beta),\\
T &=& (P, \beta P\beta, \alpha S\alpha, S),\\
T\beta\gamma &=& (P\gamma, \beta P\alpha, S \gamma\alpha, S\beta).\end{eqnarray*}

Restriction to the third coordinate of the wreath recursion is the
homomorphism
\[\beta\mapsto \alpha\]
\[\gamma\mapsto\gamma\]
\[T\beta\gamma\mapsto S\gamma\alpha.\]

Note that the subgroups $\langle\gamma\alpha, S\rangle$ and
$\langle\beta\gamma, T\rangle$ are also subgroups of $\group{B}$ and $\group{A}$,
respectively. The corresponding wreath recursions for these groups
are
\[\gamma\alpha=\sigma(1, \beta\gamma),\]
\[S=\sigma(1, T),\]
\[S\gamma\alpha=(T, \beta\gamma),\]
(here we use a conjugated version~\eqref{eq:BA1}--\eqref{eq:BA2}),
and
\[\beta\gamma\mapsto\gamma\alpha,\]
\[T\mapsto S,\]
\[T\beta\gamma\mapsto S\gamma\alpha.\]

It follows from the recursions and post-critical dynamics of the
polynomials $16z^2(1-z)^2$ and $(2z^2-4z+1)^2$ that the limit
space of $\langle\gamma\alpha, S\rangle$ with respect to these
recursions is the natural direct product of the boundary of the
Fatou component of $f$ containing $0$ (which is the limit space of
the subgroup $\langle S\rangle$) and the boundary of the Fatou
component of $(2z^2-4z+1)^2$ containing $0$ (the limit space of
$\langle\gamma\alpha\rangle$). Similarly, the limit space of
$\langle\beta\gamma, T\rangle$ is the direct product of the the
boundary of the Fatou component of $f$ containing $1$ (the limit
space of $\langle T\rangle$) with the boundary of the Fatou
component of $16z^2(1-z)^2$ containing $1$ (the limit space of
$\langle\beta\gamma\rangle$). Hence, both limit spaces are tori.

In particular, the natural maps from the limit spaces of the
groups $\langle\gamma\alpha, S\rangle$ and $\langle\beta\gamma,
T\rangle$ to $\lims[\group{B}]$ and $\lims[\group{A}]$,
respectively, are injective.

Let us compare the limit spaces of these groups with the limit
spaces of their extensions $\group{B}_1$ and $\group{A}_1$. We have
$\group{B}_1=\langle\alpha, \gamma\rangle\times\langle
S\gamma\alpha\rangle$, $\group{A}_1=\langle\beta,
\gamma\rangle\times\langle T\beta\gamma\rangle$. The direct
factors $\langle\alpha, \gamma\rangle$ and $\langle\beta,
\gamma\rangle$ are infinite dihedral with the wreath recursions
\[\alpha=\sigma,\quad\gamma=(\gamma, \beta),\]
and
\[\beta\mapsto\alpha,\quad\gamma\mapsto\gamma.\]
It follows that the limit spaces of the subgroups $\langle\alpha,
\gamma\rangle$ and $\langle\beta, \gamma\rangle$ are segments with
singular endpoints (one with the isotropy group
$\langle\alpha\rangle$ or $\langle\beta\rangle$ and the other with
the isotropy group $\langle\gamma\rangle$).

Consequently, the limit spaces of $\group{B}_1$ and $\group{A}_1$ are annuli
(direct products of a circle and a segment) and that the natural
map from the limit spaces of $\langle\gamma\alpha, S\rangle$ and
$\langle\beta\gamma, T\rangle$ to the limit spaces of $\group{B}_1$ and
$\group{A}_1$ ``flattens'' the tori into annuli, by flattening the circles
$\lims[\langle\gamma\alpha\rangle]$ and
$\lims[\langle\beta\gamma\rangle]$ to the segments
$\lims[\langle\gamma, \alpha\rangle]$ and $\lims[\langle\beta,
\gamma\rangle]$. Note that the natural direct product decomposition of the spaces
$\lims[\langle\gamma\alpha, S\rangle]$ and
$\lims[\langle\beta\gamma, T\rangle]$ comes from the direct
product decompositions $\langle\gamma\alpha\rangle\times\langle
S\rangle$ and $\langle\beta\gamma\rangle\times\langle T\rangle$,
while we have decompositions of self-similar groups
$\group{B}_1=\langle\gamma, \alpha\rangle\times\langle
S\gamma\alpha\rangle$ and $\group{A}_1=\langle\beta,
\gamma\rangle\times\langle T\beta\gamma\rangle$. Consequently, the
preimages of the boundaries of the annuli in the tori are
diagonals with respect to the natural decomposition of the torus
into the product of the boundaries of the Fatou components. It
means that the diameter with respect to which we flatten the
boundary of the Fatou component of the polynomials $16z^2(1-z)^2$
and $(2z^2-4z+1)^2$ rotates as $p$ travels along the boundary of
the Fatou component of $f$.

This flattening of the circles into segments can be interpreted as
a ``rotated tuning'' of the polynomials $16z^2(1-z)^2$ and
$(2z^2-4z+1)^2$ by the polynomial $z^2-2$ (which is the quadratic
polynomial with the dihedral iterated monodromy group and the
Julia set a segment). It can be nicely illustrated by the slices
of the Julia set of $F$ as $p$ travels close to the boundary of
the Fatou component of $f$, but stays inside in it. Then the
slices are still homeomorphic to the Julia sets of $16z^2(1-z)^2$
and $(2z^2-4z+1)^2$, but are close to the dendrite slices of the
Julia set of $F$ along to boundary of the Fatou component of $f$.
 
See Figure~\ref{fig:tuning} where slices of the Julia set of $F$ are
shown when $p$ is traveling close to the boundary of the Fatou
component of $f$ containing 1 (top half of the figure). Corresponding
slices for $p$ on the boundary of the Fatou component are shown in the
bottom part of the figure.

\begin{figure}[h]
\centering
\includegraphics{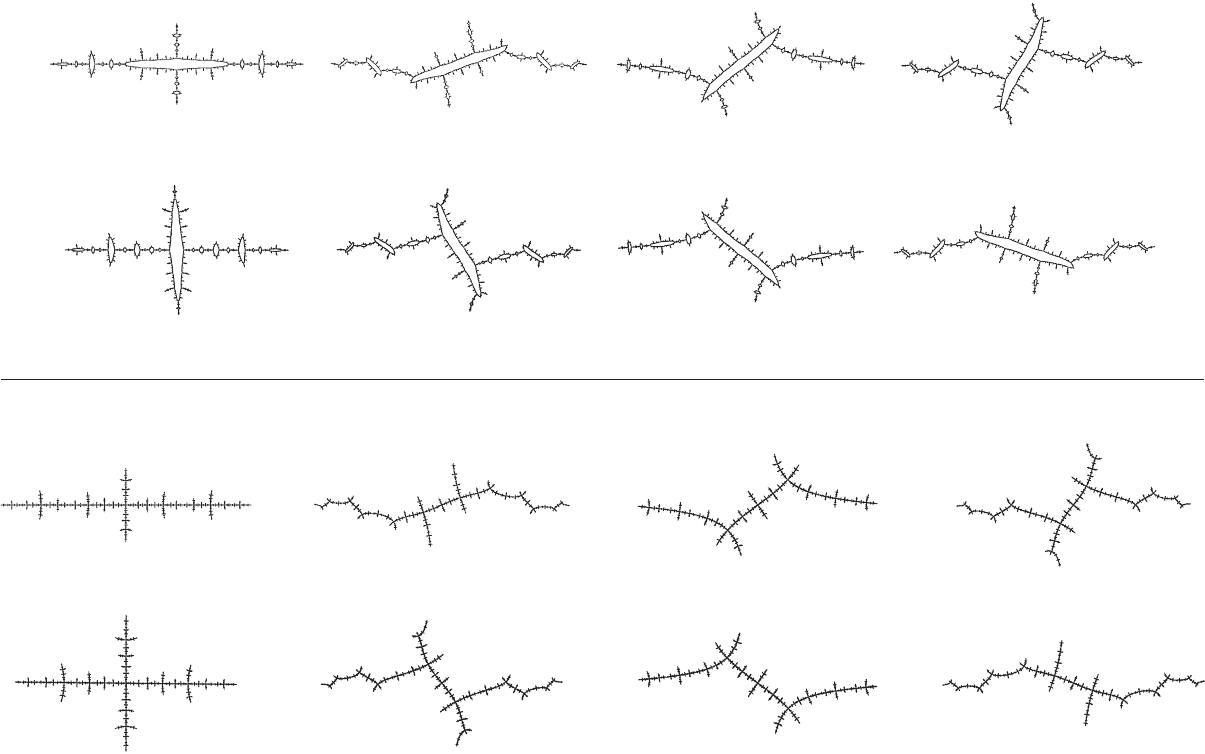}
\caption{Rotated tuning}
\label{fig:tuning}
\end{figure}

\def\cprime{$'$}\def\ocirc#1{\ifmmode\setbox0=\hbox{$#1$}\dimen0=\ht0
  \advance\dimen0 by1pt\rlap{\hbox to\wd0{\hss\raise\dimen0
  \hbox{\hskip.2em$\scriptscriptstyle\circ$}\hss}}#1\else {\accent"17 #1}\fi}
  \def\cprime{$'$}


\begin{thebibliography}{10}

\bibitem{barth:GAP}
Laurent Bartholdi.
\newblock {FR}, {GAP} package {F}unctionally recursive groups.
\newblock \verb!http://laurentbartholdi.github.io/fr/chap0.html!, 2014.

\bibitem{bartnek:rabbit}
Laurent Bartholdi and Volodymyr~V. Nekrashevych.
\newblock Thurston equivalence of topological polynomials.
\newblock {\em Acta Math.}, 197(1):1--51, 2006.

\bibitem{bartnek:mand}
Laurent Bartholdi and Volodymyr~V. Nekrashevych.
\newblock Iterated monodromy groups of quadratic polynomials {I}.
\newblock {\em Groups, Geometry, and Dynamics}, 2(3):309--336, 2008.

\bibitem{buffepsteinkochpilgrim:pullback}
Xavier Buff, Adam Epstein, Sarah Koch, and Kevin Pilgrim.
\newblock On {T}hurston's pullback map.
\newblock In {\em Complex dynamics}, pages 561--583. A K Peters, Wellesley, MA,
  2009.

\bibitem{davisknuth:dragons}
Chandler Davis and Donald~E. Knuth.
\newblock Number representations and dragon curves.
\newblock {\em J. Recreational Math.}, 3:66--81, 1965.

\bibitem{fornsibon:higher}
J.~E. Forn{\ae}ss and N.~Sibony.
\newblock Complex dynamics in higher dimension.
\newblock In {\em Several complex variables (Berkeley, CA, 1995--1996)},
  volume~37 of {\em Math. Sci. Res. Inst. Publ.}, pages 273--296. Cambridge
  Univ. Press, Cambridge, 1999.

\bibitem{knuth}
Donald~E. Knuth.
\newblock {\em The art of computer programming, {V}ol. 2, {S}eminumerical
  Algorithms}.
\newblock Addison-Wesley Publishing Company, 1969.

\bibitem{kobayashi:spaces}
Shoshichi Kobayashi.
\newblock {\em Hyperbolic complex spaces}, volume 318 of {\em Grundlehren der
  Mathematischen Wissenschaften [Fundamental Principles of Mathematical
  Sciences]}.
\newblock Springer-Verlag, Berlin, 1998.

\bibitem{koch:french}
S.~Koch.
\newblock {\em {Teichm\"uller} theory and endomorphisms of {$\mathbb{P}^n$}}.
\newblock PhD thesis, Universit\'e de Provence, 2007.

\bibitem{mandelbrot}
Benoit~B. Mandelbrot.
\newblock {\em {The fractal geometry of nature. Rev. ed. of "Fractals", 1977}}.
\newblock {San Francisco: W. H. Freeman and Company.}, 1982.

\bibitem{milnor:dragons}
John Milnor.
\newblock Pasting together {J}ulia sets: a worked out example of mating.
\newblock {\em Experiment. Math.}, 13(1):55--92, 2004.

\bibitem{muntyansavchuk:gap}
Y.~Muntyan and D.~Savchuk.
\newblock {AutomGrp} {GAP} package for computation in groups and semigroups
  generated by automata.
\newblock \verb!http://www.gap-system.org/Packages/automgrp.html!, 2014.

\bibitem{nek:book}
Volodymyr Nekrashevych.
\newblock {\em Self-similar groups}, volume 117 of {\em Mathematical Surveys
  and Monographs}.
\newblock Amer. Math. Soc., Providence, RI, 2005.

\bibitem{nek:ssfamilies}
Volodymyr Nekrashevych.
\newblock A minimal {Cantor} set in the space of 3-generated groups.
\newblock {\em Geometriae Dedicata}, 124(2):153--190, 2007.

\bibitem{nek:filling}
Volodymyr Nekrashevych.
\newblock Symbolic dynamics and self-similar groups.
\newblock In Mikhail Lyubich and Michael Yampolsky, editors, {\em Holomorphic
  dynamics and renormalization. {A volume in honour of John Milnor's 75th
  birthday}}, volume~53 of {\em Fields Institute Communications}, pages 25--73.
  A.M.S., 2008.

\bibitem{nek:nonuniform}
Volodymyr Nekrashevych.
\newblock A group of non-uniform exponential growth locally isomorphic to
  {$IMG(z^2+i)$}.
\newblock {\em Transactions of the AMS.}, 362:389--398, 2010.

\bibitem{nek:bath}
Volodymyr Nekrashevych.
\newblock Iterated monodromy groups.
\newblock In {\em {Groups St. Andrews in Bath. Vol. 1}}, volume 387 of {\em
  London Math. Soc. Lect. Note Ser.}, pages 41--93. Cambridge Univ. Press,
  Cambridge, 2011.

\bibitem{nek:dendrites}
Volodymyr Nekrashevych.
\newblock The {Julia} set of a post-critically finite endomorphism of
  {$\mathbb{CP}^2$}.
\newblock {\em Journal of Modern Dynamics}, 6(3):327--375, 2012.

\bibitem{nek:models}
Volodymyr Nekrashevych.
\newblock Combinatorial models of expanding dynamical systems.
\newblock {\em Ergodic Theory and Dynamical Systems}, 34:938--985, 2014.

\bibitem{nek:genus}
Volodymyr Nekrashevych.
\newblock An uncountable family of 3-generated groups with isomorphic profinite
  completions.
\newblock {\em International Journal of Algebra and Computation}, 24(1):33--46,
  2014.

\end{thebibliography}
\end{document}